\documentclass[fleqn,11pt,oneside]{article}

\usepackage{hyperref}
\usepackage[final]{graphicx}
\usepackage{bbm}  
\usepackage{bm}  
\usepackage{amsmath,amsthm,amssymb,amscd}
\usepackage{caption,subcaption}
\usepackage{booktabs,multirow}
\usepackage{setspace} 
\usepackage[top=1.1in, bottom=1.1in, left=1.6in, right=1.1in]{geometry}
\usepackage[final]{showlabels}
\usepackage{dashbox}
\usepackage{microtype}
\usepackage[medium]{titlesec}
\usepackage{algpseudocode}
\usepackage[ruled]{algorithm}
\usepackage{xpatch}
\usepackage{accents}
\usepackage{bbding}
\usepackage{undertilde}

\newtheorem{theorem}{Theorem}[section]
\newtheorem{lemma}[theorem]{Lemma}

\newtheorem{definition1}{Definition}[section]
\newtheorem{observe}{Observation}[section]
\newtheorem{remark1}[observe]{Remark}
\newtheorem{example1}{Example}[section]
\newtheorem{aside1}[observe]{Aside}

\newenvironment{definition}[1][]{\begin{definition1}[#1] \rm}{\end{definition1}}

\newenvironment{remark}{\begin{remark1} \rm}{\end{remark1}}

\def\qed{\hfill$\blacksquare$\\} \renewenvironment{proof}{\noindent {\bf 
Proof.}}{\qed}

\algrenewcommand\algorithmicindent{3.0em}

\makeatletter
\xpatchcmd{\algorithmic}{\itemsep\z@}{\itemsep=2.3ex}{}{}
\makeatother

\newif\ifshowboxes \showboxestrue

\renewcommand{\hat}{\widehat}
\renewcommand{\tilde}{\widetilde}

\newcommand{\SU}{\mathrm{SU}}

\renewcommand\Re{\operatorname{Re}}
\renewcommand\Im{\operatorname{Im}}

\newcommand{\fl}{\mathop{\mathit{fl}}}
\newcommand{\op}{\mathop{\mathrm{op}}}
\newcommand{\ur}{\mathit{u}}

\newcommand{\norm}[1]{\ensuremath{ {\lVert #1 \rVert} }}

\newcommand{\abs}[1]{\ensuremath{ {\lvert #1 \rvert} }}
\newcommand{\babs}[1]{\ensuremath{ {\bigl\lvert #1 \bigr\rvert} }}

\def\C{\mathbbm{C}}
\def\R{\mathbbm{R}}
\def\1{\mathbbm{1}}

\newcommand{\e}[1]{\ensuremath{\cdot 10^{#1}}}
\newcommand\T{\rule{0pt}{2.6ex}}       

\begin{document}

\begin{center}
   \begin{minipage}[t]{6.0in}

The roots of a monic polynomial expressed in a Chebyshev basis are known to
be the eigenvalues of the so-called colleague matrix, which is a Hessenberg
matrix that is the sum of a symmetric tridiagonal matrix and a rank-1
matrix.  The rootfinding problem is thus reformulated as an eigenproblem,
making the computation of the eigenvalues of such matrices a subject of
significant practical importance.  In this manuscript, we describe an $O(n^2)$
explicit structured QR algorithm for colleague matrices and prove that it is
componentwise backward stable, in the sense that the backward error in the
colleague matrix can be represented as relative perturbations to its
components.  A recent result of Noferini, Robol, and Vandebril shows that
componentwise backward stability implies that the backward error~$\delta c$
in the vector~$c$ of Chebyshev expansion coefficients of the polynomial has
the bound $\norm{\delta c}\lesssim \norm{c}\ur$, where $\ur$ is machine
precision. Thus, the algorithm we describe has both the optimal backward
error in the coefficients and the optimal cost $O(n^2)$.  We
illustrate the performance of the algorithm with several numerical examples.

\thispagestyle{empty}

  \vspace{ -100.0in}

  \end{minipage}
\end{center}

\vspace{ 2.60in}
\vspace{ 0.50in}

\begin{center}
  \begin{minipage}[t]{4.4in}
    \begin{center}

\textbf{A Provably Componentwise Backward Stable $\mathbf{O(n^2)}$ QR
Algorithm for the Diagonalization of Colleague Matrices} \\

  \vspace{ 0.50in}

K. Serkh$\mbox{}^{\dagger\, \diamond}$,
V. Rokhlin$\mbox{}^{\ddagger \, \oplus}$,  \\
              \today

    \end{center}
  \vspace{ -100.0in}
  \end{minipage}
\end{center}

\vspace{ 2.00in}

\vfill

\noindent 
$\mbox{}^{\diamond}$  This author's work  was supported in part by the NSERC
Discovery Grants RGPIN-2020-06022 and DGECR-2020-00356.
\\
\noindent 
$\mbox{}^{\oplus}$  This author's work was supported in part under ONR
N00014-18-1-2353 and NSF DMS-1952751. \\

\vspace{2mm}

\noindent
$\mbox{}^{\dagger}$ Dept.~of Math. and Computer Science, University of Toronto,
Toronto, ON M5S 2E4 \\
\noindent
$\mbox{}^{\ddagger}$ Dept.~of Mathematics, Yale University, New Haven, CT 06511

\vspace{2mm}


\vfill
\eject

\tableofcontents

\section{Introduction}

The problem of finding the roots of the polynomial
  \begin{align}
p(x) = c_0 + c_1 x + \cdots + c_{n-1} x^{n-1} + x^n
  \end{align}
is one of the oldest and most classical problems in mathematics. Countless
methods have been proposed for its solution (see, for example,~\cite{pan}
for a history, and the two volumes~\cite{mcnameebook1}
and~\cite{mcnameebook2} for a detailed account of such methods).  In the
1800's, it was observed by Frobenius that the roots of the polynomial are
the eigenvalues of a certain matrix called the \textit{companion matrix},
formed using the polynomial coefficients. A matrix whose eigenvalues are the
roots of $p(x)$ is called a \textit{linearization} of $p(x)$. Given a
linearization of $p(x)$, the roots of the polynomial can thus be recovered
by computing the eigenvalues of the matrix.

If the roots of the polynomial are found numerically, than the computed roots
can be viewed as the exact roots of a perturbed polynomial $p(x)+\delta
p(x)$ with coefficients $c_i + \delta c_i$, where the size of the vector
$\delta c$ is called the \textit{backward error} in the polynomial
coefficients.  The best possible bound on the backward error that a
linearization method can have for general polynomials is
  \begin{align}
\norm{\delta c} \lesssim \norm{c}\ur,
    \label{cbnd1}
  \end{align}
or, in other words, that the relative normwise backward error in the
polynomial coefficients is bounded by machine precision $\ur$ (see, for
example,~\cite{nakatsu}).  The backward
error in the companion matrix method was revealed by the influential
paper~\cite{murakami}, which analyzed the relationship between perturbations
in the companion matrix and perturbations in the polynomial coefficients.
There, the authors proved that if the companion matrix $C$ is perturbed by
the matrix $E$, then the matrix $C+E$ is a linearization of the polynomial
with perturbed coefficients $c_i+\delta c_i$, and that the perturbation
satisfies the normwise bound
  \begin{align}
\norm{\delta c} \lesssim \norm{c} \norm{E}\ur.
  \end{align}
If the eigenvalues are computed by a standard QR algorithm, which is known
to be backward stable (see, for example,~\cite{tisseur}), then the computed
eigenvalues are the exact eigenvalues of $C+E$, where $\norm{E} \lesssim
\norm{C}\ur$. Since $\norm{C} \approx \norm{c}$, it follows that the
backward error in the polynomial coefficients is bounded by
  \begin{align}
\norm{\delta c} \lesssim \norm{c}^2\ur.
    \label{cbnd2}
  \end{align}
Thus, as $\norm{c}$ get larger, the relative backward error in the
coefficients increases.  The companion matrix method, at first glance, would
appear then to have two drawbacks: it falls short of the optimal backward
error bound~(\ref{cbnd1}), and it costs $O(n^3)$ operations as a result of
using the QR algorthim.

The situation improved dramatically in~2007, when~Bini, Eidelman, Gemignani,
and Gohberg  published a paper~\cite{bini1} describing a stable, $O(n^2)$
explicit QR method for companion matrices (around the same time,
Chandrasekaran, Gu, Xia, and Zhu also discovered an $O(n^2)$ method for
companion matrices, see~\cite{chandra}). The algorithm is based on the
observation that the companion matrix and its QR iterates have a certain
structure which allows them to be represented by a collection of $O(n)$
parameters called \textit{generators} (specifically, the companion matrix is
a Hessenberg matrix that is the sum of a unitary matrix and a rank-1
perturbation; matrices of this form are called \textit{fellow matrices}). In
2010, an implicit version of this algorithm, also stable and $O(n^2)$, and
also based on generators, was introduced in~\cite{bini2}. Around the same time,
Van Barel, Vandebril, Van Dooren, and Frederix discovered in~\cite{barel} an
alternative stable, $O(n^2)$ implicit QR algorithm based on representing the
unitary part by so-called \textit{core transformations}, which are rotation
matrices acting only on two adjacent rows at a time (see, for
example,~\cite{core}).  The first example of a proof of backward stability
for an implicit $O(n^2)$ QR algorithm for companion matrices was given by
Aurentz, Mach, Vandebril, and Watkins in~\cite{aurentz1}; this algorithm is
again based on core transformations.  The backward stability result
accompanying this QR algorithm guarantees the sub-optimal
bound~(\ref{cbnd1}), but has the optimal complexity of $O(n^2)$.  Amazingly,
the authors then discovered that the algorithm they had constructed, with
some minor modifications, actually yields the optimal bound~(\ref{cbnd2}) in
practice. An investigation showed that the reason for this remarkable
behavior is that their algorithm is not just backward stable, but is
\textit{componentwise} backward stable, meaning that the backward error in
the companion matrix can be decomposed into proportional backward errors in
each of its components. They published a proof of the componentwise backward
stability of their algorithm, together with a proof that componentwise
backward stability guarantees the bound~(\ref{cbnd1}), in~\cite{aurentz2},
along with numerical experiments.  

Thus, if the coefficients of the polynomial $p(x)$ in the monomial basis are
known, then the algorithm of~\cite{aurentz2} is optimal in both error and
time complexity. However, if the coefficients are not known, then the
companion matrix cannot be used to the find the roots accurately, since the
relationship between the values of the polynomial $p(x)$ and the
coefficients of its monomial expansion is highly unstable (this fact has
been known for many decades, at least as early as
Wilkinson~\cite{peterswilk}).  If the polynomial $p(x)$ is instead expanded
in a basis of Chebyshev polynomials
  \begin{align}
p(x) = c_0 + c_1 T_1(x) + \cdots + c_{n-1} T_{n-1}(x) + T_n(x),
  \end{align}
where $T_i(x)$ is the Chebyshev polynomial of order $i$, then the
relationship between the coefficients and the polynomial is perfectly stable
(see, for example,~\cite{nickapprox}).  In fact, this observation is the
basis for the Chebfun software package (see~\cite{battles}
and~\cite{chebfun}).  An analogue of the companion matrix, constructed from
the Chebyshev expansion coefficients, was discovered in~1961 by
Good~\cite{good}, who called it the \textit{colleague matrix}, and
independently by Spect in~1957~\cite{specht1}--\cite{specht2}.  The first
$O(n^2)$ algorithm for colleague matrices was discovered by Bini, Gemignani,
and Pan in~2005 (even before~\cite{bini1} appeared) and is a stable,
explicit QR algorithm based on generators~\cite{bini3}. Like the companion
matrix, the colleague matrix has a special structure that is preserved over
QR iterations (specifically, the colleague matrix is a Hessenberg matrix
that is the sum of a Hermitian matrix and a rank-1 perturbation). In~2008, Eidelman,
Gemignani, and Gohberg, in~\cite{eidelman}, introduced a stable, $O(n^2)$
implicit QR algorithm.  The relationship between the backward error in the
Chebyshev expansion coefficients and perturbations to the colleague matrix
was first investigated Nakasukasa and Noferini in~\cite{nakatsu}, where the
authors found a lower bound for the backward error in the coefficients,
showing that a backward stable QR algorithm can do no better
than~(\ref{cbnd2}) (around the same time, Lawrence, Van Barel, and Van
Dooren published a general analysis in~\cite{lawrence}, where they also
proved a lower bound for colleague matrices).  In~\cite{perez}, Perez and
Noferini improved on this result and found an upper bound as well, proving
that if the perturbation to the colleague matrix is small, then the
bound~(\ref{cbnd1}) is achieved.  The relationship between componentwise
perturbations to the colleague matrix and the backward error in the
coefficients was described completely in~2019 by Noferini, Robol, and
Vandebril in~\cite{noferini}.

Recently (see~\cite{talk} and~\cite{casulli}), it was observed that certain
$O(n^2)$ structured QR algorithms for colleague matrices are surprisingly
stable, attaining the bound~(\ref{cbnd1}) in many cases, an observation that
mirrors the discovery in~\cite{aurentz2} for the case of companion matrices.
However, unlike in~\cite{aurentz2}, all previously proposed $O(n^2)$
structured QR algorithms for colleague matrices have polynomials for which
the worst-case bound~(\ref{cbnd2}) is attained. Thus, the question of
whether or not there exists a structured $O(n^2)$ QR algorithm that, when
used to find the roots of a colleague matrix, attains the optimal
bound~(\ref{cbnd1}), has remained open.  In this manuscript, we answer this
question in the affirmative by presenting a new, explicit $O(n^2)$ QR
algorithm for colleague matrices (in fact, for all Hessenberg matrices that
have a Hermitian plus rank-1 structure), and prove that our algorithm is
componentwise backward stable.  Combined with the result in~\cite{noferini},
this amounts to a proof that the optimal bound~(\ref{cbnd1}) is attained for
all polynomials $p(x)$.  We demonstrate that this is indeed the case with
several numerical experiments.

The structure of this manuscript is as follows. Section~\ref{sec:prelim}
describes the mathematical and numerical preliminaries.
Section~\ref{sec:algorithm} describes the algorithm, and explains the
significance of each step. In Section~\ref{sec:combackstab}, we prove
rigorously that the algorithm is componentwise backward stable.
Section~\ref{sec:numerical} presents the results of several numerical
experiments. In Section~\ref{sec:conc}, we discuss possible extensions and
generalizations of the algorithm.

\section{Preliminaries}
  \label{sec:prelim}

In this section, we describe the mathematical and numerical preliminaries.

\subsection{Linear Algebra}

The following lemma states that if the sum of a Hermitian matrix and a
rank-1 update $p q^*$ is lower Hessenberg, then the matrix is determined
entirely by its diagonal and superdiagonal together with the vectors $p$ and
$q$.

\begin{lemma}[Eidelman, Gemignani, Gohberg~\cite{eidelman}]
  \label{lem:apr1}
Suppose that $A \in \C^{n\times n}$ is Hermitian, and let $d$ and $\beta$
denote the diagonal and superdiagonal of $A$, respectively. Suppose that
$p, q \in \C^n$ and that $A+pq^*$ is lower Hessenberg. Then
  \begin{align}
a_{i,j} = \left\{
  \begin{array}{ll}
  -p_i q_j^*  & \text{if $j > i+1$} \\
  \beta_i  & \text{if $j = i+1$} \\
  d_i  & \text{if $j = i$} \\
  \overline{\beta_j} & \text{if $j = i-1$} \\
  -q_j p_i^*  & \text{if $j < i-1$}
  \end{array} \right.
  \end{align}
where $a_{i,j}$ denotes the $(i,j)$-th entry of $A$.

\end{lemma}

The following lemma states that if the sum of a matrix and a rank-1 update
$pq^*$ is lower triangular, then the upper Hessenberg part of the matrix is
determined entirely by its diagonal and subdiagonal, together with the vectors
$p$ and $q$.

\begin{lemma}
  \label{lem:tpr1}
Suppose that $B \in \C^{n \times n}$ and let $d$ and $\gamma$ denote the
diagonal and subdiagonal of $B$, respectively. Suppose that $p, q \in \C^n$ and
that $B+pq^*$ is lower triangular. Then
  \begin{align}
b_{i,j} = \left\{
  \begin{array}{ll}
  -p_i q_j^*  & \text{if $j > i$} \\
  d_i  & \text{if $j = i$} \\
  \gamma_j & \text{if $j = i-1$}
  \end{array} \right.
  \end{align}
where $b_{i,j}$ denotes the $(i,j)$-th entry of $B$.

\end{lemma}

The following definition introduces two matrix seminorms that we will need
in our error analysis.

\begin{definition}
  \label{defht}
Suppose that $A \in \C^{n\times n}$ and let $a_{i,j}$ denote the $(i,j)$-th
entry of $A$. We will use the notation $\norm{\cdot}_H$ to denote the square
root of the sum of squares of the entries in the upper Hessenberg part of a
matrix, so that
  \begin{align}
\norm{A}_H = \sqrt{\sum_{j \ge i-1} \abs{a_{i,j}}^2}.
  \end{align}
Likewise, we will use the notation $\norm{\cdot}_T$ to denote the square
root of the sum of squares of the entries in the upper triangular part, so
that
  \begin{align}
\norm{A}_T = \sqrt{\sum_{j \ge i} \abs{a_{i,j}}^2}.
  \end{align}

\end{definition}

The following is a straightforward lemma stating that if a certain sequence of
transformations is applied to a matrix on the right, then the upper triangular
part of the result is determined by only the upper Hessenberg part of the original
matrix.

\begin{lemma}
  \label{lem:hesstri}
Suppose that $B\in \C^{n\times n}$, and let $P_2, P_3, \ldots, P_n \in
\C^{n\times n}$ be matrices such that $P_k$ only affects the $(k-1,k)$-plane
of any vector it is applied to. Define $P \in \C^{n\times n}$ by the formula
$P=P_2 P_3 \cdots P_n$. Then the upper triangular part of $BP^*$ is
determined entirely by the upper Hessenberg part of $B$.  Furthermore, if
$P_2,P_3,\ldots,P_n$ are unitary, then $\norm{BP^*}_T \le \norm{B}_H$.

\end{lemma}

\subsection{Error Analysis}

The following definition introduces the notation used in the error analysis
that appears in this manuscript. We follow the notation used
in~\cite{higham} and~\cite{aurentz2}.

\begin{definition}
Evaluation of an expression in floating point arithmetic is denoted by
$\fl(\cdot)$, and we denote the unit roundoff (or machine epsilon) by
$\ur$. We assume that
  \begin{align}
\fl(x \op y) = (x \op y)(1+\delta), \qquad \abs{\delta} \le \ur,
  \end{align}
where $\op$ stands for any of the basic arithmetic operations $+,-,*,/$.  We
denote computed quantities by a hat, so that $\hat x$ denotes the computed
approximation to $x$. We use the notation $\lesssim$ to mean ``less than or
equal to the right hand side times a modest multiplicative constant
depending on $n$ as a low-degree polynomial'', where the meaning of $n$ is
clear from the context.  Whenever a matrix or vector norm appears to the
left or right of $\lesssim$, we omit the particular choice of norm, since in
finite dimensions all norms are equivalent. When $\ur$ appears in an
expression on the right hand side of $\lesssim$, we ignore all higher order
powers of $\ur$.

\end{definition}

\subsubsection{Floating point computation of complex plane rotations}

The following lemma bounds the forward error of the floating point computation
of a complex plane rotation (see, for example,~\S20 of~\cite{wilkinson}).

\begin{lemma}
  \label{lem:su2}
Suppose that $x=(x_1,x_2)^T \in \C^2$, and let $Q \in \SU(2)$ be 
the complex rotation matrix which eliminates the first entry, so that
$(Qx)_1 = 0$. Let $\hat Q \in \C^{2\times 2}$ be the floating point matrix
defined by
  \begin{alignat}{2}
&\hat Q_{1,1} = c &\qquad &\hat Q_{1,2} = -s \\
&\hat Q_{2,1} = \overline s &\qquad &\hat Q_{2,2} = \overline c \\
  \end{alignat}
where $c = \fl\bigl(x_2 / \sqrt{ \abs{x_1}^2 + \abs{x_2}^2 }\bigr)$
and $s = \fl\bigl(x_1 / \sqrt{ \abs{x_1}^2 + \abs{x_2}^2 }\bigr)$, and
where $c=1$ and $s=0$ if $\norm{x}=0$. Then
  \begin{align}
\norm{ \hat Q - Q } \lesssim \ur.
  \label{qerr}
  \end{align}
\end{lemma}

\subsubsection{Multiplication by complex plane rotations}

The following lemma estimates the forward error of applying a plane rotation to
a vector (see, for example,~\S21 of~\cite{wilkinson}).

\begin{lemma}
  \label{lem:rot}
Suppose that $Q \in \SU(2)$ is a complex rotation matrix, and
let $\hat Q$ be a floating point approximation to $Q$
satisfying~(\ref{qerr}). Suppose further that $x=(x_1,x_2)^T \in \C^2$.
Then
  \begin{align}
\norm{ \fl(\hat Q x) - Qx } \lesssim \norm{x} u.
  \end{align}

\end{lemma}

\subsection{Colleague Matrices and Polynomial Rootfinding}
  \label{sec:precoll}

Suppose that $p(x)$ is a monic polynomial of order $n$ represented by
  \begin{align}
p(x) = \sum_{j=0}^n c_j T_j(x),
  \end{align}
where $c_j \in \R$, $c_n=1$, and $T_j(x)$ is the Chebyshev polynomial of order $j$.
It turns out that the roots of $p(x)$ are the eigenvalues of the (scaled)
$n\times n$ colleague matrix
  \begin{align}
C = \left(
  \begin{array}{ccccc}
0 & \frac{1}{\sqrt{2}} &  &  &   \\
\frac{1}{\sqrt{2}} & 0 & \frac{1}{2}  &   \\
& \frac{1}{2} & \ddots & \ddots & \\
&  & \ddots & 0 & \frac{1}{2} \\
&  &  & \frac{1}{2} & 0
  \end{array} \right)
- \frac{1}{2} e_n \left( 
  \begin{array}{ccccc}
c_0\sqrt{2} & c_1 & c_2 & \cdots & c_{n-1} 
  \end{array} \right),
    \label{colleague}
  \end{align}
where $e_n$ is the $n$-th unit vector (see, for example,~\cite{good}). A
matrix $C$ whose eigenvalues are the roots of $p(x)$ is called a
\textit{linearization} of $p(x)$.
Letting
  \begin{align}
A = \left(
  \begin{array}{ccccc}
0 & \frac{1}{\sqrt{2}} &  &  &   \\
\frac{1}{\sqrt{2}} & 0 & \frac{1}{2}  &   \\
& \frac{1}{2} & \ddots & \ddots & \\
&  & \ddots & 0 & \frac{1}{2} \\
&  &  & \frac{1}{2} & 0
  \end{array} \right)
  \end{align}
and
  \begin{align}
q^*= -\frac{1}{2} \left( 
  \begin{array}{ccccc}
c_0\sqrt{2} & c_1 & c_2 & \cdots & c_{n-1}
  \end{array} \right),
  \end{align}
we see that the colleague matrix $C$ can be written as
  \begin{align}
C = A + e_n q^*,
  \end{align}
where $A$ is Hermitian and $C$ is lower Hessenberg.

The following beautiful theorem by Noferini, Robol, and Vandebril (see
Corollary~5.4 of~\cite{noferini}) bounds the change in the coefficients of
the polynomial being linearized by the componentwise perturbations of
the colleague matrix $C = A + e_n q^*$.

\begin{theorem}
  \label{thm:coefbnds}
Let $C=A+e_n q^*$ be the linearization~(\ref{colleague}) of the monic
polynomial $p(x)$, expressed in the Chebyshev basis. Consider the
perturbations $\norm{\delta A} \le \epsilon_A$, $\norm{\delta e_n} \le
\epsilon_n$, and $\norm{\delta q} \le \epsilon_q$. Then, the matrix
  \begin{align}
C+\delta C = A+\delta A + (e_n + \delta e_n)(q + \delta q)^*
  \end{align}
is a linearization of the polynomial
  \begin{align}
p(x) + \delta p(x) = \sum_{j=0}^n (c_j+\delta c_j) T_j(x),
  \end{align}
where $\norm{\delta c} \lesssim \epsilon_n + \epsilon_q + \norm{c}
\epsilon_A$.

\end{theorem}

\subsection{Stability of Rootfinding Using Linearizations}
  \label{sec:prestab}

Suppose that $p(x)$ is a monic polynomial of order $n$ represented in the
Chebshev basis
  \begin{align}
p(x) = \sum_{j=0}^n c_j T_j(x),
  \end{align}
where $c_j \in \R$, $c_n=1$, and $T_j(x)$ is the Chebyshev polynomial of
order $j$, and let the roots of $p(x)$ be denoted by $x_1,x_2, \ldots x_n
\in \C$. Suppose that a rootfinding algorithm returns the computed roots
$\hat x_1, \hat x_2,\ldots, \hat x_n \in \C$. If the computed roots 
are the exact roots of some perturbed polynomial
  \begin{align}
p(x) + \delta p(x) = \sum_{j=0}^n (c_j+\delta c_j) T_j(x),
  \end{align}
where
  \begin{align}
\frac{\norm{\delta c}}{\norm{c}} \lesssim \ur,
  \end{align}
then we say that the rootfinding algorithm is backward stable. In fact, this
is the best backward stability bound that can be hoped for, for general
polynomials $p(x)$ (see the discussion in Appendix~A of~\cite{nakatsu}). 

\begin{remark}
  \label{rem:nonmonic}
Suppose that $p(x)$ is a polynomial of order $n$, that is \textit{not}
monic, expressed in the Chebyshev basis
  \begin{align}
p(x) = \sum_{j=0}^n a_j T_j(x),
    \label{nonmonexp}
  \end{align}
where $a_j \in \R$ and $T_j(x)$ is the Chebyshev polynomial of
order $j$. Clearly, the roots of $p(x)$ are identical to the roots of
$p(x)/a_n$. Let $c_j=a_j/a_n$, for $j=0,1,\ldots,n$. If a backward stable
rootfinding algorithm is applied to the monic polynomial $p(x)/a_n$, then,
letting $\delta a = \delta c \cdot a_n$, the algorithm is also backward
stable with respect to the original coefficients $a_j$, since
  \begin{align}
\frac{\norm{\delta a}}{\norm{a}} = 
\frac{\frac{1}{a_n}\norm{\delta a}}{\frac{1}{a_n}\norm{a}} = 
\frac{\norm{\delta c}}{\norm{c}} \lesssim \ur.
  \end{align}
\end{remark}

When linearization is used as a rootfinding algorithm, the stability of the
computed roots comes from the stability of the eigenvalue algorithm applied to
the colleague matrix $C$.  If the eigenvalues of $C$ are computed by an 
unstructured QR algorithm, then the backward error $\delta C$ on $C$
is bounded by $\norm{\delta C}\lesssim \norm{C}\ur$. Since the backward
error is unstructured, it follows that $\norm{\delta A} \approx \norm{C}\ur$,
so, by Theorem~\ref{thm:coefbnds} together with the fact that $\norm{C}
\approx \norm{c}$, the backward error in the coefficients is bounded by
$\norm{\delta c} \lesssim \norm{c}^2 \ur$.

\begin{remark}
This backward error can be reduced by partially balancing the matrix $C$.
Suppose that, instead of computing the eigenvalues of $C$, we compute the
eigenvalues of
  \begin{align}
&\hspace*{-4.5em}
\tilde C = \left(
  \begin{array}{ccccc}
0 & \frac{1}{\sqrt{2}} &  &  &   \\
\frac{1}{\sqrt{2}} & 0 & \frac{1}{2}  &   \\
& \frac{1}{2} & \ddots & \ddots & \\
&  & \ddots & 0 & \frac{\norm{c}^\frac{1}{2}}{2} \\
&  &  & \frac{1}{2 \norm{c}^\frac{1}{2} } & 0
  \end{array} \right)
- \frac{1}{2 \norm{c}^\frac{1}{2}} e_n \left( 
  \begin{array}{ccccc}
c_0\sqrt{2} & c_1 & c_2 & \cdots & \norm{c}^\frac{1}{2} c_{n-1} 
  \end{array} \right).
  \end{align}
Provided that the entry in the $(n,n)$-position is small,
we have that $\norm{\tilde C} \approx \norm{c}^\frac{1}{2}$, so the backward
error $\delta \tilde C$ of unstructured QR is bounded by $\norm{\delta
\tilde C} \lesssim \norm{c}^\frac{1}{2}\ur$. In practice, it turns out that
$\norm{\delta A} \approx \norm{\delta \tilde C} \lesssim
\norm{c}^\frac{1}{2}$, so $\norm{\delta c} \lesssim \norm{c}^\frac{3}{2}
\ur$. The assumption that the $(n,n)$-th element is small is not always
satisfied. However, usually the norm of $c$ is large because the last
coefficient $a_n$ in the non-monic expansion~(\ref{nonmonexp}) is small.
When this is the case, we can simply raise the order of the expansion by one
by taking an additional term.  The last two terms will both be small and
roughly the same size, making the $(n,n)$-th element small.
Notice also that, by adjusting the last row, the matrix $\tilde C$
can be represented as a symmetric tridiagonal matrix of magnitude
$\norm{c}^\frac{1}{2}$ plus a rank-1 matrix of magnitude
$\norm{c}^\frac{1}{2}$.

\end{remark}

\begin{remark}
Let $\texttt{bal}(C)$ denote the matrix $C$ after complete balancing.
Remarkably, in some situations, $\norm{\texttt{bal}(C)} \approx 1$ even
when $\norm{c}$ is large. Thus, complete balancing can completely eliminate
large entries in $C$, at the expense of destroying its symmetric tridiagonal
plus rank-1 structure. See Remark~\ref{rem:balanstab}, as well as the
paper~\cite{parlett}, for a more detailed discussion.

\end{remark}

\begin{remark}
While unstructured QR, applied to the colleague matrix, is known to achieve
only the backward error bound $\norm{\delta c} \lesssim \norm{c}^2 \ur$, the
QZ algorithm, applied to an appropriately scaled matrix pencil,
does result in a backward stable rootfinder with the bound $\norm{\delta c}
\lesssim \norm{c} \ur$ (see, for example,~\cite{nakatsu}). This is because
the eigenvalue problem for the colleague matrix can be written as a matrix
pencil $A-\lambda B$, where both $A$ and $B$ are small, and a backward
stable QZ algorithm applied to the pencil  computes the exact eigenvalues of
a perturbed pencil $(A+\delta A) - \lambda (B+\delta B)$, where
$\norm{\delta A} \lesssim \norm{A}\ur$ and $\norm{\delta B} \lesssim
\norm{B}\ur$.  Unfortunately, it appears to be very difficult to construct
structured, $O(n^2)$ QZ algorithms for colleague matrices that retain the
nice stability properties of the unstructured, $O(n^3)$ QZ algorithm.

\end{remark}

The following theorem, stated in a slightly different form
in~\cite{noferini}, says that if the eigenvalues of the colleague matrix
$C=A+e_n q^*$ are computed using a componentwise backward stable algorithm,
then linearization is backward stable as a rootfinding algorithm. It follows
immediately from Theorem~\ref{thm:coefbnds}.

\begin{theorem}
  \label{thm:linbackstab}
Suppose that the eigenvalues of the colleague matrix $C=A+e_n q^*$ are
computed by a componentwise backward stable algorithm, in the sense that the
computed eigenvalues are the exact eigenvalues of the matrix
  \begin{align}
C+\delta C = A+\delta A + (e_n + \delta e_n)(q + \delta q)^*,
  \end{align}
where $\norm{\delta A} \lesssim \norm{A}\ur \approx \ur$, $\norm{\delta e_n}
\lesssim \norm{e_n} \ur \approx \ur$, and $\norm{\delta q} \lesssim
\norm{q}\ur$. Then, linearization is backward stable as a rootfinding
algorithm, with $\norm{\delta c} \lesssim \norm{c} \ur$.

\end{theorem}

\subsection{Conventions}

It was pointed out to the authors that, while the Hessenberg matrices in
this manuscript are all lower Hessenberg, the standard convention in
numerical linear algebra is to study the transpose of the problem, and
consider only upper Hessenberg matrices (see~\cite{corless-email}).  The
upper Hessenberg form is much better notationally since, in upper Hessenberg
form, the first elimination step eliminates the entry in the
$(2,1)$-position, while, in lower Hessenberg form, the entry in the
$(n-1,n)$-position is eliminated first. Furthermore, the upper Hessenberg
form is more convenient when representing polynomials in the Lagrange basis
(see, for example,~\cite{corless}). Unfortunately, at the time that this was
all pointed out, most of the writing and numerical codes were complete, and
had been written in lower Hessenberg form because of a historical fluke
related to the structure of old explicit QR codes that were used as a
template for our algorithm.

\section{The Algorithm}
  \label{sec:algorithm}

In this section, we give an overview of our algorithm. We begin by describing
the class of matrices our algorithm can be applied to.
Let $\mathcal{F}_n \subset \C^{n\times n}$ be the set of lower Hessenberg
matrices of the form
  \begin{align}
A + pq^*,
  \end{align}
where $A\in C^{n\times n}$ is Hermitian and $p, q^* \in \C^n$. Eidelman,
Gemignani, and Gohberg observed in~\cite{eidelman} that the matrix $A$
is determined entirely by:
\begin{enumerate}

\item The diagonal entries $d_i=a_{i,i}$, for $i=1,2,\ldots,n$;

\item The superdiagonal entries $\beta_i=a_{i,i+1}$ for $i=1,2,\ldots,n-1$;

\item The vectors $p$ and $q$

\end{enumerate}
(see Lemma~\ref{lem:apr1}).  Following~\cite{eidelman}, we call these four
vectors the \textit{basic elements} or \textit{generators} of $A$.
In~\cite{eidelman}, the authors construct an implicit QR algorithm
that takes advantage of this structure to achieve a cost of $O(n^2)$. They
also prove that their algorithm is backward stable, in the sense that, when
the algorithm is used to compute the eigenvalues of a matrix $C \in
\mathcal{F}_n$, the computed eigenvalues are the exact eigenvalues of
$C+\delta C$, where $\norm{\delta C} \lesssim \norm{C}\ur$. This is the same
backward stability bound that is provided by an unstructured QR algorithm.

\emph{
In this manuscript, we describe a new explicit QR algorithm for matrices
$A+pq^* \in \mathcal{F}_n$, that also has the cost $O(n^2)$, and prove that
our algorithm is \emph{componentwise} backward stable, in the sense that the
computed eigenvalues are the exact eigenvalues of $(A+\delta A) + (p+\delta
p)(q+\delta q)^*$, where $\norm{\delta A} \lesssim \norm{A}\ur$,
$\norm{\delta p} \lesssim \norm{p}\ur$, and $\norm{\delta q} \lesssim
\norm{q}\ur$.
}

To motivate our algorithm, consider first the naive unshifted QR algorithm
in exact arithmetic, applied to the matrix $C=A+pq^*$. Let the matrix $U_n
\in C^{n\times n}$ be the unitary matrix that rotates the $(n-1,n)$-plane so
that 
  \begin{align}
(U_n C)_{n-1,n} = 0,
  \end{align}
eliminating the superdiagonal in the $(n-1,n)$-th position.
Likewise, let $U_{n-1} \in C^{n\times n}$ denote the unitary matrix rotating
the $(n-2,n-1)$-plane so that
  \begin{align}
(U_{n-1} U_n C)_{n-2,n-1} = 0,
  \end{align}
eliminating the superdiagonal in the $(n-2,n-1)$-th position. Continuing in
this fashion, let $U_{n-2}, U_{n-3},\ldots, U_2$ be the unitary matrices
eliminating the superdiagonal entries in the $(n-3,n-2), \allowbreak
(n-4,n-3),\allowbreak \ldots, (1,2)$ positions of the matrices $(U_{n-1} U_n
C), (U_{n-2}U_{n-1} U_n C), \allowbreak \ldots,\allowbreak (U_{3}\cdots U_{n-1}U_n C)$,
respectively.  Letting $U=U_2 U_3 \cdots U_n$, we have that $UC$ is lower
triangular.  This matrix has the form
  \begin{align}
UC = B + (Up)q^*,
  \end{align}
where the upper Hessenberg part of the matrix $B=UA$ is determined entirely
by:
\begin{enumerate}

\item The diagonal entries $\underline d_i = b_{i,i}$, for $i=1,2,\ldots,n$;

\item The subdiagonal entries $\underline \gamma_i = b_{i+1,i}$, for
$i=1,2,\ldots,n-1$;

\item The vectors $\underline p = Up$ and $q$

\end{enumerate}
(see Lemma~\ref{lem:tpr1}). Like with the matrix $A$, we call these four
vectors the \textit{basic elements} or \textit{generators} of (the upper
Hessenberg part of) $B$.

Next, the matrix is multiplied by $U^*$ on the right; clearly,
  \begin{align}
UCU^* = BU^* + Up(Uq)^*,
  \end{align}
so
  \begin{align}
UCU^* = UAU^* + Up(Uq)^*.
  \end{align}
It's easy to show that, since $UC$ is lower triangular and $U^*=U_n^*
U_{n-1}^* \cdots U_2^*$, where $U_k$ rotates the $(k-1,k)$-plane, the matrix
$UCU^*$ is lower Hessenberg.  Thus, $UCU^* \in \mathcal{F}_n$, and the
matrix $\underline{A} = UAU^*$ is determined entirely by its diagonal and
superdiagonal, together with $\underline{p}=Up$ and $\underline{q}=Uq$.
Furthermore, the upper triangular part of $UAU^*$ is determined entirely by
the upper Hessenberg part of $B$ (see Lemma~\ref{lem:hesstri}), and since
$UAU^*$ is Hermitian, it follows that the whole of the matrix $UAU^*$ is
determined entirely by the upper Hessenberg part of $B$.

In our algorithm, we use only the basic elements of $A$ and $B$ to represent
our matrices. This results in a single iteration of our QR algorithm
requiring $O(n)$ operations.  Furthermore, we prove that the matrix
$\hat{\underline{A}}$ and vectors $\hat{\underline{p}}$ and
$\hat{\underline{q}}$, computed by a single iteration of our QR algorithm, have
the componentwise forward error bounds $\norm{\hat{\underline{A}} - \underline{A}}
\lesssim \norm{A}\ur$, $\norm{\hat{\underline{p}} - \underline{p}} \lesssim
\norm{p}\ur$, and $\norm{\hat{\underline{q}} - \underline{q}} \lesssim
\norm{p}\ur$. We then show that these componentwise forward error bounds result
in componentwise backward stability.

\subsection{Eliminating the Superdiagonal}
  \label{sec:algelim}

In this section, we describe how our algorithm performs a single elimination
of a superdiagonal element (see Algorithm~\ref{alg:elim}).  Suppose that we
have already eliminated the superdiagonal elements in the positions
$(n-1,n), (n-2,n-1), \ldots, (k,k+1)$.  Let $p^{(k+1)}=U_{k+1}U_{k+2}\cdots
U_n p$ and $B^{(k+1)}=U_{k+1}U_{k+2}\cdots U_n A$.  Suppose further that
$\hat p^{\,(k+1)}$ and $\hat B^{(k+1)}$ are the computed approximations to
$p^{(k+1)}$ and $B^{(k+1)}$, and that the upper Hessenberg part of the
computed matrix $\hat B^{(k+1)}$ is represented by its generators: 
\begin{enumerate}

\item The diagonal elements $\hat d^{\;(k+1)}_i = \hat b^{\,(k+1)}_{i,i}$, for
$i=1,2,\ldots,n$;

\item The superdiagional elements $\hat \beta^{\,(k+1)}_i = \hat
b^{\,(k+1)}_{i,i+1}$, for $i=1,2,\ldots,k-1$;

\item The subdiagional elements $\hat \gamma^{(k+1)}_i = \hat
b^{\,(k+1)}_{i+1,i}$, for $i=1,2,\ldots,n-1$;

\item The vectors $\hat p^{\,(k+1)}$ and $q$, from which the remaining elements
in the upper Hessenberg part are inferred.

\end{enumerate}
Suppose that $\norm{\hat B^{(k+1)} - B^{(k+1)}}_H \lesssim \norm{A}\ur$ and
$\norm{\hat p^{\,(k+1)} - p^{(k+1)}}\lesssim \norm{p}\ur$. Notice that, if we define
$\hat B^{(n+1)} = B^{(n+1)} = A$ and $\hat p^{\,(n+1)} = p^{(n+1)} = p$, then 
this is obviously true for $k=n$.

\begin{figure}
  \renewcommand*{\arraystretch}{2.0}
\centering
\begin{minipage}{0.5\textwidth}
\begin{align*}
&\hspace*{-4em}
\left(
  \begin{array}{ccccccc}
\ddots & \hat \gamma_{k-3}^{(k+1)} & \hat d_{k-2}^{\;(k+1)} & 
  \hat\beta_{k-2}^{\,(k+1)} & -\hat p_{k-2}^{\,(k+1)} q_k^* & -\hat
  p_{k-2}^{\,(k+1)} q_{k+1}^* & \cdots \\
\hline
\cdots & \times & \hat \gamma_{k-2}^{(k+1)} & \hat d_{k-1}^{\;(k+1)} &
  \hat\beta_{k-1}^{\,(k+1)} & -\hat p_{k-1}^{\,(k+1)} q_{k+1}^* & \cdots \\
\cdots & \times & \hat b_{k,k-2}^{\,(k+1)} & \hat \gamma_{k-1}^{(k+1)} & \hat
  d_k^{\;(k+1)} & -\hat p_k^{\,(k+1)} q_{k+1}^* & \cdots \\
\hline
\cdots & \times & \times & \times & \hat \gamma_{k}^{(k+1)} &
  \hat d_{k+1}^{\;(k+1)} & \ddots
  \end{array}
\right)
\end{align*}
\end{minipage}
  \caption{The $(k-1)$-th and $k$-th rows of $\hat B^{(k+1)}$, represented by
  its generators.  \label{fig:bhat}}
\end{figure}

To eliminate the superdiagonal element in the $(k-1,k)$ position of $\hat
B^{(k+1)} + \hat p^{\,(k+1)} q^*$, we first compute the rotation matrix $Q_k
\in \SU(2)$ that eliminates it by a rotation in the $(k-1,k)$-plane (see
Line~\ref{alg:elim:qk} of Algorithm~\ref{alg:elim}).  Next, we apply the
rotation matrix separately to the generators of $\hat B^{(k+1)}$ and to
the vector $\hat p^{\,(k+1)}$. Since we are only interested in computing the upper
Hessenberg part of $\hat B^{(k)}$, we need to update the subdiagonal element
in the $(k-1,k-2)$ position of $\hat B^{(k+1)}$, represented by $\hat
\gamma^{(k+1)}_{k-2}$ (see Figure~\ref{fig:bhat}). However, this calculation
requires the sub-subdiagonal entry in the $(k,k-2)$ position of $\hat
B^{(k+1)}$, which is unknown to us since only the upper Hessenberg part of
$\hat B^{(k+1)}$ is available. Fortunately, it can be recovered by the
following trick.

Since $B^{(k+1)} = U_{k+1} U_{k+2} \cdots U_n A$ and $A$ is Hermitian, it
follows that $B^{(k+1)} U_n^* U_{n-1}^* \cdots U_{k+1}^*$ is also Hermitian.
Thus,
  \begin{align}
(B^{(k+1)} U_n^* U_{n-1}^* \cdots U_{k+1}^*)_{k,k-2} = 
\overline{(B^{(k+1)} U_n^* U_{n-1}^* \cdots U_{k+1}^*)}_{k-2,k}.
  \end{align}
Furthermore, since right-multiplication by $U_j^*$ only affects columns $j$
and $j-1$ (see Figure~\ref{fig:bhat}), we have that right-multiplication by $U_n^*
U_{n-1}^* \cdots U_{k+1}^*$ leaves $b_{k,k-2}^{(k+1)}$ unchanged. Therefore,
  \begin{align}
b^{(k+1)}_{k,k-2} = 
\overline{(B^{(k+1)} U_n^* U_{n-1}^* \cdots U_{k+1}^*)}_{k-2,k}.
    \label{bkp1form}
  \end{align}
We know that the entries in the $(k-2,k), (k-2,k+1),\ldots, (k-2,n)$ positions
of $\hat B^{(k+1)}$ are inferred from $\hat p^{\,(k+1)}$ and $q$ by the formula
  \begin{align}
\hat b^{\,(k+1)}_{k-2,\ell} = -\hat p_{k-2}^{\,(k+1)} q_\ell^*,
    \label{bkp1infer}
  \end{align}
for $\ell=k,k+1,\ldots,n$. Combining~(\ref{bkp1form}) and~(\ref{bkp1infer}),
we thus have that the sub-subdiagonal entry in the $(k,k-2)$ position of
$\hat B^{(k+1)}$ can be recovered by the formula
  \begin{align}
\hat b^{\,(k+1)}_{k,k-2} = 
\overline{(-\hat p_{k-2}^{\,(k+1)}q^* U_n^* U_{n-1}^* \cdots U_{k+1}^*)}_{k}.
  \end{align}
Defining $\tilde q^{\,(k+1)} = U_{k+1} U_{k+2} \cdots U_n q$, we have
  \begin{align}
\hat b^{\,(k+1)}_{k,k-2} = -\tilde q_k^{\,(k+1)} \hat p^{\,(k+1)*}_{k-2}.
  \end{align}
By computing the vector $\hat{\tilde{q}}^{\,(k+1)}$ (see Line~\ref{alg:elim:qtil}
of Algorithm~\ref{alg:elim}), we use this formula to recover the
sub-subdiagonal element $\hat b^{\,(k+1)}_{k,k-2}$.

Thus, the element in the $(k-1,k-2)$ position of $\hat B^{(k+1)}$,
represented by $\hat \gamma_{k-2}^{(k+1)}$, is updated in
Line~\ref{alg:elim:subxsubsub} of Algorithm~\ref{alg:elim}. Next, the
elements in the $(k-1,k-1)$ and $(k,k-1)$ positions of $\hat B^{(k+1)}$,
represented by $\hat d_{k-1}^{\;(k+1)}$ and $\hat \gamma_{k-1}^{(k+1)}$,
respectively, are updated in a straightforward way in
Line~\ref{alg:elim:diagxsub}.  Finally, the elements in the $(k-1,k)$ and
$(k,k)$ positions of $\hat B^{(k+1)}$, represented by $\hat
\beta_{k-1}^{\,(k+1)}$ and $\hat d_{k}^{\;(k+1)}$, respectively, are updated in
Line~\ref{alg:elim:supxdiag}, and the vector $\hat p^{\,(k+1)}$ is rotated in
Line~\ref{alg:elim:p}.

Since we've eliminated the superdiagonal element in the $(k-1,k)$ position
of $\hat B^{(k+1)} + \hat p^{\,(k+1)} q^*$, we have that the $(k-1,k)$ element
of the matrix $\hat B^{(k)}$ is inferred from $\hat p^{\,(k)}$ and $q$ by the
formula
  \begin{align}
\hat b_{k-1,k}^{\,(k)} = -\hat p_{k-1}^{\,(k)} q_k^*.
  \end{align}
Now, we would like the upper Hessenberg part of $\hat B^{(k)}$ to have a
small componentwise error, so that $\norm{\hat B^{(k)} - B^{(k)}}_H \lesssim
\norm{A}\ur$. However, consider the following scenario. Suppose that the norm of
$(\hat p_{k-1}^{\,(k+1)} q_k^*, \hat p_k^{\,(k+1)} q_k^*)^T$ is much larger than
$\norm{A}$. By Lemma~\ref{lem:rot},
the error in $\hat p_{k-1}^{\,(k)} q_k^*$ will be approximately
$\Bigl(\sqrt{\abs{\hat p_{k-1}^{\,(k+1)} q_k^*}^2 + \abs{\hat p_k^{\,(k+1)}
q_k^*}^2}\Bigr)\ur$, which will be much larger than $\norm{A}\ur$. In this situation
then, even if $\norm{ \hat p^{\,(k+1)} - p^{(k+1)}} \lesssim \norm{p}\ur$ and
$\norm{ \hat B^{(k+1)} - B^{(k+1)}}_H \lesssim \norm{A}\ur$, we will \textit{not}
have $\norm{ \hat B^{(k)} - B^{(k)}}_H \lesssim \norm{A}\ur$.
To remedy this, we must apply a correction to $\hat p_{k-1}^{\,(k)}$. Recall that
the rotation matrix $Q_k$ was defined to be the matrix eliminating the $(k-1,k)$-th
entry of $\hat B^{(k+1)} + \hat p^{\,(k+1)} q^*$ in exact arithmetic.
If we let $(\mathring p_{k-1}^{(k)},\mathring p_{k}^{(k)})^T$ denote
the result of applying $Q_k$ to $(\hat p_{k-1}^{\,(k+1)}, \hat p_k^{\,(k+1)})^T$
in exact arithmetic, and likewise let $(\mathring \beta_{k-1}^{(k)}, \mathring
d_k^{(k)})^T$ denote the result of applying $Q_k$ to 
$(\hat \beta_{k-1}^{\,(k+1)}, \hat d_k^{\;(k+1)})^T$ in exact arithmetic,
then, by the definition of $Q_k$, we have
  \begin{align}
\mathring \beta_{k-1}^{(k)} + \mathring p_{k-1}^{(k)} q_k^* = 0.
  \end{align}
By Lemma~\ref{lem:rot}, we have that 
  \begin{align}
\abs{\mathring \beta_{k-1}^{(k)} - \hat
\beta_{k-1}^{\,(k)}} \lesssim \Bigl( \sqrt{ \abs{\hat\beta_{k-1}^{\,(k+1)}}^2
+ \abs{\hat d_k^{\;(k+1)}}^2 }\Bigr)\ur
  \end{align}
and
  \begin{align}
\abs{ \mathring p_{k-1}^{(k)}q_k^* - \hat p_{k-1}^{\,(k)} q_k^*}
\lesssim \Bigl( \sqrt{ \abs{\hat p_{k-1}^{\,(k+1)}q_k^*}^2 + 
\abs{\hat p_{k}^{\,(k+1)} q_k^*}^2 }\Bigr) \ur.
  \end{align}
Thus, if $\abs{\hat p_{k-1}^{\,(k+1)}q_k^*}^2 + \abs{\hat p_{k}^{\,(k+1)}
q_k^*}^2 > \abs{\hat\beta_{k-1}^{\,(k+1)}}^2 + \abs{\hat d_k^{\;(k+1)}}^2$,
then  we set
  \begin{align}
\hat p_{k-1}^{\,(k)} q_k^* = -\hat \beta_{k-1}^{\,(k)},
  \end{align}
so
  \begin{align}
\hat p_{k-1}^{\,(k)} = -\hat \beta_{k-1}^{\,(k)}/q_k^*
  \end{align}
(see Line~\ref{alg:elim:corr} of Algorithm~\ref{alg:elim}).
With this correction to $\hat p_{k-1}^{\,(k)}$, it is easy to see that
$\norm{ \hat B^{(k)} - B^{(k)} }_H \lesssim \norm{A}\ur$ and
$\norm{ \hat p^{\,(k)} - p^{(k)} }\lesssim \norm{p}\ur$. 
If, on the other hand,  $\abs{\hat p_{k-1}^{\,(k+1)}q_k^*}^2 + \abs{\hat
p_{k}^{\,(k+1)} q_k^*}^2 \le \abs{\hat\beta_{k-1}^{\,(k+1)}}^2 + \abs{\hat
d_k^{\;(k+1)}}^2$, then the correction is not neccessary, since in this case
the error in $\hat p_{k-1}^{\,(k)}q_k^*$ is smaller than the error in $\hat
\beta_{k-1}^{\,(k)}$.

This process of eliminating the superdiagonal elements can be repeated,
until the upper Hessenberg part of the matrix $\hat B$, approximating
$B=U_{2}U_{3}\cdots U_n A$, is obtained, together with
$\hat{\underline{p}}$, approximating $\underline{p}=U_{2}U_{3}\cdots U_n p$
(see Algorithm~\ref{alg:elim}).  In Section~\ref{sec:forerr},
Lemma~\ref{lem:elim}, we prove that the forward errors in the upper
Hessenberg part of $\hat B$ and in the vector $\hat{\underline{p}}$ are
proportional to $\norm{A}\ur$ and $\norm{p}\ur$, respectively.

\begin{algorithm}
\caption{
  \label{alg:elim}
(A single elimination of the superdiagonal) \textbf{Inputs:}
This algorithm accepts
as inputs two vectors $d$ and $\beta$ representing the diagonal and
superdiagonal, respectively, of an $n\times n$ Hermitian matrix $A$, as well
as two vectors $p$ and $q$ of length $n$, where $A+pq^*$ is lower
Hessenberg.  \textbf{Outputs:} It returns as its outputs the rotation matrices
$Q_2,Q_3,\ldots, Q_n \in \C^{2 \times 2}$ so that, letting $U_k \in
\C^{n\times n}$, $k=2,3,\ldots,n$, denote the matrices that rotate the
$(k-1,k)$-plane by $Q_k$, $U_2 U_3 \cdots U_n (A + pq^*)$ is lower
triangular. It also returns the vectors $\underline d$, $\underline \gamma$,
and $\underline p$, where $\underline d$ and $\underline \gamma$ represent
the diagonal and subdiagonal, respectively, of the matrix $U_2 U_3 \cdots
U_n A$, and $\underline p = U_2 U_3 \cdots U_n p$.}

\begin{algorithmic}[1]

\State Set $\gamma\gets \overline\beta$, where $\gamma$ represents the subdiagonal.
\State Make a copy of $q$, setting $\tilde q\gets q$.

\For{$k=n,n-1,\ldots,2$}

  \State 
    \label{alg:elim:qk}
  Construct the $2\times 2$ rotation matrix $Q_k \in \SU(2)$ so that
  \begin{align*}
    \Bigl( Q_k \left[
      \begin{array}{c}
      \beta_{k-1} + p_{k-1} q_k^* \\
      d_k + p_k q_k^*
      \end{array}\right]
    \Bigr)_1 = 0.
  \end{align*}

  \If{$k\ne 2$}
    \State 
      \label{alg:elim:subxsubsub}
    Rotate the subdiagonal and the sub-subdiagonal:
    \begin{align*}
      \gamma_{k-2} \gets \Bigl( Q_k
      \left[
      \begin{array}{c}
        \gamma_{k-2} \\
       -\tilde{q}_k p_{k-2}^* 
      \end{array} \right] \Bigr)_1
    \end{align*}
  \EndIf

  \State 
    \label{alg:elim:diagxsub}
  Rotate the diagonal and the subdiagonal:
  \begin{math}
    \left[
    \begin{array}{c}
      d_{k-1} \\
      \gamma_{k-1}
    \end{array} \right]
    \gets Q_k
    \left[
    \begin{array}{c}
      d_{k-1} \\
      \gamma_{k-1}
    \end{array} \right].
  \end{math}

  \State 
    \label{alg:elim:supxdiag}
  Rotate the superdiagonal and the diagonal:
  \begin{math}
    \left[
    \begin{array}{c}
      \beta_{k-1} \\
      d_{k}
    \end{array} \right]
    \gets Q_k
    \left[
    \begin{array}{c}
      \beta_{k-1} \\
      d_{k}
    \end{array} \right].
  \end{math}

  \State 
    \label{alg:elim:p}
  Rotate $p$:
  \begin{math}
    \left[
    \begin{array}{c}
      p_{k-1} \\
      p_{k}
    \end{array} \right]
    \gets Q_k
    \left[
    \begin{array}{c}
      p_{k-1} \\
      p_{k}
    \end{array} \right]
  \end{math}

  \If{ $\abs{ p_{k-1}q_k^* }^2 + \abs{ p_k q_k^* }^2 >
    \abs{\beta_{k-1}}^2 + \abs{d_k}^2$ } 
    \State 
      \label{alg:elim:corr}
    Correct the vector $p$, setting
      \begin{math}
        p_{k-1} \gets -\frac{\beta_{k-1}}{q_k^*}
      \end{math}
  \EndIf

  \State 
    \label{alg:elim:qtil}
  Rotate $\tilde q$:
  \begin{math}
    \left[
    \begin{array}{c}
      \tilde{q}_{\,k-1} \\
      \tilde{q}_{\,k}
    \end{array} \right]
    \gets Q_k
    \left[
    \begin{array}{c}
      \tilde{q}_{k-1} \\
      \tilde{q}_{k}
    \end{array} \right]
  \end{math}

\EndFor

\State Set $\underline d \gets d$, $\underline \gamma \gets \gamma$, and
$\underline p \gets p$.

\end{algorithmic}
\end{algorithm}

\subsection{Rotating Back to Hessenberg Form}
  \label{sec:algrotback}

In this section, we describe how our algorithm rotates the triangular matrix
produced by an elimination of the superdiagonal back to lower Hessenberg form
(see Algorithm~\ref{alg:rotback}).
Suppose that $B$ is an $n\times n$ matrix and that $p$ and $q$ are vectors
such that $B+p q^*$ is lower triangular. Notice that this condition is
satisfied by the matrix $\hat{B}$ and the vectors $\hat{\underline{p}}$ and $q$
from the preceding section, produced by an elimination of the superdiagonal.
Let $\gamma$ denote the subdiagonal of $B$.  Suppose that we have already
applied the rotation matrices $U_n^*, U_{n-1}^*, \cdots, U_{k+1}^*$ to the
right of $B$ and $q^*$, and let $q^{(k+1)}=U_{k+1} U_{k+2} \cdots U_n q$ and
$A^{(k+1)}=B U_n^* U_{n-1}^* \cdots U_{k+1}^*$.  Suppose that
$\hat{q}^{\,(k+1)}$ and $\hat A^{(k+1)}$ are the computed approximations to
$q^{(k+1)}$ and $A^{(k+1)}$, respectively, and that the upper triangular part
of $\hat A^{(k+1)}$ is represented by its
generators:
\begin{enumerate}

\item The diagonal entries $\hat d^{\;(k+1)}_i=\hat a^{\,(k+1)}_{i,i}$, for
$i=1,2,\ldots,n$;

\item The superdiagonal entries $\hat \beta^{\,(k+1)}_i=\hat
a^{\,(k+1)}_{i,i+1}$ for $i= k,k+1,\ldots,n-1$;

\item The vectors $p$ and $\hat q^{\,(k+1)}$, from which the remaining
elements in the upper triangular part are inferred.

\end{enumerate}
Suppose that $\norm{\hat A^{(k+1)} - A^{(k+1)}}_T \lesssim \norm{B}_H\ur$ and
$\norm{\hat q^{\,(k+1)} - q^{(k+1)}}\lesssim \norm{q}\ur$. Notice that, if we define
$\hat A^{(n+1)} = A^{(n+1)} = B$ and $\hat q^{\,(n+1)} = q^{(n+1)} = q$, then 
this is obviously true for $k=n$.
\begin{figure}
  \renewcommand*{\arraystretch}{1.7}
\centering
\begin{minipage}{0.5\textwidth}
\begin{align*}
&\hspace*{-4em}
\left(
  \begin{array}{c|cc|c}
\ddots & \vdots & \vdots & \vdots \\
\hat d_{k-2}^{\;(k+1)} & -p_{k-2} \hat q_{k-1}^{\,(k+1)*} &
  -p_{k-2} \hat q_{k}^{\,(k+1)*} & -p_{k-2} \hat q_{k+1}^{\,(k+1)*} \\
\gamma_{k-2} & \hat d_{k-1}^{\;(k+1)} & -p_{k-1} \hat q_{k}^{\,(k+1)*} &
  -p_{k-1} \hat q_{k+1}^{\,(k+1)*} \\
\times & \gamma_{k-1} & \hat d_{k}^{\;(k+1)}  & \hat \beta_{k}^{\,(k+1)}  \\
\times & \times & \times & \hat d_{k+1}^{\;(k+1)}  \\
\vdots & \vdots & \vdots & \ddots 
  \end{array}
\right)
\end{align*}
\end{minipage}
  \caption{The $(k-1)$-th and $k$-th columns of $\hat A^{(k+1)}$,
  represented by its generators.  \label{fig:ahat}}
\end{figure}

To apply the matrix $U^*_k$ to $\hat A^{(k+1)} + p \hat q^{\,(k+1)}$ on the
right,  we apply the matrix $Q_k^* \in \SU(2)$ separately to the generators
of $\hat A^{(k+1)}$ and to the vector $\hat q^{\,(k+1)}$. We start by rotating
the diagonal and superdiagonal elements in the $(k-1,k-1)$ and $(k-1,k)$
positions of $\hat A^{(k+1)}$, represented by $\hat d_{k-1}^{\;(k+1)}$ and
$-p_{k-1} \hat q_k^{\,(k+1)*}$, respectively, in
Line~\ref{alg:rotback:diagsup} of  Algorithm~\ref{alg:rotback}, saving the
superdiagonal element in $\hat \beta_{k-1}^{\,(k)}$ (see 
Figure~\ref{fig:ahat}).  Next, we rotate the elements in the $(k,k-1)$ and
$(k,k)$ positions, represented by $\gamma_{k-1}$ (the $(k-1)$-st element of the
subdiagonal of $B$) and $\hat d_{k}^{\;(k+1)}$, respectively, in a
straightforward way in Line~\ref{alg:rotback:subdiag}; since we are only
interested in computing the upper triangular part of $\hat A^{(k)}$, we only
update the diagonal entry. Finally, we rotate the vector $\hat q^{\,(k+1)}$ in
Line~\ref{alg:rotback:q}.

The process of applying the rotation matrices on the right can be repeated,
until the upper triangular part of the matrix $\hat{\underline{A}}$,
approximating $\underline{A} = BU_n^* U_{n-1}^* \cdots U_2^*$, is obtained,
together with $\hat{\underline{q}}$, approximating $\underline q = U_2 U_3 \cdots U_n q$
(see Algorithm~\ref{alg:rotback}). In Section~\ref{sec:forerr},
Lemma~\ref{lem:rotback}, we prove that the forward errors in the upper
triangular part of $\hat{\underline{A}}$ and in the vector
$\hat{\underline{q}}$ are proportional to $\norm{B}_H \ur$ and $\norm{q}
\ur$, respectively.

\subsection{The QR Algorithms}

The elimination of the superdiagonal described in Algorithm~\ref{alg:elim},
followed by the rotation back to Hessenberg form described in
Algorithm~\ref{alg:rotback}, can be iterated to find the eigenvalues of $A +
pq^*$. Our unshifted explicit QR algorithm, based on this iteration, is
described in Algorithm~\ref{alg:qrnoshift}.  This unshifted QR algorithm can
be accelerated by the introduction of shifts; our explicit shifted QR
algorithm, with Wilkinson shifts, is described in
Algorithm~\ref{alg:qrshift}.

In Section~\ref{sec:forerr}, we show that the forward error of one iteration
of our QR algorithm (Algorithm~\ref{alg:elim} followed by
Algorithm~\ref{alg:rotback}) satisfies componentwise forward error bounds.
In Section~\ref{sec:backerr}, we use this result to prove that both our
explicit unshifted QR algorithm (Algorithm~\ref{alg:qrnoshift}) and our
shifted QR algorithm (Algorithm~\ref{alg:qrshift}) are componentwise
backward stable.

\begin{algorithm}
\caption{
  \label{alg:rotback}
(Rotating the matrix back to Hessenberg form) \textbf{Inputs:} This algorithm
accepts as inputs $n-1$ rotation matrices $Q_2, Q_2, \ldots, Q_n \in \C^{n
\times n}$, two vectors $d$ and $\gamma$ representing the diagonal and
subdiagonal, respectively, of an $n\times n$ complex matrix $B$, and two
vectors $p$ and $q$ of length $n$, where $B+pq^*$ is lower triangular.
\textbf{Outputs:} Letting $U_k \in \C^{n\times n}$, $k=2,3,\ldots,n$, denote
the matrices that rotate the $(k-1,k)$-plane by $Q_k$, this algorithm
returns as its outputs the vectors $\underline d$, $\underline \beta$, and
$\underline q$, where $\underline d$ and $\underline \beta$ represent the
diagonal and superdiagonal, respectively, of the matrix $B U_n^* U_{n-1}^*
\cdots U_2^*$, and $\underline q = U_2 U_3 \cdots U_n q$.}

\begin{algorithmic}[1]

\For{$k=n,n-1,\ldots,2$}

  \State 
    \label{alg:rotback:diagsup}
  Rotate the diagonal and the superdiagonal:
  \begin{align*}
    \left[
    \begin{array}{c}
      d_{k-1} \\
      \beta_{k-1}
    \end{array} \right]
    \gets
    \overline{Q_k}
    \left[
    \begin{array}{c}
      d_{k-1} \\
      -p_{k-1} q_k^*
    \end{array} \right].
  \end{align*}

  \State 
    \label{alg:rotback:subdiag}
  Rotate the subdiagonal and the diagonal:
  \begin{align*}
    d_k \gets \Bigl( \overline{Q_k}
    \left[
    \begin{array}{c}
      \gamma_{k-1} \\
      d_k 
    \end{array} \right] \Bigr)_2
  \end{align*}

  \State 
    \label{alg:rotback:q}
  Rotate $q$:
  \begin{math}
    \left[
    \begin{array}{c}
      q_{k-1} \\
      q_{k}
    \end{array} \right]
    \gets Q_k
    \left[
    \begin{array}{c}
      q_{k-1} \\
      q_{k}
    \end{array} \right]
  \end{math}

\EndFor

\State Set $\underline d \gets d$, $\underline \beta \gets \beta$, and
$\underline q \gets q$.

\end{algorithmic}
\end{algorithm}

\begin{algorithm}
\caption{
  \label{alg:qrnoshift}
(Unshifted explicit QR) \textbf{Inputs:} 
This algorithm accepts as inputs two vectors $d$ and $\beta$ representing
the diagonal and superdiagonal, respectively, of an $n\times n$ Hermitian
matrix $A$, as well as two vectors $p$ and $q$ of length $n$, where $A+pq^*$
is lower Hessenberg. It also accepts a tolerance $\epsilon > 0$, which
determines the accuracy the eigenvalues are computed to.
\textbf{Outputs:} It returns as its output the vector $\lambda$ of length
$n$ containing the eigenvalues of the matrix $A+pq^*$.}

\begin{algorithmic}[1]

\For{$i=1,2,\ldots,n-1$}
  \While{$\beta_i + p_i q_{i+1}^* \ge \epsilon$}
    \Comment{Check if $(A+pq^*)_{i,i+1}$ is close to zero}

    \State Perform one iteration of QR (one step of Algorithm~\ref{alg:elim}
    followed by one step of Algorithm~\ref{alg:rotback}) on the submatrix
    $(A+pq^*)_{i:n,i:n}$ defined by the vectors $d_{i:n}$, $\beta_{i:n-1}$,
    $p_{i:n}$, and $q_{i:n}$.

  \EndWhile
\EndFor

\State Set $\lambda \gets d$.

\end{algorithmic}
\end{algorithm}

\begin{algorithm}
\caption{
  \label{alg:qrshift}
(Shifted explicit QR) \textbf{Inputs:} 
This algorithm accepts as inputs two vectors $d$ and $\beta$ representing
the diagonal and superdiagonal, respectively, of an $n\times n$ Hermitian
matrix $A$, as well as two vectors $p$ and $q$ of length $n$, where $A+pq^*$
is lower Hessenberg. It also accepts a tolerance $\epsilon > 0$, which
determines the accuracy the eigenvalues are computed to.
\textbf{Outputs:} It returns as its output the vector $\lambda$ of length
$n$ containing the eigenvalues of the matrix $A+pq^*$.}

\begin{algorithmic}[1]

\For{$i=1,2,\ldots,n-1$}
  \label{alg:qrshift:outer}
  
  \State Set $\mu_\text{sum} \gets 0$.
  \While{$\beta_i + p_i q_{i+1}^* \ge \epsilon$}
    \Comment{Check if $(A+pq^*)_{i,i+1}$ is close to zero}

    \State Compute the eigenvalues $\mu_1$ and $\mu_2$ of the $2\times 2$ 
    submatrix
    \begin{math}
      \left[
        \begin{array}{cc}
        d_i + p_i q_i^* & \beta_i + p_i q_{i+1}^* \\
        \overline{\beta}_i + p_{i+1} q_i^* & d_{i+1} + p_{i+1} q_{i+1}^*
        \end{array}
      \right].
    \end{math}
      \Comment{This is just ${(A+pq^*)_{i:i+1,i:i+1}}$}

    \State Set $\mu$ to whichever of $\mu_1$ and $\mu_2$ is closest
    to $d_i + p_i q_i^*$.

    \State Set $\mu_\text{sum} \gets \mu_\text{sum} + \mu$.

    \State Set $d_{i:n} \gets d_{i:n} - \mu$.

    \State Perform one iteration of QR (one step of Algorithm~\ref{alg:elim}
    followed by one step of Algorithm~\ref{alg:rotback}) on the
    submatrix $(A+pq^*)_{i:n,i:n}$ defined by the vectors $d_{i:n}$,
    $\beta_{i:n-1}$, $p_{i:n}$, and $q_{i:n}$.
  \EndWhile

  \State Set $d_{i:n} \gets d_{i:n} + \mu_\text{sum}$.

\EndFor

\State Set $\lambda_i \gets d_i + p_i q_i^*$, for $i=1,2,\ldots,n$.

\end{algorithmic}
\end{algorithm}

\clearpage

\section{Componentwise Backward Stability}
  \label{sec:combackstab}

The principal results of this section are Theorems~\ref{thm:qrunshift}
and~\ref{thm:qrshift}, which state that our unshifted and shifted QR
algorithms, respectively, are componentwise backward stable. In
Section~\ref{sec:forerr}, we prove that the forward error of a single sweep
of our unshifted QR algorithm satisfies componentwise bounds.  In
Section~\ref{sec:backerr}, we use these bounds show componentwise backward
stability of our QR algorithms.

\subsection{Forward Error Analysis of a Single Sweep of $QR$}
  \label{sec:forerr}

Suppose that $A$ is Hermitian and $A+pq^*$ is lower Hessenberg.  In this
section, we prove in Theorem~\ref{thm:sweeperr} that the forward errors in
$A$, $p$, and $q$ of single sweep of our explicit $QR$ algorithm are
proportional to $\norm{A}\ur$, $\norm{p}\ur$, and $\norm{q}\ur$, respectively.

The following lemma bounds the forward error of Algorithm~\ref{alg:elim}
(the elimination of the superdiagonal).

\begin{lemma}
  \label{lem:elim}
Suppose that $A \in \C^{n \times n}$ is a Hermitian matrix, and that $p, q
\in \C^n$. Suppose further that $A+pq^*$ is lower Hessenberg, and let $d$
and $\beta$ denote the diagonal and superdiagonal of $A$, respectively.
Suppose that Algorithm~\ref{alg:elim} is carried out in floating point
arithmetic with $d$, $\beta$, $p$, and $q$ as inputs, and let $Q_2, Q_3,
\ldots, Q_n \in \SU(2)$ be the unitary matrices generated by an exact step
of Line~\ref{alg:elim:qk} of Algorithm~\ref{alg:elim} applied to the
computed vectors at that step.  Let $U_k \in \C^{n \times n}$,
$k=2,3,\ldots,n$, denote the matrices that rotate the $(k-1,k)$-plane by
$Q_k$, and define $U \in \C^{n \times n}$ by the formula $U=U_2 U_3 \cdots
U_n$.  Suppose finally that $\underline{\hat d}$, $\underline{\hat \gamma}$,
and $\underline{\hat p}$ are the outputs generated by
Algorithm~\ref{alg:elim}, and define the upper Hessenberg part of the matrix
${\hat B} \in \C^{n\times n}$ by the formula
  \begin{align}
{\hat b}_{i,j} = \left\{
  \begin{array}{ll}
  -\underline{\hat p}_i {q}_j^* & \text{if $j > i$}, \\
  \underline{\hat d}_i      & \text{if $j=i$}, \\
  \underline{\hat \gamma}_j & \text{if $j=i-1$}.
  \end{array} \right.
    \label{bhatdef}
  \end{align}
where $\hat b_{i,j}$ denotes the $(i,j)$-th entry of $\hat B$.
Let $B=UA$ and $\underline p=Up$. Then
  \begin{align}
\norm{ {\hat B} - B}_H \lesssim \norm{A} \ur
  \end{align}
and
  \begin{align}
\norm{ \underline{\hat p} - \underline p} \lesssim \norm{p} \ur,
  \end{align}
where $\norm{\cdot}_H$ denotes the square root of the sum of squares of the
entries in the upper Hessenberg part of its argument (see
Definition~\ref{defht}).

\end{lemma}

\begin{proof}
Suppose that $\hat d^{\,(k)}$, $\hat \gamma^{(k)}$, $\hat \beta^{\,(k)}$, 
$\hat p^{\,(k)}$, and $\hat{\tilde{q}}^{\,(k)}$ denote the computed vectors in
Algorithm~\ref{alg:elim} after the elimination of the superdiagonal elements
in the positions $(n-1,n), (n-2,n-1), \ldots, (k-1,k)$.  Suppose further
that the upper Hessenberg part of the matrix $\hat B^{(k)} \in \C^{n\times
n}$ is defined by the formula
  \begin{align}
\hat b^{\,(k)}_{i,j} = \left\{
  \begin{array}{ll}
  -\hat p_i^{\, (k)} q_j^* & \text{if $j > i+1$ or if $j=i+1$ and
  $j \ge k$}, \\
  \hat \beta_i^{\, (k)} & \text{if $j=i+1$ and $j<k$}, \\
  \hat{d}_i^{\;(k)} & \text{if $j=i$}, \\
  \hat \gamma_i^{\,(k)} & \text{if $j=i-1$},
  \end{array} \right.
  \end{align}
where $\hat b_{i,j}^{\,(k)}$ denotes the $(i,j)$-th entry of $\hat B^{(k)}$.
Clearly, $\hat{\underline{d}} = \hat{d}^{\,(2)}$,
$\hat{\underline{\gamma}} = \hat{\gamma}^{(2)}$, 
$\hat{\underline{p}} = \hat{p}^{\,(2)}$, 
and $\hat{B} = \hat{B}^{(2)}$.
Let $B^{(k)} = U_k U_{k+1} \cdots U_n A$ and $p^{(k)} = U_k U_{k+1} \cdots
U_n p$.   
We will prove that $\norm{ \hat B^{(k)} - B^{(k)} }_H \lesssim
\norm{A}\ur$ and $\norm{ \hat p^{\,(k)} - p^{(k)}} \lesssim \norm{p}\ur$, 
for each $k=n,n-1,\ldots,2$.

We begin by proving this statement for $k=n$. From Line~\ref{alg:elim:qk},
we have that the matrix $Q_n \in \SU(2)$ satisfies
  \begin{align}
\biggl( Q_n \left[
  \begin{array}{c}
  \beta_{n-1} + p_{n-1} q_n^* \\
  d_n + p_n q_n^*
  \end{array}\right]
\biggr)_1 = 0,
  \label{qndef}
  \end{align}
with the computed matrix $\hat Q_n$ satisfying $\norm{ \hat Q_n - Q_n }
\lesssim u$ by Lemma~\ref{lem:su2}. In Line~\ref{alg:elim:subxsubsub}, we
have
  \begin{align}
\hat \gamma_{n-2}^{(n)} = 
\fl \biggl( \hat Q_n
\left[
\begin{array}{c}
  \gamma_{n-2} \\
 -\tilde{q}_n p_{n-2}^* 
\end{array} \right] \biggr)_1.
  \end{align}
At this stage $\tilde q$ is still equal to $q$ and, according to
Lemma~\ref{lem:apr1},
$a_{n,n-2} = -q_n p_{n-2}^*$. By definition, $a_{n-1,n-2} = \gamma_{n-2}$.
Therefore, by Lemma~\ref{lem:rot}, we have that $\abs{ \hat \gamma_{n-2}^{(n)} 
- b_{n-1,n-2}^{(n)} } \lesssim \norm{A}u$, where $b_{i,j}^{(n)}$ denotes the
$(i,j)$-th entry of $B^{(n)}$. In Line~\ref{alg:elim:diagxsub}, we have
  \begin{align}
\left[
\begin{array}{c}
  \hat d_{n-1}^{\;(n)} \\
  \hat \gamma_{n-1}^{(n)}
\end{array} \right]
= \fl\biggl( \hat Q_n
\left[
\begin{array}{c}
  d_{n-1} \\
  \gamma_{n-1}
\end{array} \right] \biggr).
  \end{align}
Since $a_{n-1,n-1}=d_{n-1}$ and $a_{n,n-1}=\gamma_{n-1}$, by
Lemma~\ref{lem:rot}, we have that
$\abs{ \hat d_{n-1}^{\;(n)} 
- b_{n-1,n-1}^{(n)} } \lesssim \norm{A}u$
and
$\abs{ \hat \gamma_{n-1}^{(n)} 
- b_{n,n-1}^{(n)} } \lesssim \norm{A}u$.
In Line~\ref{alg:elim:supxdiag}, we have
  \begin{align}
\left[
\begin{array}{c}
  \hat \beta_{n-1}^{\,(n)} \\
  \hat d_{n}^{\;(n)}
\end{array} \right]
= \fl\biggl( \hat Q_n
\left[
\begin{array}{c}
  \beta_{n-1} \\
  d_n
\end{array} \right] \biggr).
  \end{align}
Since, by definition, $a_{n-1,n}=\beta_{n-1}$ and $a_{n,n}=d_n$, it follows from
Lemma~\ref{lem:rot} that $\abs{ \hat \beta_{n-1}^{\,(n)} - b_{n-1,n}^{(n)} } \lesssim
\Bigl(\sqrt{ \abs{\beta_{n-1}}^2 + \abs{d_n}^2 }\Bigr) \ur \le \norm{A} \ur$
and $\abs{ \hat d_{n}^{\;(n)} - b_{n,n}^{(n)} } \lesssim
\Bigl(\sqrt{ \abs{\beta_{n-1}}^2 + \abs{d_n}^2 }\Bigr) \ur \le \norm{A} \ur$.
In Line~\ref{alg:elim:p}, we have
  \begin{align}
\left[
\begin{array}{c}
  \hat p_{n-1}^{\,(n)\dagger} \\
  \hat p_{n}^{\,(n)}
\end{array} \right]
= \fl\biggl( \hat Q_n
\left[
\begin{array}{c}
  p_{n-1} \\
  p_n
\end{array} \right] \biggr),
  \end{align}
where $\hat p_{n-1}^{\,(n)\dagger}$ is a temporary value that will be
corrected later.
Once again, Lemma~\ref{lem:rot} tells us that $\abs{ \hat p_{n-1}^{\,(n)\dagger} -
p_{n-1}^{(n)}} \lesssim \Bigl(\sqrt{ \abs{p_{n-1}}^2 + \abs{p_n}^2} \Bigr) \ur \le
\norm{p}\ur$ and $\abs{ \hat p_{n}^{\,(n)} - p_{n}^{(n)}} \lesssim \Bigl(\sqrt{
\abs{p_{n-1}}^2 + \abs{p_n}^2} \Bigr) \ur \le \norm{p}\ur$.
Next, we observe that, by the definition of $Q_n$, we have that
  \begin{align}
b_{n-1,n}^{(n)} + p_{n-1}^{(n)} q_n^* = 0.
  \label{qnident}
  \end{align}
In Line~\ref{alg:elim:corr}, we apply a correction to 
$\hat{p}_{n-1}^{\,(n)\dagger}$, so that
  \begin{align}
\hat p_{n-1}^{\,(n)} = \left\{
  \begin{array}{ll}
-\hat \beta_{n-1}^{\,(n)} / q_n^* & \text{if $\abs{p_{n-1} q_n^*}^2 + 
    \abs{p_n q_n^*}^2 > \abs{\beta_{n-1}}^2 + \abs{d_n}^2$}, \\
\hat p_{n-1}^{\,(n)\dagger} & \text{if $\abs{p_{n-1} q_n^*}^2 + 
    \abs{p_n q_n^*}^2 \le \abs{\beta_{n-1}}^2 + \abs{d_n}^2$}. 
  \end{array} \right.
  \end{align}
From this, we see that, if $\abs{p_{n-1} q_n^*}^2 + \abs{p_n q_n^*}^2 >
\abs{\beta_{n-1}}^2 + \abs{d_n}^2$, then
  \begin{align}
&\abs{ \hat p_{n-1}^{\,(n)} - p_{n-1}^{(n)} }
= \abs{ -\hat \beta_{n-1}^{\,(n)} / q_n^* - p_{n-1}^{(n)}} \notag \\
& = \abs{ b_{n-1,n}^{(n)}/q_n^* -\hat \beta_{n-1}^{\,(n)} / q_n^* } \notag \\
& = \frac{1}{\abs{q_n^*}} \abs{ b_{n-1,n}^{(n)} -\hat \beta_{n-1}^{\,(n)} } \notag \\
& \lesssim \frac{\sqrt{ \abs{\beta_{n-1}}^2 + \abs{d_n}^2 }}{\abs{q_n^*}} \ur 
  \notag \\
& \le \Bigl(\sqrt{ \abs{p_{n-1}}^2 + \abs{p_n}^2}\Bigr) \ur \notag \\
& \le \norm{p} \ur.
  \end{align}
where the second equality is due to~(\ref{qnident}). Furthermore,
  \begin{align}
&\abs{ -\hat p_{n-1}^{\,(n)} q_n^* - b_{n-1,n}^{(n)}} 
= \abs{ \hat \beta_{n-1}^{\,(n)} - b_{n-1,n}^{(n)} } \notag \\
& \lesssim \Bigl(\sqrt{ \abs{\beta_{n-1}}^2 + \abs{d_n}^2 }\Bigr) \ur \notag \\
& \le \norm{A} \ur.
  \end{align}
If, on the other hand, $\abs{p_{n-1} q_n^*}^2 + \abs{p_n q_n^*}^2 \le 
\abs{\beta_{n-1}}^2 + \abs{d_n}^2$, then
  \begin{align}
& \abs{ \hat p_{n-1}^{\,(n)} - p_{n-1}^{(n)} }
= \abs{ \hat p_{n-1}^{\,(n)\dagger} - p_{n-1}^{(n)} } \notag \\
& \lesssim \sqrt{ \abs{p_{n-1}}^2 + \abs{p_n}^2 } \ur \notag \\
& \le \norm{p} \ur.
  \end{align}
Moreover,
  \begin{align}
& \abs{ -\hat p_{n-1}^{\,(n)} q_n^* - b_{n-1,n}^{(n)} }
= \abs{ -\hat p_{n-1}^{\,(n)\dagger} q_n^* - b_{n-1,n}^{(n)} } \notag \\
& = \abs{ -\hat p_{n-1}^{\,(n)\dagger} q_n^* + p_{n-1}^{(n)} q_n^* } \notag \\
& \lesssim \abs{q_n^*} \Bigl(\sqrt{ \abs{p_{n-1}}^2 + \abs{p_n}^2 }\Bigr)
  \ur  \notag \\
& \le \Bigl(\sqrt{ \abs{\beta_{n-1}}^2 + \abs{d_n}^2 }\Bigr) \ur \notag \\
& \le \norm{A} \ur.
\end{align}
where the second equality follows from~(\ref{qnident}). This completes the
proof that $\norm{ \hat B^{(n)} - B^{(n)} }_H \lesssim \norm{A}u$ and
$\norm{ \hat p^{\,(n)} - p^{(n)}} \lesssim \norm{p}u$.

Now, we will show that, if $\norm{ \hat B^{(k+1)} - B^{(k+1)} }_H \lesssim
\norm{A}u$ and $\norm{ \hat p^{\,(k+1)} - p^{(k+1)}} \lesssim \norm{p}u$, then
$\norm{ \hat B^{(k)} - B^{(k)} }_H \lesssim \norm{A}u$ and $\norm{ \hat
p^{\,(k)} - p^{(k)}} \lesssim \norm{p}u$. From Line~\ref{alg:elim:qk}, we have
that the matrix $Q_{k} \in \SU(2)$ satisfies
  \begin{align}
\biggl( Q_k \left[
  \begin{array}{c}
  \hat \beta_{k-1}^{\,(k+1)} + \hat p_{k-1}^{\,(k+1)} q_k^* \\
  \hat d_k^{\;(k+1)} + \hat p_k^{\,(k+1)} q_k^*
  \end{array}\right]
\biggr)_1 = 0,
  \label{qkdef}
  \end{align}
with the computed matrix $\hat Q_k$ satisfying $\norm{ \hat Q_k - Q_k }
\lesssim \ur$ by Lemma~\ref{lem:su2}.
In Line~\ref{alg:elim:subxsubsub}, we
have
  \begin{align}
\hat \gamma_{k-2}^{(k)} = 
\fl \biggl( \hat Q_k
\left[
\begin{array}{c}
  \hat \gamma_{k-2}^{(k+1)} \\
 -\hat{\tilde{q}}_k^{\,(k+1)} \hat p_{k-2}^{\,(k+1)*} 
\end{array} \right] \biggr)_1.
  \label{gamqk}
  \end{align}
We must first show that
  \begin{align}
\babs{ -\hat{\tilde{q}}_k^{\,(k+1)} \hat p_{k-2}^{\,(k+1)*} 
- b_{k,k-2}^{(k+1)} } \lesssim \norm{A} \ur.
  \end{align}
We begin by observing that
  \begin{align}
b_{k,k-2}^{(k+1)} = (B^{(k+1)})_{k,k-2}
= (B^{(k+1)} U_n^* U_{n-1}^* \cdots U_{k+1}^*)_{k,k-2},
  \label{bkident}
  \end{align}
since right-multiplication by $U_j^*$ only affects columns $j$ and $j-1$.
We now observe that
  \begin{align}
B^{(k+1)} U_n^* U_{n-1}^* \cdots U_{k+1}^* = 
U_{k+1} U_{k+2} \cdots U_n A U_n^* U_{n-1}^* \cdots U_{k+1}^*
  \end{align}
is Hermitian, so from~(\ref{bkident}) we have that
  \begin{align}
b_{k,k-2}^{(k+1)}
= \overline{(B^{(k+1)} U_n^* U_{n-1}^* \cdots U_{k+1}^*)}_{k-2,k}.
  \label{bksymm}
  \end{align}
By the induction hypothesis,
  \begin{align}
\hat b_{k-2,\ell}^{\,(k+1)} = -\hat p_{k-2}^{\,(k+1)} q_\ell^*
  \label{bkm2}
  \end{align}
and
  \begin{align}
\abs{ \hat b_{k-2,\ell}^{\,(k+1)} - b_{k-2,\ell}^{(k+1)} } \lesssim
\norm{A} \ur,
  \label{bkacc}
  \end{align}
for all $\ell = k,k+1,\ldots,n$.  Thus,
  \begin{align}
&\hspace*{-3em} 
\babs{ (-\hat p_{k-2}^{\,(k+1)} q^* U_n^* U_{n-1}^* \cdots U_{k+1}^*)_k
- (B^{(k+1)} U_n^* U_{n-1}^* \cdots U_{k+1}^*)_{k-2,k} } 
\lesssim \norm{A} \ur.
  \label{pkbkident}
  \end{align}
Combining~(\ref{pkbkident}) with~(\ref{bksymm}),
  \begin{align}
\babs{ -\hat p_{k-2}^{\,(k+1)} \tilde q_{k}^{\,(k+1)*} 
- \overline{ b_{k,k-2}^{(k+1)}}  } \lesssim \norm{A} \ur,
  \label{qtilbk}
  \end{align}
where $\tilde q^{\,(k+1)} = U_{k+1} U_{k+2} \cdots U_n q$.
From Line~\ref{alg:elim:qtil} we have
  \begin{align}
\hat{\tilde{q}}^{\,(k+1)} = \fl( \hat U_{k+1} \hat U_{k+2} \cdots
\hat U_n q ),
  \end{align}
and, by repeated application of Lemma~\ref{lem:rot},
  \begin{align}
\babs{ \hat{\tilde{q}}_{k}^{\,(k+1)} - {\tilde{q}}_{k}^{\,(k+1)} }
\lesssim \Bigl( \sqrt{ \abs{q_{k}}^2 + \abs{q_{k+1}}^2 + \cdots
+ \abs{q_n}^2 } \Bigr) \ur.
  \end{align}
Thus,
  \begin{align}
&\hspace*{-6em} \babs{ \hat p_{k-2}^{\,(k+1)} \hat{\tilde{q}}_{k}^{\,(k+1)*} 
- \hat p_{k-2}^{\,(k+1)} {\tilde{q}}_{k}^{\,(k+1)*} }
\lesssim \Bigl( \sqrt{ \abs{\hat p_{k-2}^{\,(k+1)}q_{k}^*}^2 + \abs{\hat
p_{k-2}^{\,(k+1)}q_{k+1}^*}^2 + \cdots + \abs{\hat p_{k-2}^{\,(k+1)}q_n^*}^2 }
\Bigr) \ur \notag \\
& = \Bigl( \sqrt{ \abs{\hat b_{k-2,k}^{\,(k+1)}}^2 + \abs{\hat
b_{k-2,k+1}^{\,(k+1)}}^2 + \cdots + \abs{\hat b_{k-2,n}^{\,(k+1)}}^2 }
\Bigr) \ur \notag \\
& \lesssim \norm{A} \ur.
  \label{qtilhat}
  \end{align}
where the first equality follows by~(\ref{bkm2}) and the second inequality 
by~(\ref{bkacc}).
Combining~(\ref{qtilbk}) and~(\ref{qtilhat}),
  \begin{align}
\babs{ -\hat p_{k-2}^{\,(k+1)} \hat{\tilde{q}}_{k}^{\,(k+1)*} 
- \overline{b_{k,k-2}^{(k+1)}} } \lesssim \norm{A} \ur,
  \end{align}
or, equivalently,
  \begin{align}
\babs{ -\hat{\tilde{q}}_k^{\,(k+1)} \hat p_{k-2}^{\,(k+1)*} 
- b_{k,k-2}^{(k+1)} } \lesssim \norm{A} \ur.
  \label{qtilbkp1}
  \end{align}
Finally, since, by the induction hypothesis,
  \begin{align}
\babs{ \hat \gamma_{k-2}^{(k+1)} - b_{k-1,k-2}^{(k+1)} } 
\lesssim \norm{A} \ur,
    \label{hgambkp1}
  \end{align}
we use~(\ref{qtilbkp1}) and~(\ref{hgambkp1}) and apply Lemma~\ref{lem:rot}
to~(\ref{gamqk}) to find that $\abs{ \hat \gamma_{k-2}^{(k)} -
b_{k-1,k-2}^{(k)} } \lesssim \norm{A} \ur$.
In Line~\ref{alg:elim:diagxsub}, we have
  \begin{align}
\left[
\begin{array}{c}
  \hat d_{k-1}^{\;(k)} \\
  \hat \gamma_{k-1}^{(k)}
\end{array} \right]
= \fl\biggl( \hat Q_k
\left[
\begin{array}{c}
  \hat d_{k-1}^{\;(k+1)} \\
  \hat \gamma_{k-1}^{(k+1)}
\end{array} \right] \biggr).
  \end{align}
By the induction hypothesis, $\abs{ \hat d_{k-1}^{\;(k+1)} -
b_{k-1,k-1}^{(k+1)} } \lesssim \norm{A}\ur$ and
$\abs{ \hat \gamma_{k-1}^{(k+1)} -
b_{k,k-1}^{(k+1)} } \lesssim \norm{A}\ur$.
Thus, another application of Lemma~\ref{lem:rot} shows that 
$\abs{ \hat d_{k-1}^{\;(k)} -
b_{k-1,k-1}^{(k)} } \lesssim \norm{A}\ur$ and
$\abs{ \hat \gamma_{k-1}^{(k)} - b_{k,k-1}^{(k)} } \lesssim \norm{A}\ur$.
In Line~\ref{alg:elim:supxdiag}, we have
  \begin{align}
\left[
\begin{array}{c}
  \hat \beta_{k-1}^{\,(k)} \\
  \hat d_{k}^{\;(k)}
\end{array} \right]
= \fl\biggl( \hat Q_k
\left[
\begin{array}{c}
  \hat \beta_{k-1}^{\,(k+1)} \\
  \hat d_k^{\;(k+1)}
\end{array} \right] \biggr).
  \end{align}
By the induction hypothesis, 
$\abs{\hat \beta_{k-1}^{\,(k+1)} - b_{k-1,k}^{(k+1)} } \lesssim \norm{A}\ur$ and
$\abs{\hat d_k^{\;(k+1)} - b_{k,k}^{(k+1)} } \lesssim \norm{A}\ur$,
so it follows from
Lemma~\ref{lem:rot} that $\abs{ \hat \beta_{k-1}^{\,(k)} - b_{k-1,k}^{(k)} }
\lesssim \norm{A} \ur$
and $\abs{ \hat d_{k}^{\;(k)} - b_{k,k}^{(k)} } \lesssim \norm{A} \ur$.
In Line~\ref{alg:elim:p}, we then have
  \begin{align}
\left[
\begin{array}{c}
  \hat p_{k-1}^{\,(k)\dagger} \\
  \hat p_{k}^{\,(k)}
\end{array} \right]
= \fl\biggl( \hat Q_k
\left[
\begin{array}{c}
  \hat p_{k-1}^{\,(k+1)} \\
  \hat p_k^{\,(k+1)}
\end{array} \right] \biggr),
  \end{align}
where $\hat p_{k-1}^{\,(k)\dagger}$ is a temporary value that will be
corrected later.
By the induction hypothesis, 
$\abs{\hat p_{k-1}^{\,(k+1)} - p_{k-1}^{(k+1)} } \lesssim \norm{p}\ur$ and
$\abs{\hat p_k^{\,(k+1)} - p_{k}^{(k+1)} } \lesssim \norm{p}\ur$,
so it follows from
Lemma~\ref{lem:rot} that $\abs{ \hat p_{k-1}^{\,(k)\dagger} - p_{k-1}^{(k)} }
\lesssim \norm{p} \ur$
and $\abs{ \hat p_{k}^{\,(k)} - p_{k}^{(k)} } \lesssim \norm{p} \ur$.
Define $\mathring \beta^{(k)}_{k-1}$  and
$\mathring d_k^{(k)}$ by the formula
  \begin{align}
\left[
\begin{array}{c}
  \mathring \beta_{k-1}^{(k)} \\
  \mathring d_{k}^{(k)}
\end{array} \right]
= Q_k
\left[
\begin{array}{c}
  \hat \beta_{k-1}^{\,(k+1)} \\
  \hat d_k^{\;(k+1)}
\end{array} \right],
  \end{align}
and define $\mathring p_{k-1}^{(k)}$ and $\mathring p_{k}^{(k)}$ by
  \begin{align}
\left[
\begin{array}{c}
  \mathring p_{k-1}^{(k)} \\
  \mathring p_{k}^{(k)}
\end{array} \right]
= Q_k
\left[
\begin{array}{c}
  \hat p_{k-1}^{\,(k+1)} \\
  \hat p_k^{\,(k+1)}
\end{array} \right].
    \label{pdotdef}
  \end{align}
Clearly,
  \begin{align}
\mathring \beta_{k-1}^{(k)} + \mathring p_{k-1}^{(k)} q_k^* = 0,
  \end{align}
by the definition of $Q_k$ (see~(\ref{qkdef})). Also, by
Lemma~\ref{lem:rot}, we have that $\abs{ \mathring \beta_{k-1}^{(k)} - \hat
\beta_{k-1}^{\,(k)} } \lesssim \Bigl(\sqrt{ \abs{\hat
\beta_{k-1}^{\,(k+1)}}^2 + \abs{\hat d_k^{\;(k+1)}}^2 }\Bigr) \ur \le
\norm{A} \ur$ and $\abs{ \mathring p_{k-1}^{(k)} - \hat
p_{k-1}^{\,(k)\dagger} } \lesssim \Bigl(\sqrt{ \abs{\hat
p_{k-1}^{\,(k+1)}}^2 + \abs{\hat p_k^{\,(k+1)}}^2 }\Bigr) \ur \le \norm{p}
\ur$.
In Line~\ref{alg:elim:corr}, we apply a correction to $\hat
p_{k-1}^{\,(k)\dagger}$, so that
  \begin{align}
&\hspace*{-3em}
\hat p_{k-1}^{\,(k)} = \left\{
  \begin{array}{ll}
-\hat \beta_{k-1}^{\,(k)} / q_k^* & \text{if $\abs{\hat 
  p_{k-1}^{\,(k+1)} q_k^*}^2 + 
  \abs{\hat p_k^{\,(k+1)} q_k^*}^2 > \abs{\hat \beta_{k-1}^{\,(k+1)}}^2 
  + \abs{\hat d_k^{\;(k+1)}}^2$}, \\
\hat p_{k-1}^{\,(k)\dagger} & \text{if $\abs{\hat 
  p_{k-1}^{\,(k+1)} q_k^*}^2 + 
  \abs{\hat p_k^{\,(k+1)} q_k^*}^2 \le \abs{\hat \beta_{k-1}^{\,(k+1)}}^2 
  + \abs{\hat d_k^{\;(k+1)}}^2$}.
  \end{array} \right.
  \end{align}
Thus, if $\abs{\hat p_{k-1}^{\,(k+1)} q_k^*}^2 + \abs{\hat p_k^{\,(k+1)}
q_k^*}^2 > \abs{\hat \beta_{k-1}^{\,(k+1)}}^2 + \abs{\hat d_k^{\;(k+1)}}^2$,
then
  \begin{align}
\babs{ \hat p_{k-1}^{\,(k)} - p_{k-1}^{(k)} } = 
\babs{ -\hat \beta_{k-1}^{\,(k)} / q_k^* - p_{k-1}^{(k)} }.
  \label{phmp1}
  \end{align}
We then observe that
  \begin{align}
\babs{ -\mathring \beta_{k-1}^{(k)} / q_k^* - p_{k-1}^{(k)} }
= \babs{ \mathring p_{k-1}^{(k)} - p_{k-1}^{(k)} } \lesssim \norm{p} \ur,
  \label{phmp2}
  \end{align}
and
  \begin{align}
&\babs{ -\mathring \beta_{k-1}^{(k)} / q_k^* + \hat\beta_{k-1}^{\,(k)} / q_k^* }
= \frac{1}{\abs{q_k^*}}  \babs{ \hat\beta_{k-1}^{\,(k)} 
  - \mathring \beta_{k-1}^{(k)} } \notag \\
& \lesssim \frac{\sqrt{ \abs{\hat
\beta_{k-1}^{\,(k+1)}}^2 + \abs{\hat d_k^{\;(k+1)}}^2 }}{\abs{q_k^*}} \ur  \notag \\
& \le \Bigl( \sqrt{\abs{\hat p_{k-1}^{\,(k+1)}}^2 
  + \abs{\hat p_k^{\,(k+1)}}^2} \Bigr)\ur \notag \\
& \le \norm{p}\ur.
  \label{phmp3}
  \end{align}
Finally, combining~(\ref{phmp1}),~(\ref{phmp2}), and~(\ref{phmp3}), we
find that $\abs{ \hat p_{k-1}^{\,(k)} - p_{k-1}^{(k)} } \lesssim
\norm{p}\ur$. Furthermore,
  \begin{align}
\babs{ -\hat p_{k-1}^{\,(k)} q_k^* - b_{k-1,k}^{(k)} }
= \babs{ \hat \beta_{k-1}^{\,(k)} - b_{k-1,k}^{(k)} } \lesssim \norm{A} \ur.
  \end{align}
If, conversely, $\abs{\hat p_{k-1}^{\,(k+1)} q_k^*}^2 + \abs{\hat
p_k^{\,(k+1)} q_k^*}^2 \le \abs{\hat \beta_{k-1}^{\,(k+1)}}^2 + \abs{\hat
d_k^{\;(k+1)}}^2$, then
  \begin{align}
\babs{ \hat p_{k-1}^{\,(k)} - p_{k-1}^{(k)} } = 
\babs{ \hat p_{k-1}^{\,(k)\dagger} - p_{k-1}^{(k)} } \lesssim \norm{p}\ur.
  \end{align}
Next, we observe that
  \begin{align}
\babs{ -\hat p_{k-1}^{\,(k)} q_k^* - b_{k-1,k}^{(k)} }
= \babs{ -\hat p_{k-1}^{\,(k)\dagger} q_k^* - b_{k-1,k}^{(k)} }.
  \label{pqmb1}
  \end{align}
Since
  \begin{align}
\babs{ -\mathring p_{k-1}^{(k)} q_k^* - b_{k-1,k}^{(k)} }
= \babs{ \mathring \beta_{k-1}^{(k)} - b_{k-1,k}^{(k)} }
\lesssim \norm{A}\ur,
  \label{pqmb2}
  \end{align}
and
  \begin{align}
& \babs{ -\mathring p_{k-1}^{(k)} q_k^* + \hat p_{k-1}^{\,(k)\dagger} q_k^* }
= \abs{q_k^*} 
  \babs{ -\mathring p_{k-1}^{(k)} + \hat p_{k-1}^{\,(k)\dagger} } \notag \\
&\lesssim \abs{q_k^*} \Bigl(\sqrt{ \abs{\hat
p_{k-1}^{\,(k+1)}}^2 + \abs{\hat p_k^{\,(k+1)}}^2 }\Bigr) \ur \notag \\
&\le \Bigl(\sqrt{ \abs{\hat \beta_{k-1}^{\,(k+1)}}^2 + \abs{\hat
d_k^{\;(k+1)}}^2 } \Bigr) \ur \notag \\
&\le \norm{A} \ur,
  \label{pqmb3}
  \end{align}
we combine~(\ref{pqmb1}),~(\ref{pqmb2}), and~(\ref{pqmb3}) to see that
$\abs{ -\hat p_{k-1}^{\,(k)} q_k^* - b_{k-1,k}^{(k)} }\lesssim \norm{A}\ur$.

Now all that's left is to show that $\abs{ -\hat p_{k-1}^{\,(k)} q_\ell^*
- b_{k-1,\ell}^{(k)} } \lesssim \norm{A}\ur$ and $\abs{ -\hat p_{k}^{\,(k)}
q_\ell^* - b_{k,\ell}^{(k)} } \lesssim \norm{A}\ur$, 
for all $\ell=k+1,k+2,\ldots, n$. By the induction hypothesis,
  \begin{align}
& \babs{ -\hat p_{k-1}^{\,(k+1)} q_\ell^* - b_{k-1,\ell}^{(k+1)}} 
  \lesssim \norm{A}\ur
    \label{bkl1}
  \end{align}
and
  \begin{align}
& \babs{ -\hat p_{k}^{\,(k+1)} q_\ell^* - b_{k,\ell}^{(k+1)}} 
  \lesssim \norm{A}\ur,
    \label{bkl2}
  \end{align}
for all $\ell = k+1,k+2,\ldots,n$. Multiplying~(\ref{pdotdef}) by
$q_\ell^*$, we have
  \begin{align}
\left[
\begin{array}{c}
  \mathring p_{k-1}^{(k)} q_\ell^* \\
  \mathring p_{k}^{(k)} q_\ell^*
\end{array} \right]
= Q_k
\left[
\begin{array}{c}
  \hat p_{k-1}^{\,(k+1)} q_\ell^* \\
  \hat p_k^{\,(k+1)} q_\ell^*
\end{array} \right],
    \label{pdotqldef}
  \end{align}
which, combined with~(\ref{bkl1}) and~(\ref{bkl2}), means that
  \begin{align}
& \babs{ -\mathring p_{k-1}^{(k)} q_\ell^* - b_{k-1,\ell}^{(k)}} 
  \lesssim \norm{A}\ur
    \label{pdotkm1eq}
  \end{align}
and
  \begin{align}
& \babs{ -\mathring p_{k}^{(k)} q_\ell^* - b_{k,\ell}^{(k)}} 
  \lesssim \norm{A}\ur,
    \label{pdotkeq}
  \end{align}
for all $\ell = k+1,k+2,\ldots,n$. 
It is not difficult to show (see~(\ref{phmp3})) that
  \begin{align}
\abs{ \mathring p_{k-1}^{(k)} - \hat
p_{k-1}^{\,(k)} } \lesssim \Bigl(\sqrt{ \abs{\hat
p_{k-1}^{\,(k+1)}}^2 + \abs{\hat p_k^{\,(k+1)}}^2 }\Bigr) \ur
  \end{align}
and
  \begin{align}
\abs{ \mathring p_{k}^{(k)} - \hat
p_{k}^{\,(k)} } \lesssim \Bigl(\sqrt{ \abs{\hat
p_{k-1}^{\,(k+1)}}^2 + \abs{\hat p_k^{\,(k+1)}}^2 }\Bigr) \ur,
  \end{align}
from which it follows that
  \begin{align}
&\abs{ \mathring p_{k-1}^{(k)} q_\ell^* - \hat
p_{k-1}^{\,(k)} q_\ell^* } \lesssim \Bigl(\sqrt{ \abs{\hat
p_{k-1}^{\,(k+1)} q_\ell^*}^2 + \abs{\hat p_k^{\,(k+1)} q_\ell^* }^2 }\Bigr) \ur 
\lesssim \norm{A} \ur
    \label{pdotq_phatq1}
  \end{align}
and
  \begin{align}
&\abs{ \mathring p_{k}^{(k)} q_\ell^* - \hat
p_{k}^{\,(k)} q_\ell^* } \lesssim \Bigl(\sqrt{ \abs{\hat
p_{k-1}^{\,(k+1)} q_\ell^*}^2 + \abs{\hat p_k^{\,(k+1)} q_\ell^* }^2 }\Bigr) \ur 
\lesssim \norm{A} \ur,
    \label{pdotq_phatq2}
  \end{align}
for all $\ell = k+1,k+2,\ldots,n$, where the second inequality follows
from~(\ref{bkl1}) and~(\ref{bkl2}).  Finally,
combining~(\ref{pdotkm1eq}),~(\ref{pdotkeq}),~(\ref{pdotq_phatq1}),
and~(\ref{pdotq_phatq2}), we find that
$\abs{ -\hat p_{k-1}^{\,(k)} q_\ell^*
- b_{k-1,\ell}^{(k)} } \lesssim \norm{A}\ur$ and $\abs{ -\hat p_{k}^{\,(k)}
q_\ell^* - b_{k,\ell}^{(k)} } \lesssim \norm{A}\ur$, 
for all $\ell=k+1,k+2,\ldots, n$, and we are done.

\end{proof}

The following lemma bounds the forward error of Algorithm~\ref{alg:rotback}
(the rotation back to Hessenberg form).

\begin{lemma}
  \label{lem:rotback}
Suppose that $B \in \C^{n \times n}$ and $p,q\in \C^n$. Suppose further that
$B+pq^*$ is lower triangular, and let $d$ and $\gamma$ denote the diagonal
and subdiagonal of $B$, respectively. Suppose that $Q_2, Q_3, \ldots, Q_n
\in \SU(2)$, and suppose that
Algorithm~\ref{alg:rotback} is carried out in floating point arithmetic, using
$d$, $\gamma$, $p$, $q$, and $Q_2,Q_3,\ldots,Q_n$ as inputs.
Suppose finally that $\hat{\underline{d}}$, $\hat{\underline{\beta}}$, and
$\hat{\underline{q}}$ are the outputs generated by
Algorithm~\ref{alg:rotback}, and define the upper triangular part of the matrix
$\hat{\underline A}\in \C^{n\times n}$ by the formula
  \begin{align}
{\hat{\underline{a}}}_{i,j} = \left\{
  \begin{array}{ll}
  -p_i \hat{\underline{q}}_j^{\,*} & \text{if $j > i+1$}, \\
  \underline{\hat \beta}_i & \text{if $j=i+1$}, \\
  \underline{\hat d}_i      & \text{if $j=i$},
  \end{array} \right.
  \end{align}
where $\hat{\underline{a}}_{i,j}$ denotes the $(i,j)$-th entry of
$\hat{\underline{A}}$.  Let $U_k \in \C^{n \times n}$, $k=2,3,\ldots,n$,
denote the matrices that rotate the $(k-1,k)$-plane by $Q_k$. Define $U \in
\C^{n \times n}$ by the formula $U=U_2 U_3 \cdots U_n$, and let
$\underline{A}=BU^*$ and $\underline q=U q$.
Then
  \begin{align}
\norm{ \hat{\underline{A}} - \underline{A} }_T \lesssim \norm{B}_H\ur
  \end{align}
and
  \begin{align}
\norm{ \hat{\underline{q}} - \underline{q} } \lesssim \norm{q}\ur,
  \end{align}
where $\norm{\cdot}_T$ denotes the square root of the sum of squares of the
entries in the upper triangular part of its argument and $\norm{\cdot}_H$
denotes the square root of the sum of squares of the upper
Hessenberg part (see Definition~\ref{defht}).

\end{lemma}

\begin{proof}
Suppose that $\hat d^{\,(k)}$, $\hat \beta^{\,(k)}$, and
$\hat q^{\,(k)}$ denote the computed vectors in
Algorithm~\ref{alg:rotback} after rotations in the positions
$(n-1,n), (n-2,n-1), \ldots, (k-1,k)$.  Suppose further
that the upper triangular part of the matrix $\hat{A}^{(k)} 
\in \C^{n\times n}$ is defined by the formula
  \begin{align}
{\hat{a}^{\,(k)}}_{i,j} = \left\{
  \begin{array}{ll}
  -p_i \hat{q}_j^{\,(k)*} & \text{if $j > i+1$ or if $j=i+1$ and
    $j<k$}, \\
  \hat{\beta}^{\,(k)}_i & \text{if $j=i+1$ and $j \ge k$}, \\
  \hat{d}^{\;(k)}_i      & \text{if $j=i$},
  \end{array} \right.
  \end{align}
where $\hat{a}_{i,j}^{\,(k)}$ denotes the $(i,j)$-th entry of
$\hat{A}^{(k)}$.  Clearly, $\hat{\underline{d}} = \hat{d}^{\,(2)}$,
$\hat{\underline{\beta}} = \hat{\beta}^{(2)}$, $\hat{\underline{q}} =
\hat{q}^{\,(2)}$, and $\hat{\underline{A}} = \hat{A}^{(2)}$.  Let $A^{(k)} =
B U_n^* U_{n-1}^* \cdots U_k^*$ and $q^{(k)} = U_k U_{k+1} \cdots U_n q$.
We will prove that $\norm{ \hat A^{(k)} - A^{(k)} }_T \lesssim \norm{B}_H
\ur$ and $\norm{ \hat q^{\,(k)} - q^{(k)}} \lesssim \norm{q}u$, for each
$k=n,n-1,\ldots,2$.

Define $\hat{d}^{\,(n+1)}=d$, $\hat{q}^{\,(n+1)}=q$, $q^{(n+1)}=q$, and
$A^{(n+1)}=B$. Obviously, $\hat{A}^{(n+1)} = A^{(n+1)}$
and $\hat{q}^{\,(n+1)} = q^{(n+1)}$, so the above statement
is true for $k=n+1$. We will prove it for the cases $k=n,n-1,\ldots, 2$
by induction.
In Line~\ref{alg:rotback:diagsup}, we have
  \begin{align}
\left[
\begin{array}{c}
  \hat d_{k-1}^{\;(k)} \\
  \hat \beta_{k-1}^{\,(k)}
\end{array} \right]
= \fl\biggl( \hat{\overline{Q_k}}
\left[
\begin{array}{c}
  \hat{d}_{k-1}^{\;(k+1)} \\
  -p_{k-1} \hat{q}_k^{\,(k+1)*}
\end{array} \right] \biggr).
  \end{align}
By the induction hypothesis, $\abs{ \hat d_{k-1}^{\;(k+1)} -
a_{k-1,k-1}^{(k+1)} } \lesssim \norm{B}_H\ur$ and
$\abs{ -p_{k-1} \hat{q}_k^{\,(k+1)*}
- a_{k-1,k}^{(k+1)} } \lesssim \norm{B}_H\ur$.
Since, by Lemma~\ref{lem:hesstri}, $\norm{A^{(k+1)}}_T \le \norm{B}_H$,
an application of Lemma~\ref{lem:rot} gives us
$\abs{ \hat d_{k-1}^{\;(k)} -
a_{k-1,k-1}^{(k)} } \lesssim \norm{B}_H\ur$ and
$\abs{ -p_{k-1} \hat{q}_k^{\,(k)*}
- a_{k-1,k}^{(k)} } \lesssim \norm{B}_H\ur$.
In Line~\ref{alg:rotback:subdiag}, we have
  \begin{align}
\hat d_{k}^{\;(k)} = 
\fl \biggl( \hat{\overline{Q_k}}
\left[
\begin{array}{c}
  \gamma_{k-1} \\
  \hat d_k^{\;(k+1)} 
\end{array} \right] \biggr)_2.
  \label{dkk}
  \end{align}
We first observe that
  \begin{align}
\gamma_{k-1} = b_{k,k-1} = (B U_n^* U_{n-1}^* \cdots U_{k+1}^*)_{k,k-1}
= a^{(k+1)}_{k,k-1},
  \end{align}
since right multiplication by $U_j^*$ only affects columns $j$ and $j-1$.
Since, by the induction hypothesis, $\abs{\hat{d}_k^{\;(k+1)} - a_{k,k}^{(k+1)}}
\lesssim \norm{B}_H \ur$, an application of Lemma~\ref{lem:rot} together
with the inequality $\norm{A^{(k+1)}}_T \le \norm{B}_H$ gives us
$\abs{\hat{d}_k^{\;(k)} - a_{k,k}^{(k)}} \lesssim \norm{B}_H \ur$.
In Line~\ref{alg:rotback:q}, 
  \begin{align}
\left[
\begin{array}{c}
  \hat q_{k-1}^{\,(k)} \\
  \hat q_{k}^{\,(k)}
\end{array} \right]
= \fl\biggl( \hat Q_k
\left[
\begin{array}{c}
  \hat q_{k-1}^{\,(k+1)} \\
  \hat q_k^{\,(k+1)}
\end{array} \right] \biggr).
  \end{align}
By the induction hypothesis, 
$\abs{\hat q_{k-1}^{\,(k+1)} - q_{k-1}^{(k+1)} } \lesssim \norm{q}\ur$ and
$\abs{\hat q_k^{\,(k+1)} - q_{k}^{(k+1)} } \lesssim \norm{q}\ur$, so it
follows from Lemma~\ref{lem:rot} that $\abs{ \hat q_{k-1}^{\,(k)} -
q_{k-1}^{(k)} } \lesssim \norm{q} \ur$ and $\abs{ \hat q_{k}^{\,(k)} -
q_{k}^{(k)} } \lesssim \norm{q} \ur$.

All that's left now is to prove that
$\abs{ -p_\ell \hat q_{k-1}^{\,(k)*}
- a_{\ell,k-1}^{(k)} } \lesssim \norm{B}_H\ur$ and 
$\abs{ -p_\ell \hat q_{k}^{\,(k)*}
- a_{\ell,k}^{(k)} } \lesssim \norm{B}_H\ur$, for all 
for all $\ell=1,2,\ldots,k-2$. By the induction hypothesis,
  \begin{align}
\abs{ -p_\ell \hat q_{k-1}^{\,(k+1)*}
- a_{\ell,k-1}^{(k+1)} } \lesssim \norm{B}_H\ur
    \label{akl1}
  \end{align}
and
  \begin{align}
\abs{ -p_\ell \hat q_{k}^{\,(k+1)*}
- a_{\ell,k}^{(k+1)} } \lesssim \norm{B}_H\ur,
    \label{akl2}
  \end{align}
for all $\ell = 1,2,\ldots,k-2$. 
Define $\mathring q_{k-1}^{(k)}$ and $\mathring q_{k}^{(k)}$ by
  \begin{align}
\left[
\begin{array}{c}
  \mathring q_{k-1}^{(k)} \\
  \mathring q_{k}^{(k)}
\end{array} \right]
= Q_k
\left[
\begin{array}{c}
  \hat q_{k-1}^{\,(k+1)} \\
  \hat q_k^{\,(k+1)}
\end{array} \right].
    \label{qdotdef}
  \end{align}
Multiplying~(\ref{qdotdef}) by
$p_\ell^*$, we have
  \begin{align}
\left[
\begin{array}{c}
  \mathring q_{k-1}^{(k)} p_\ell^* \\
  \mathring q_{k}^{(k)} p_\ell^*
\end{array} \right]
= Q_k
\left[
\begin{array}{c}
  \hat q_{k-1}^{\,(k+1)} p_\ell^* \\
  \hat q_k^{\,(k+1)} p_\ell^*
\end{array} \right],
    \label{qdotpldef}
  \end{align}
which, combined with~(\ref{akl1}) and~(\ref{akl2}) and the fact that
$\norm{A^{(k+1)}}_T \le \norm{B}_H$, means that
  \begin{align}
& \babs{ -p_\ell \mathring q_{k-1}^{(k)*} - a_{\ell,k-1}^{(k)}} 
  \lesssim \norm{B}_H\ur
    \label{qdotkm1eq}
  \end{align}
and
  \begin{align}
& \babs{ -p_\ell \mathring q_{k}^{(k)*} - a_{\ell,k}^{(k)}} 
  \lesssim \norm{B}_H\ur,
    \label{qdotkeq}
  \end{align}
for all $\ell = 1,2,\ldots,k-2$. 
By Lemma~\ref{lem:rot},
  \begin{align}
\abs{ \mathring q_{k-1}^{(k)} - \hat
q_{k-1}^{\,(k)} } \lesssim \Bigl(\sqrt{ \abs{\hat
q_{k-1}^{\,(k+1)}}^2 + \abs{\hat q_k^{\,(k+1)}}^2 }\Bigr) \ur
  \end{align}
and
  \begin{align}
\abs{ \mathring q_{k}^{(k)} - \hat
q_{k}^{\,(k)} } \lesssim \Bigl(\sqrt{ \abs{\hat
q_{k-1}^{\,(k+1)}}^2 + \abs{\hat q_k^{\,(k+1)}}^2 }\Bigr) \ur,
  \end{align}
from which it follows that
  \begin{align}
&\abs{ \mathring q_{k-1}^{(k)} p_\ell^* - \hat
q_{k-1}^{\,(k)} p_\ell^* } \lesssim \Bigl(\sqrt{ \abs{\hat
q_{k-1}^{\,(k+1)} p_\ell^*}^2 + \abs{\hat q_k^{\,(k+1)} p_\ell^* }^2 }\Bigr) \ur 
\lesssim \norm{B}_H \ur
    \label{qdotp_qhatp1}
  \end{align}
and
  \begin{align}
&\abs{ \mathring q_{k}^{(k)} p_\ell^* - \hat
q_{k}^{\,(k)} p_\ell^* } \lesssim \Bigl(\sqrt{ \abs{\hat
q_{k-1}^{\,(k+1)} p_\ell^*}^2 + \abs{\hat q_k^{\,(k+1)} p_\ell^* }^2 }\Bigr) \ur 
\lesssim \norm{B}_H \ur,
    \label{qdotp_qhatp2}
  \end{align}
for all $\ell = 1,2,\ldots,k-2$, where the second inequality follows
from~(\ref{akl1}) and~(\ref{akl2}) and the inequality
$\norm{A^{(k+1)}}_T \le \norm{B}_H$.  Finally,
combining~(\ref{qdotkm1eq}),~(\ref{qdotkeq}),~(\ref{qdotp_qhatp1}),
and~(\ref{qdotp_qhatp2}), we find that
$\abs{ -p_\ell \hat q_{k-1}^{\,(k)*}
- a_{\ell,k-1}^{(k)} } \lesssim \norm{B}_H\ur$ and 
$\abs{ -p_\ell \hat q_{k}^{\,(k)*}
- a_{\ell,k}^{(k)} } \lesssim \norm{B}_H\ur$, for all 
for all $\ell=1,2,\ldots,k-2$, and we are done.

\end{proof}

The following theorem bounds the forward errors of a full sweep of our $QR$
algorithm, and is the principal result of this subsection.

\begin{theorem}
  \label{thm:sweeperr}
Suppose that $A\in \C^{n\times n}$ is a Hermitian matrix, that 
$p, q\in \C^n$, and that $A+pq^*$ is lower Hessenberg. Let $d$ and $\beta$
denote the diagonal and superdiagonal of $A$, respectively.
Suppose that Algorithm~\ref{alg:elim} is carried out in floating point arithmetic,
and let $Q_2, Q_3, \ldots, Q_n \in \SU(2)$ be the
unitary matrices generated by an exact step of Line~\ref{alg:elim:qk} of
Algorithm~\ref{alg:elim} applied to the computed vectors at that step.  Let
$U_k \in \C^{n \times n}$, $k=2,3,\ldots,n$, denote the matrices that rotate
the $(k-1,k)$-plane by $Q_k$, and define $U \in \C^{n \times n}$ by the
formula $U=U_2 U_3 \cdots U_n$. Suppose that Algorithm~\ref{alg:rotback}
is then carried out in floating point arithmetic, using the outputs of 
Algorithm~\ref{alg:elim} as inputs. Suppose finally that $\hat{\underline{p}}$
is an output of Algorithm~\ref{alg:elim} and $\hat{\underline{q}}$, 
$\hat{\underline{d}}$, and $\hat{\underline{\beta}}$ are all outputs of
Algorithm~\ref{alg:rotback}, and define the matrix $\hat{\underline{A}}$
by the formula
  \begin{align}
\hat{\underline{a}}_{i,j} = \left\{
  \begin{array}{ll}
  -\hat{\underline{p}}_i \hat{\underline{q}}_j^*  & \text{if $j > i+1$} \\
  \hat{\underline{\beta}}_i  & \text{if $j = i+1$} \\
  \hat{\underline{d}}_i  & \text{if $j = i$} \\
  \overline{(\hat{\underline{\beta}}_j)} & \text{if $j = i-1$} \\
  -\hat{\underline{q}}_j \hat{\underline{p}}_i^*  & \text{if $j < i-1$}
  \end{array} \right.
    \label{ahatdef}
  \end{align}
where $\hat{\underline{a}}_{i,j}$ denotes the $(i,j)$-th entry of
$\hat{\underline{A}}$.  Let $\underline{A} = UAU^*$, $\underline{p}=Up$, and
$\underline{q}=Uq$. Then
  \begin{align}
\norm{\hat{\underline{A}} - \underline{A}} \lesssim \norm{A}\ur,
  \end{align}
  \begin{align}
\norm{\hat{\underline{p}} - \underline{p}} \lesssim \norm{p}\ur,
  \end{align}
and
  \begin{align}
\norm{\hat{\underline{q}} - \underline{q}} \lesssim \norm{q}\ur.
  \end{align}

\end{theorem}

\begin{proof}
Suppose that $\hat B$ (defined by~(\ref{bhatdef})), $\hat{\underline{p}}$,
and $\hat Q_2, \hat Q_3, \ldots, \hat Q_n \in \C^{n\times n}$ are outputs of
Algorithm~\ref{alg:elim}. Let $B=UA$ and $\underline p=Up$. By
Lemma~\ref{lem:elim}, 
  \begin{align}
\norm{ {\hat B} - B}_H \lesssim \norm{A} \ur
  \label{bhatineq}
  \end{align}
and
  \begin{align}
\norm{ \underline{\hat p} - \underline p} \lesssim \norm{p} \ur,
  \end{align}
where $\norm{\cdot}_H$ denotes the square root of the sum of squares of the
entries in the upper Hessenberg part of its argument (see
Definition~\ref{defht}).  Now suppose that $\hat B$, $\underline{\hat p}$,
$q$, and $\hat Q_2, \hat Q_3, \ldots, \hat Q_n \in \C^{n\times n}$ are used
as inputs to Algorithm~\ref{alg:rotback}.  Let $\hat U_k \in \C^{n \times
n}$, $k=2,3,\ldots,n$, denote the matrices that rotate the $(k-1,k)$-plane
by $\hat Q_k$, and define $\hat U \in \C^{n \times n}$ by the formula $\hat
U=\hat U_2 \hat U_3 \cdots \hat U_n$.  Let $\undertilde A = \hat B \hat U^*$
and $\undertilde q= \hat Uq$, and observe that the upper triangular part of
$\undertilde A$ is well-defined due to Lemma~\ref{lem:hesstri}. By
Lemma~\ref{lem:rotback} we have that
  \begin{align}
\norm{ {\hat{\underline{A}}} - \undertilde A}_T \lesssim \norm{\hat B}_H \ur
    \label{ahatineq}
  \end{align}
and
  \begin{align}
\norm{ \underline{\hat q} - \undertilde q} \lesssim \norm{q} \ur,
    \label{qhatineq}
  \end{align}
where $\norm{\cdot}_T$ denotes the square root of the sum of squares of the
entries in the upper triangular part of its argument and $\norm{\cdot}_H$
denotes the square root of the sum of squares of the upper Hessenberg part
(see Definition~\ref{defht}). Let $\underline A = BU^* = UAU^*$ and let
$\underline q=Uq$. We observe that
  \begin{align}
&\norm{ \undertilde A - \underline A }_T = 
\norm{ \hat B \hat U^* - BU^* }_T  \notag \\
&\le \norm{ \hat B \hat U^* - B \hat U^*}_T + 
  \norm{ B\hat U^* - BU^*}_T \notag \\
&= \norm{ (\hat B - B) \hat U^*}_T + 
  \norm{ B(\hat U^* - U^*)}_T \notag \\
&\lesssim \norm{A} \ur,
    \label{atilineq}
  \end{align}
where the last inequality follows from~(\ref{bhatineq}) and the fact that 
$\norm{ \hat U - U } \lesssim \ur$.
Since, clearly, $\norm{\hat B}_H\ur \lesssim \norm{A}\ur$, we 
combine~(\ref{ahatineq}) and~(\ref{atilineq}) to get
  \begin{align}
\norm{\hat{\underline{A}} - \underline{A}}_T \lesssim \norm{A}\ur.
  \end{align}
Now we observe that, since both $\underline A = UAU^*$ and $\underline{\hat{A}}$
are Hermitian,
  \begin{align}
\norm{\hat{\underline{A}} - \underline{A}} \lesssim \norm{A}\ur.
  \end{align}
Next, we observe that
  \begin{align}
&\norm{ \undertilde q - \underline q} = \norm{ \hat Uq - U q} \notag \\
&= \norm{ (\hat U - U)q } \notag \\
&\lesssim \norm{q}\ur, 
    \label{qtilineq}
  \end{align}
so, combining~(\ref{qhatineq}) and~(\ref{qtilineq}), we have
  \begin{align}
\norm{\hat{\underline{q}} - \underline{q}} \lesssim \norm{q}\ur,
  \end{align}
and we are done.

\end{proof}

\subsection{Backward Error Analysis of the $QR$ Algorithms}
  \label{sec:backerr}

Suppose that $A$ is Hermitian and $A+pq^*$ is lower Hessenberg.  In this
section, we prove in Theorems~\ref{thm:qrunshift} and~\ref{thm:qrshift} that
the backward errors in $A$, $p$, and $q$ of both our explicit unshifted $QR$
algorithm (see Algorithm~\ref{alg:qrnoshift}) and explicit shifted $QR$
algorithm (see Algorithm~\ref{alg:qrshift}) are proportional to
$\norm{A}\ur$, $\norm{p}\ur$, and $\norm{q}\ur$, respectively.

The following lemma states that a single iteration of our QR algorithm is
componentwise backward stable.

\begin{lemma}
  \label{lem:backstab1}
Suppose that $A\in \C^{n\times n}$ is a Hermitian matrix, that 
$p, q\in \C^n$, and that $A+pq^*$ is lower Hessenberg. Let $d$ and $\beta$
denote the diagonal and superdiagonal of $A$, respectively.
Suppose that a single iteration of our QR algorithm
(Algorithm~\ref{alg:elim} followed by Algorithm~\ref{alg:rotback}) is
carried out in floating point arithmetic, and let 
$\hat{\underline{p}}$, $\hat{\underline{q}}$, 
$\hat{\underline{d}}$, and $\hat{\underline{\beta}}$ denote the outputs
of the algorithm.
Define the matrix $\hat{\underline{A}}$ by the formula~(\ref{ahatdef}). Then
there exists a unitary matrix $U \in \C^{n\times n}$, a matrix $\delta A \in
\C^{n\times n}$, and vectors $\delta p,\delta q \in \C^n$, such that
  \begin{align}
\hat{\underline{A}} = U(A + \delta A)U^*,
  \end{align}
  \begin{align}
\hat{\underline{p}} = U(p + \delta p),
  \end{align}
and
  \begin{align}
\hat{\underline{q}} = U(q + \delta q),
  \end{align}
where $\norm{\delta A}\lesssim \norm{A}\ur$, $\norm{\delta p} \lesssim
\norm{p}\ur$, and $\norm{\delta q} \lesssim \norm{q}\ur$.

\end{lemma}

\begin{proof}
Let $U\in \C^{n\times n}$ be the unitary matrix defined in the statement 
of Theorem~\ref{thm:sweeperr}, and let $\underline{A} = UAU^*$,
$\underline{p}=Up$, and $\underline{q}=Uq$. By Theorem~\ref{thm:sweeperr},
  \begin{align}
\hat{\underline{A}} = \underline{A} + \underline{\delta A},
  \end{align}
where $\norm{\underline{\delta A}} \lesssim \norm{A}\ur$. Thus,
  \begin{align}
\hat{\underline{A}} = UAU^* + \underline{\delta A}
= U(A+\delta A) U^*,
  \end{align}
where $\delta A= U^* \underline{\delta A} U$. Since $U$ is unitary,
clearly $\norm{\delta A} \lesssim \norm{A}\ur$.
Likewise, by Theorem~\ref{thm:sweeperr},
  \begin{align}
\hat{\underline{p}} = \underline{p} + \underline{\delta p},
  \end{align}
where $\norm{\underline{\delta p}} \lesssim \norm{p}\ur$, so
  \begin{align}
\hat{\underline{p}} =U(p+\delta p),
  \end{align}
where $\delta p=U^* \underline{\delta p}$ and $\norm{\delta p}\lesssim
\norm{p}\ur$. Similarly,
  \begin{align}
\hat{\underline{q}} =U(q+\delta q),
  \end{align}
where $\norm{\delta q}\lesssim \norm{q}\ur$.

\end{proof}

The following lemma states that repeated iterations of our QR algorithm are
componentwise backward stable.

\begin{lemma}
  \label{lem:backstabiter}
Suppose that $A\in \C^{n\times n}$ is a Hermitian matrix, that 
$p, q\in \C^n$, and that $A+pq^*$ is lower Hessenberg. Let $d$ and $\beta$
denote the diagonal and superdiagonal of $A$, respectively.
Suppose that $k$ iterations of our QR algorithm
(Algorithm~\ref{alg:elim} followed by Algorithm~\ref{alg:rotback}) are
carried out in floating point arithmetic, and let 
$\hat{p}^{\,(k)}$, $\hat{q}^{\,(k)}$, 
$\hat{d}^{\,(k)}$, and $\hat{\beta}^{(k)}$ denote the outputs
of the algorithm.
Define the matrix $\hat{A}^{(k)}$ by the formula~(\ref{ahatdef}), making
the obvious substitutions. Then
there exists a unitary matrix $U \in \C^{n\times n}$, a matrix $\delta
A \in \C^{n\times n}$, and vectors $\delta p,\delta q \in \C^n$, such that
  \begin{align}
\hat{A}^{(k)} = U(A + \delta A)U^*,
  \end{align}
  \begin{align}
\hat{p}^{\,(k)} = U(p + \delta p),
  \end{align}
and
  \begin{align}
\hat{q}^{\,(k)} = U(q + \delta q),
  \end{align}
where $\norm{\delta A}\lesssim \norm{A}\ur$, $\norm{\delta p} \lesssim
\norm{p}\ur$, and $\norm{\delta q} \lesssim \norm{q}\ur$.

\end{lemma}

\begin{proof}
We will prove this statement only for the matrix $\hat A^{(k)}$, since the proofs
for $\hat p^{\,(k)}$ and $\hat q^{\,(k)}$ are essentially identical.
By repeated application of Lemma~\ref{lem:backstab1}, we know that there exist unitary
matrices $U^{(1)}, U^{(2)}, \ldots, U^{(k)}$ and matrices $\delta A^{(0)},
\delta \hat A^{(1)}, \delta \hat A^{(2)},\ldots, \delta \hat A^{(k-1)}$
such that
  \begin{align}
&\hat A^{(k)} = U^{(k)} (\hat A^{(k-1)} + \delta \hat A^{(k-1)}) U^{(k)*}, 
  \label{ahatkstab} \\
&\hat A^{(k-1)} = U^{(k-1)} (\hat A^{(k-2)} + \delta \hat A^{(k-2)}) U^{(k-1)*}, \\
&\hspace*{7em} \vdots \notag \\
&\hat A^{(2)} = U^{(2)} (\hat A^{(1)} + \delta \hat A^{(1)}) U^{(2)*}, \\
&\hat A^{(1)} = U^{(1)} (A + \delta A^{(0)}) U^{(1)*},
  \label{ahat1stab}
  \end{align}
where $\norm{\delta A^{(0)}} \lesssim \norm{A}\ur$ and $\norm{\delta \hat
A^{(\ell)}} \lesssim \norm{A}\ur$, for $\ell=1,2,\ldots,k-1$.
Combining~(\ref{ahatkstab})--(\ref{ahat1stab}) and expanding, we find that
  \begin{align}
&\hat A^{(k)} = U^{(k)} U^{(k-1)} \cdots U^{(1)} A U^{(1)*} U^{(2)*} \cdots 
U^{(k)*} + U^{(k)} \delta \hat A^{(k-1)} U^{(k)*} \notag \\
& +  U^{(k)} U^{(k-1)} \delta \hat A^{(k-2)} U^{(k-1)*} U^{(k)*} + \cdots \notag \\
& + U^{(k)} U^{(k-1)} \cdots U^{(1)} \delta A^{(0)} U^{(1)*} 
  \cdots U^{(k-1)*} U^{(k)*}.
  \end{align}
Letting $U=U^{(k)}U^{(k-1)} \cdots U^{(1)}$, this becomes
  \begin{align}
&\hat A^{(k)} = U A U^* + U^{(k)} \delta \hat A^{(k-1)} U^{(k)*} \notag \\
& +  U^{(k)} U^{(k-1)} \delta \hat A^{(k-2)} U^{(k-1)*} U^{(k)*} + \cdots \notag \\
& + U^{(k)} U^{(k-1)} \cdots U^{(1)} \delta A^{(0)} U^{(1)*} 
  \cdots U^{(k-1)*} U^{(k)*}.
    \label{ahatkeq}
  \end{align}
Suppose now that the matrix $\delta A$ is defined by
  \begin{align}
&\delta A = U^* \bigl( U^{(k)} \delta \hat A^{(k-1)} U^{(k)*} 
 +  U^{(k)} U^{(k-1)} \delta \hat A^{(k-2)} U^{(k-1)*} U^{(k)*} + \cdots \notag \\
& + U^{(k)} U^{(k-1)} \cdots U^{(1)} \delta A^{(0)} U^{(1)*} 
  \cdots U^{(k-1)*} U^{(k)*} \bigr) U.
    \label{deladef}
  \end{align}
Clearly, $\norm{\delta A} \lesssim \norm{A}\ur$. Combining~(\ref{ahatkeq})
and~(\ref{deladef}), we have
  \begin{align}
\hat A^{(k)} = U(A+\delta A)U^*,
  \end{align}
and we are done.

\end{proof}

The following theorem states that our explicit unshifted QR algorithm is
componentwise backward stable.

\begin{theorem}[Explicit unshifted QR]
  \label{thm:qrunshift}
Suppose that $A\in \C^{n\times n}$ is a Hermitian matrix, that 
$p, q\in \C^n$, and that $A+pq^*$ is lower Hessenberg. Let $d$ and $\beta$
denote the diagonal and superdiagonal of $A$, respectively.
Suppose that Algorithm~\ref{alg:qrnoshift} is
carried out in floating point arithmetic with $\epsilon \lesssim \norm{A}\ur$,
and let $\hat \lambda_1, \hat \lambda_2, \ldots, \hat \lambda_n$ denote
the outputs.  Then there exist a matrix $\delta A \in \C^{n\times n}$ and
vectors $\delta p,\delta q \in \C^n$ such that $\hat \lambda_1, \hat
\lambda_2, \ldots, \hat \lambda_n$  are the exact eigenvalues of the matrix
  \begin{align}
(A + \delta A) + (p + \delta p)(q + \delta q)^*,
  \end{align}
where $\norm{\delta A}\lesssim \norm{A}\ur$, $\norm{\delta p} \lesssim
\norm{p}\ur$, and $\norm{\delta q} \lesssim \norm{q}\ur$.

\end{theorem}

\begin{proof}
Suppose that we carry out QR iterations until the entry in the
$(1,2)$~position is less than $\epsilon$ in absolute value. Let $\hat
d^{\,(1)}$, $\hat \beta^{(1)\dagger}$, $\hat p^{\,(1)}$, and $\hat
q^{\,(1)}$ denote the resulting outputs, and let $\hat A^{(1)\dagger}$ be
the resulting matrix, defined by formula~(\ref{ahatdef}) (making the obvious
substitutions). By Lemma~\ref{lem:backstabiter}, there exist a unitary
matrix $U^{(1)}\in \C^{n\times n}$, a matrix $\delta A^{(0)\dagger} \in
\C^{n\times n}$, and vectors $\delta p^{(0)}, \delta q^{(0)} \in \C^n$ such
that
  \begin{align}
\hat A^{(1)\dagger} = U^{(1)} (A+\delta A^{(0)\dagger}) U^{(1)*},
    \label{ahatdagdef}
  \end{align}
  \begin{align}
\hat p^{\,(1)} = U^{(1)} (p+\delta p^{(0)}),
    \label{phat1def}
  \end{align}
and
  \begin{align}
\hat q^{\,(1)} = U^{(1)} (q+\delta q^{(0)}),
    \label{qhat1def}
  \end{align}
where $\norm{\delta A^{(0)\dagger}}\lesssim \norm{A}\ur$, $\norm{\delta p^{(0)}}
\lesssim \norm{p}\ur$, and $\norm{\delta q^{(0)}} \lesssim \norm{q}\ur$.
Let $\hat A^{(1)}$ be equal to $\hat A^{(1)\dagger}$, except that the
entry in the $(1,2)$~position of $\hat A^{(1)}$ is equal to $-\hat
p_1^{\,(1)} \hat q_2^{\,(1)*}$, so that $\hat A_{1,2}^{(1)} +\hat
p_1^{\,(1)} \hat q_2^{\,(1)*} = 0$.  Since $\abs{\hat A_{1,2}^{(1)\dagger}
+\hat p_1^{\,(1)} \hat q_2^{\,(1)*}} < \epsilon$, we have
  \begin{align}
\norm{ \hat A^{(1)\dagger} - \hat A^{(1)} }_F < \epsilon,
  \end{align}
where $\norm{\cdot}_F$ denotes the Frobenius norm, and since $\epsilon \lesssim
\norm{A}\ur$,
  \begin{align}
\norm{ \hat A^{(1)\dagger} - \hat A^{(1)} } \lesssim \norm{A}\ur.
    \label{ahatdagineq}
  \end{align}
Letting
  \begin{align}
\delta A^{(0)} = \delta A^{(0)\dagger} + U^{(1)*} (\hat A^{(1)} - \hat
A^{(1)\dagger}) U^{(1)}
    \label{dela0def}
  \end{align}
and combining~(\ref{ahatdagdef}) and~(\ref{dela0def}), we have
  \begin{align}
\hat A^{(1)} = U^{(1)} (A+\delta A^{(0)}) U^{(1)*},
    \label{ahat1def}
  \end{align}
where $\norm{\delta A^{(0)}}\lesssim \norm{A}\ur$ by~(\ref{ahatdagineq}).
Clearly, since $\hat A^{\,(1)}_{1,2} + \hat p_1^{\,(1)} \hat q_2^{\,(1)*} =
0$ and $\hat A^{(1)} + \hat p^{\,(1)} \hat q^{\,(1)*}$ is lower Hessenberg,
$\hat \lambda_1 = \hat A_{1,1}^{\,(1)} + \hat p_1^{\,(1)} \hat q_1^{\,(1)*}$
is an eigenvalue of $\hat A^{(1)} + \hat p^{\,(1)} \hat q^{\,(1)*}$.
Thus, from~(\ref{phat1def}),~(\ref{qhat1def}), and~(\ref{ahat1def}), we see
that $\hat \lambda_1$ is an eigenvalue of $(A+\delta A^{(0)}) + (p+\delta
p^{(0)})(q+\delta q^{(0)})^*$.

Now suppose that we deflate the matrix, and perform QR iterations on the submatrix
$\hat A^{\,(1)}_{2:n,2:n} + \hat p_{2:n}^{\,(1)} \hat q_{2:n}^{\,(1)*}$, until the
entry in the $(1,2)$~position of the deflated matrix is less than
$\epsilon$.  Let $\hat{\undertilde{d}}^{(2)} \in \C^{n-1}$,
$\hat{\undertilde{\beta}}^{(2)\dagger} \in \C^{n-2}$, $\hat{\undertilde{p}}^{(2)}
\in \C^{n-1}$, and $\hat{\undertilde{q}}^{(2)} \in \C^{n-1}$ denote the
resulting outputs, and let $\hat{\undertilde{A}}^{(2)\dagger} \in
\C^{(n-1)\times (n-1)}$ be the
resulting matrix, defined by formula~(\ref{ahatdef}) (again making the
obvious substitutions). By Lemma~\ref{lem:backstabiter}, there 
exist a unitary
matrix $\undertilde{U}^{(2)}\in \C^{(n-1)\times (n-1)}$, a matrix $\delta
\undertilde{\hat{A}}^{(1)\dagger} \in \C^{(n-1)\times (n-1)}$, 
and vectors $\delta \hat{\undertilde{p}}^{(1)}, \delta
\hat{\undertilde{q}}^{(1)} \in \C^{n-1}$ such
that
  \begin{align}
\hat{\undertilde{A}}^{(2)\dagger} = \undertilde{U}^{(2)} 
(\hat A^{(1)}_{2:n,2:n} + \delta\hat{\undertilde{A}}^{(1)\dagger}) 
\undertilde{U}^{(2)*},
  \end{align}
  \begin{align}
\hat{\undertilde{p}}^{\,(2)} = U^{(2)} (\hat p_{2:n}^{\,(1)} +\delta 
  \hat{\undertilde{p}}^{(1)}),
  \end{align}
and
  \begin{align}
\hat{\undertilde{q}}^{\,(2)} = U^{(2)} (\hat q_{2:n}^{\,(1)} +\delta 
  \hat{\undertilde{q}}^{(1)}),
  \end{align}
where $\norm{\delta\hat{\undertilde{A}}^{(1)\dagger}}\lesssim \norm{A}\ur$, 
$\norm{\delta \hat{\undertilde{p}}^{(1)}}
\lesssim \norm{p}\ur$, and $\norm{\delta \hat{\undertilde{q}}^{(1)}}
\lesssim \norm{q}\ur$.
Like before, let 
$\hat{\undertilde{A}}^{(2)} \in \C^{(n-1)\times (n-1)}$ be equal to 
$\hat{\undertilde{A}}^{(2)\dagger}$, except that the
entry in the $(1,2)$~position of $\hat{\undertilde{A}}^{(2)}$ is equal to
$-\hat{\undertilde{p}}_1^{(2)} \hat{\undertilde{q}}_2^{(2)*}$, so that 
$\hat{\undertilde{A}}_{1,2}^{(2)} +\hat{\undertilde{p}}_1^{(2)} 
\hat{\undertilde{q}}_2^{(2)*} = 0$.  Since
$\abs{\hat{\undertilde{A}}_{1,2}^{(2)\dagger}
+\hat{\undertilde{p}}_1^{(2)} \hat{\undertilde{q}}_2^{(2)*}} < \epsilon$, 
we have
  \begin{align}
\norm{ \hat{\undertilde{A}}^{(2)\dagger} - \hat{\undertilde{A}}^{(2)} }_F <
\epsilon,
  \end{align}
where $\norm{\cdot}_F$ denotes the Frobenius norm, and since $\epsilon \lesssim
\norm{A}\ur$,
  \begin{align}
\norm{ \hat{\undertilde{A}}^{(2)\dagger} - \hat{\undertilde{A}}^{(2)} } 
\lesssim \norm{A}\ur.
  \end{align}
Letting
  \begin{align}
\delta \hat{\undertilde{A}}^{(1)} = \delta\hat{\undertilde{A}}^{(1)\dagger} 
+ \undertilde{U}^{(2)*} (\hat{\undertilde{A}}^{(2)} -
\hat{\undertilde{A}}^{(2)\dagger}) \undertilde{U}^{(2)},
  \end{align}
we have
  \begin{align}
\hat{\undertilde{A}}^{(2)} = \undertilde{U}^{(2)} (\hat A_{2:n,2:n}^{(1)} +
\delta\hat{\undertilde{A}}^{(1)}) \undertilde{U}^{(2)*},
  \end{align}
where $\norm{\delta \hat{\undertilde{A}}^{(1)}}\lesssim \norm{A}\ur$.
Since $\hat{\undertilde{A}}^{(2)}_{1,2} + \hat{\undertilde{p}}_1^{(1)} 
\hat{\undertilde{q}}_2^{(1)*} = 0$ and 
$\hat{\undertilde{A}}^{(2)} + \hat{\undertilde{p}}^{(2)}
\hat{\undertilde{q}}^{(2)*}$ is lower Hessenberg,
$\hat \lambda_2 = \hat{\undertilde{A}}_{1,1}^{(2)} + \hat{\undertilde{p}}_1^{(2)}
\hat{\undertilde{q}}_1^{(2)*}$
is an eigenvalue of $\hat{\undertilde{A}}^{(2)} + \hat{\undertilde{p}}^{(2)}
\hat{\undertilde{q}}^{(2)*}$.
Now define the unitary matrix $U^{(2)} \in \C^{n\times n}$ by the formula
  \begin{align}
U^{(2)} = \left[
  \begin{array}{c|ccc}
  1 & 0 & \cdots & 0 \\
  \hline
  0 & \multicolumn{3}{c}{\multirow{3}{*}{$\undertilde{U}^{(2)}$}} \\
  \vdots &  &  & \\
  0 &  &  &
  \end{array}
    \right],
  \end{align}
the matrix $\delta \hat A^{(1)} \in \C^{n\times n}$ by
  \begin{align}
\delta \hat A^{(1)} = \left[
  \begin{array}{c|ccc}
  0 & 0 & \cdots & 0 \\
  \hline
  0 & \multicolumn{3}{c}{\multirow{3}{*}{$\delta\hat{\undertilde{A}}^{(1)}$}} \\
  \vdots &  &  & \\
  0 &  &  &
  \end{array}
    \right],
  \end{align}
and the vectors $\delta \hat p^{(1)}, \delta \hat q^{(1)} \in \C^n$, by
  \begin{align}
\delta \hat p^{\,(1)} = \left[
  \begin{array}{c}
  0 \\
  \hline
  \delta \hat{\undertilde{p}}^{(1)}
  \end{array}
    \right],
  \end{align}
and
  \begin{align}
\delta \hat q^{\,(1)} = \left[
  \begin{array}{c}
  0 \\
  \hline
  \delta \hat{\undertilde{q}}^{(1)}
  \end{array}
    \right].
  \end{align}
Clearly, $\norm{\delta A^{(1)}} \lesssim \norm{A}\ur$,
$\norm{\delta \hat p^{\,(1)}} \lesssim \norm{p}\ur$, and
$\norm{\delta \hat q^{\,(1)}} \lesssim \norm{q}\ur$.
Let $\hat A^{(2)} \in \C^{n\times n}$ be defined by
  \begin{align}
\hat A^{(2)} = U^{(2)} (\hat A^{(1)} + \delta A^{(1)}) U^{(2)*},
    \label{ahat2def}
  \end{align}
and $\hat p^{\,(2)}, \hat q^{\,(2)} \in \C^n$ by
  \begin{align}
\hat p^{\,(2)} = U^{(2)} (\hat p^{\,(1)} + \delta \hat p^{\,(1)}),
    \label{phat2def}
  \end{align}
and
  \begin{align}
\hat q^{\,(2)} = U^{(2)} (\hat q^{\,(1)} + \delta \hat q^{\,(1)}).
    \label{qhat2def}
  \end{align}
We first notice that $\hat A_{1,\ell}^{\,(2)} + \hat p_1^{\,(2)} \hat
q_\ell^{\,(2)*} =
0$ for $\ell=2,3,\ldots,n$. Next, we observe that $\hat A_{1,1}^{(2)} + \hat
p_1^{\,(2)} \hat q_1^{\,(2)*} = \hat A_{1,1}^{(1)} + \hat p_1^{\,(1)} \hat
q_1^{\,(1)*} = \hat
\lambda_1$; therefore, $\hat \lambda_1$ is an eigenvalue of $\hat
A^{(2)} + \hat p^{\,(2)} \hat q^{\,(2)*}$.
We then observe that $(\hat A^{(2)} + \hat p^{\,(2)} \hat q^{\,(2)*})_{2:n,2:n} = 
\hat{\undertilde{A}}^{(2)} + \hat{\undertilde{p}}^{(2)}
\hat{\undertilde{q}}^{(2)*}$;
therefore, $\hat \lambda_2$ is an eigenvalue of $\hat
A^{(2)} + \hat p^{\,(2)} \hat q^{\,(2)*}$. Finally, letting $U=U^{(2)}U^{(1)}$
and substituting~(\ref{phat1def}),~(\ref{qhat1def}), and~(\ref{ahat1def})
into~(\ref{ahat2def})--(\ref{qhat2def}) and expanding,
it is straightforward to show that there exist matrices $\delta A\in
\C^{n\times n}$ and vectors $\delta p, \delta q\in \C^n$ such that
  \begin{align}
\hat A^{(2)} + \hat p^{\,(2)} \hat q^{\,(2)*}
= U(A+\delta A)U^* + U(p+\delta p)(q+\delta q)^* U^*,
  \end{align}
where $\norm{\delta A} \lesssim \norm{A}\ur$, $\norm{\delta p} \lesssim
\norm{p}\ur$, and $\norm{\delta q} \lesssim \norm{q}\ur$.
Therefore, $\hat \lambda_1$ and $\hat \lambda_2$ are eigenvalues of
the matrix
  \begin{align}
(A+\delta A) + (p+\delta p)(q+\delta q),
  \end{align}
where $\norm{\delta A} \lesssim \norm{A}\ur$, $\norm{\delta p} \lesssim
\norm{p}\ur$, and $\norm{\delta q} \lesssim \norm{q}\ur$.  The same proof
can be repeated inductively to show this for all $\hat \lambda_1,\hat
\lambda_2,\ldots,\hat \lambda_n$.

\end{proof}

The following theorem states that our explicit shifted QR algorithm is
componentwise backward stable, for those eigenvalues for which the shifts
are small.

\begin{theorem}[Explicit shifted QR]
  \label{thm:qrshift}
Suppose that $A\in \C^{n\times n}$ is a Hermitian matrix, that 
$p, q\in \C^n$, and that $A+pq^*$ is lower Hessenberg. Let $d$ and $\beta$
denote the diagonal and superdiagonal of $A$, respectively.
Suppose that Algorithm~\ref{alg:qrshift} is
carried out in floating point arithmetic with $\epsilon \lesssim \norm{A}\ur$,
and suppose that $\mu^{(\ell)}$ is the largest total shift 
encountered at any point during the 
course of the algorithm from $i=1,2,\ldots,\ell$ in the outer loop.
Let $\hat \lambda_1, \hat \lambda_2, \ldots, \hat \lambda_n$ denote
the outputs of the algorithm.  Then, for each $\ell=1,2,\ldots,n$, there
exist a matrix $\delta A \in \C^{n\times n}$ and vectors $\delta p,\delta q
\in \C^n$ such that $\hat \lambda_1, \hat \lambda_2, \ldots, \hat
\lambda_\ell$  are exact eigenvalues of the matrix
  \begin{align}
(A + \delta A) + (p + \delta p)(q + \delta q)^*,
  \end{align}
where $\norm{\delta A}\lesssim (\norm{A}+\abs{\mu^{(\ell)}})\ur$,
$\norm{\delta p} \lesssim \norm{p}\ur$, and $\norm{\delta q} \lesssim
\norm{q}\ur$.

\end{theorem}

\begin{proof}
The proof is essentially identical to the proof of
Theorem~\ref{thm:qrunshift}, and we omit it.
\end{proof}

\begin{remark}
Notice that Theorems~\ref{thm:qrunshift} and~\ref{thm:qrshift} do not make
any mention of convergence. What they say is that, \textit{if} the algorithm
converges, then it is componentwise backward stable. We observe that, in
practice, Algorithm~\ref{alg:qrshift} always converges rapidly, at least
quadratically, for $\epsilon \approx \norm{A}\ur$.

\end{remark}

\begin{remark}
Notice that the bound on $\delta A$ in Theorem~\ref{thm:qrshift} involves
$\mu^{(\ell)}$, which is the largest total shift encountered at any
point during the calculation of $\hat \lambda_1,\hat \lambda_2,\ldots, \hat
\lambda_\ell$. While this bound appears weaker than the corresponding bound
in Theorem~\ref{thm:qrunshift}, in practice it turns out to be essentially
the same, as follows.  We can always assume that $A$ is much smaller than
$p$, or $q$, or both; if this isn't the case, then componentwise stability
no longer has any special meaning, since it follows immediately from the
usual Bauer-Fike perturbation bounds~(see~\cite{bauerfike}). Furthermore, we
tend to be interested in the componentwise backward stability of small
eigenvalues $\hat\lambda_i$, where $\abs{\hat \lambda_i} \approx \norm{A}$.
If we perform a few iterations of \textit{unshifted} QR on the matrix, then
the eigenvalues of the top-left $2\times 2$ block will approach the two
smallest eigenvalues of the matrix (recalling that our QR algorithm works
with lower Hessenberg matrices).  If we now use Algorithm~\ref{alg:qrshift},
we'll find that the initial shift is small and, as a result, all the
eigenvalues are computed roughly in order from smallest to largest. This
means that $\abs{\mu^{(i)}} \approx \abs{\hat \lambda_i}$ and
(approximately) $\hat \lambda_1 < \hat \lambda_2 < \cdots < \hat \lambda_i$.
For $\hat \lambda_i$ such that $\abs{\hat \lambda_i} \approx \norm{A}$, we
have then that the bound $\norm{\delta A}\lesssim
(\norm{A}+\abs{\mu^{(i)}})\ur$ becomes $\norm{\delta A}\lesssim
\norm{A}\ur$.  Finally, we point out that the dependence of the bound on
$\mu^{(\ell)}$ could likely be removed entirely by reformulating our QR
algorithm as an implicit method.

\end{remark}

\section{Numerical Results}
  \label{sec:numerical}

In this section, we demonstrate the componentwise backward
stability of our shifted QR algorithm (see Algorithm~\ref{alg:qrshift}) by
illustrating its stability when it is used as a rootfinding algorithm (see
Sections~\ref{sec:precoll} and~\ref{sec:prestab}). Consider a polynomial
$p(x)$ of order $n$, not necessarily monic, expressed in a Chebyshev
polynomial basis
  \begin{align}
p(x) = \sum_{j=0}^n a_j T_j(x),
    \label{pdef}
  \end{align}
where $a_j \in \R$ and $T_j(x)$ is the Chebyshev polynomial of order $j$.
By Theorem~\ref{thm:linbackstab} and Remark~\ref{rem:nonmonic}, we have that
if the eigenvalues of the linearization~(\ref{colleague}), where $c_j=
a_j/a_n$, $j=0,1,\ldots,n$, are computed by a componentwise backward stable
algorithm, then the computed roots $\hat x_1, \hat x_2, \ldots, \hat x_n$
are the exact roots of the perturbed polynomial
  \begin{align}
p(x) + \delta p(x) = \sum_{j=0}^n (a_j + \delta a_j) T_j(x),
    \label{ppertdef}
  \end{align}
where
  \begin{align}
\frac{\norm{\delta a}}{\norm{a}} \lesssim \ur.
    \label{aperbnd}
  \end{align}
By applying our QR algorithm to linearizations of various polynomials
$p(x)$, we demonstrate our algorithm's componentwise backward stability
by showing that the bound~(\ref{aperbnd}) always holds.

We estimate the size of the backward error $\delta a$ in the coefficients by
using the following observation (see the discussion accompanying Table~1
in~\cite{nakatsu}). By the definition of $p(x)+\delta p(x)$, 
we have that $(p+\delta p)(\hat x_i) = 0$ for $i=1,2,\ldots,n$.
From~(\ref{pdef}) and~(\ref{ppertdef}), it follows that
  \begin{align}
p(\hat x_i) = p(\hat x_i) - (p+\delta p)(\hat x_i)
= -\sum_{j=0}^n \delta a_j T_j(\hat x_j).
  \end{align}
Since $-1 \le T_j(x) \le 1$ for all $j$ when $x\in [-1,1]$, we have
  \begin{align}
p(\hat x_i) \approx \norm{\delta a},
    \label{pxhata}
  \end{align}
whenever $\hat x_i\in \C$ is not too far from the interval $[-1,1]$.

Even though $\hat x_i$ is already a floating point number, the value $p(\hat x_i)$
cannot be computed exactly in floating point arithmetic. Letting $\hat
p(\hat x_i)$ denote the approximation to $p(\hat x_i)$ computed in floating
point arithmetic, we know that
  \begin{align}
\hat p(\hat x_i) \approx p(\hat x_i) + \kappa(p;\hat x_i)\ur,
  \end{align}
where
  \begin{align}
\kappa(p;\hat x_i) = \abs{\hat x_i} \abs{p'(\hat x_i)}
  \end{align}
is the absolute condition number of $p(x)$ at $x=\hat x_i$. When
$\kappa(p;\hat x_i)$ is large, the error in evaluating $p(\hat x_i)$ 
dominates, while when $\kappa(p;\hat x_i)$ is of modest size, we have
$\hat p(\hat x_i) \approx p(\hat x_i)$. In this section, we 
investigate the quantity
  \begin{align}
\eta(p;\hat x_i) = \frac{ \hat p(\hat x_i) }{ \max( \kappa(p;\hat x_i),
\norm{a} )},
  \end{align}
for various polynomials $p(x)$. When $\kappa(p; \hat x_i) \ge \norm{a}$,
we have
  \begin{align}
\abs{\eta(p; \hat x_i)} = \frac{\abs{\hat p(\hat x_i)}}{\kappa(p;\hat x_i)}
\approx \frac{\abs{p(\hat x_i)}}{\kappa(p;\hat x_i)}  + u
\le \frac{\abs{p(\hat x_i)}}{\norm{a}} + u.
  \end{align}
When $\kappa(p; \hat x_i) \le \norm{a}$,
  \begin{align}
\abs{\eta(p; \hat x_i)} = \frac{\abs{\hat p(\hat x_i)}}{\norm{a}}
\approx \frac{\abs{p(\hat x_i)}}{\norm{a}}  + \frac{\kappa(p;\hat x_i)}{\norm{a}}  u
\le \frac{\abs{p(\hat x_i)}}{\norm{a}} + u.
  \end{align}
Thus, if our QR algorithm is indeed componentwise backward stable and~(\ref{aperbnd})
is satisfied, then, by~(\ref{pxhata}), we expect to find that
  \begin{align}
\eta(p;\hat x_i) \approx \ur,
  \end{align}
for all polynomials $p(x)$.

For $p(\hat x_i)$ to be a good approximation to $\norm{\delta a}$
(see~(\ref{pxhata})), we stated that $\hat x_i \in \C$ should be ``not too
far from the interval $[-1,1]$.'' We make this notion precise as
follows.  Let $z_i$, $i=1,2,\ldots,n$ denote the exact roots of the
order-$n$ polynomial $p(x)$, and let $\hat z_i$ denote the computed roots.
We select roots close to the interval $[-1,1]$ by choosing some $\delta > 0$ 
(for example, $\delta=10^{-3}$), and letting $\hat x_i \in \R$ denote
the real part of all roots $\hat z_i$ that are inside the rectangle 
  \begin{align}
\{ z \in \C : 1-\delta < \Re(z) < 1+\delta, -\delta < \Im(z) < \delta \}.
    \label{zrect}
  \end{align}
If the polynomial $p(x)$ has the real root $z_i$, then taking the real part
of $\hat z_i$ will not result in any additional error. The number of real
roots inside the region~(\ref{zrect}) will often be less than the order $n$,
and we denote the number of such roots by $n_\text{roots}$.

In our numerical experiments, we compute the eigenvalues of the colleague
matrix using three different algorithms: our Algorithm~\ref{alg:qrshift};
MATLAB's \texttt{eig} function; and MATLAB's \texttt{eig} function with
balancing turned off (using the option \texttt{'nobalance'}), which
we call \texttt{eig\_nb}. For our experiments in extended (quadruple)
precision, we use the Advanpix Multiprecision Computing Toolbox and its
implementation of \texttt{eig} (see~\cite{advanpix}). Since the Advanpix
Multiprecision Computing Toolbox's \texttt{eig} function always balances the
matrix, and does not support the \texttt{'nobalance'} option, we omit the
test of \texttt{eig\_nb} in extended precision.

For each example, we report the degree of the underlying polynomial, the
order $n$ of the Chebyshev expansion used to approximate it, the size of the
vector $c$ in the Euclidean norm, the Frobenius norm of the completely
balanced colleague matrix, which we denote by $\texttt{bal}(C)$, the
number $n_\text{roots}$ of computed roots inside the region~(\ref{zrect})
for the given value of $\delta > 0$, the size $\max_i \abs{z_i}$ of the
largest complex root of the colleague matrix, and the value of $\max_i
\abs{\eta(p;\hat x_i)}$, where the maximum is taken over all of the real
parts of the computed roots inside the
region~(\ref{zrect}).

We implemented our algorithm in FORTRAN~77, and compiled it using
Lahey/Fujitsu Fortran 95 Express, Release L6.20e. For the timing
experiments, the Fortran codes were compiled using the Intel Fortran
Compiler, version 19.0.2.187, with the \texttt{-fast} flag.  The MATLAB
experiments were performed using MATLAB~R2019b, version 9.7.0.1190202, and
the extended precision MATLAB experiments were performed in quadruple
precision (\texttt{mp.Digits(34)}) using the Advanpix Multiprecision
Computing Toolbox, version 4.8.0, Build 14100. All experiments we conducted
on a ThinkPad laptop, with 16GB of RAM and an Intel Core i7-8550U CPU.

\subsection{$p_\text{rand}(x)$: Polynomials with Random Coefficients}
  \label{sec:prand}

Following~\cite{casulli}, we construct polynomials $p_\text{rand}(x)$ by
sampling Chebyshev expansion coefficients $a_i$ independently from a
standard normal distribution, so that $a_i \sim N(0,1)$, for
$i=0,1,\ldots,n-1$.  Then, we choose the desired value of $\norm{c}$ by setting
$a_n=\norm{a}/\norm{c}$, so that the vector of coefficients $c$ appearing in
the colleague matrix~(\ref{colleague}), where $c_i=a_i/a_n$ for
$i=1,2,\ldots,n$, has the specified norm. For this example, we
choose $n=30$ and set $\delta=10^{-5}$ to extract the real roots (see
formula~(\ref{zrect})).

We report the results in Figure~\ref{fig:randcoefs}. We see that our
algorithm shows the expected backward stability over the entire range of
$\norm{c}$, while MATLAB's \texttt{eig}, both balanced and unbalanced, shows
the expected growth with $\norm{c}$ (see the discussion in
Section~\ref{sec:prestab}).  Interestingly, for this example, balancing
appears to only improve the error by an order of magnitude or two,
while leaving the growth in the error with respect to $\norm{c}$ unchanged.
This turns out be completely consistent the following explanation. Balancing
the colleague matrix can reduce the magnitude of the all elements from
$\norm{c}$ to $\norm{c}^\frac{1}{2}$, except for the element in the
$(n,n)$-position, which balancing cannot change. In this example, all of the
elements of the vector $c$ are around the same size as $\norm{c}$, so there
are $n$ large elements in the colleague matrix. Thus, balancing reduces the
number of large elements of size $\norm{c}$ from $n$ to $1$, resulting in an
$n$-fold reduction in the norm of the matrix. In this example, $n=30$, which
corresponds well with the approximately $30$-fold reduction in error due to
balancing that we observe in Figure~\ref{fig:randcoefs}.

\begin{figure}[h]
  \centering
  \includegraphics[width=0.49\textwidth]{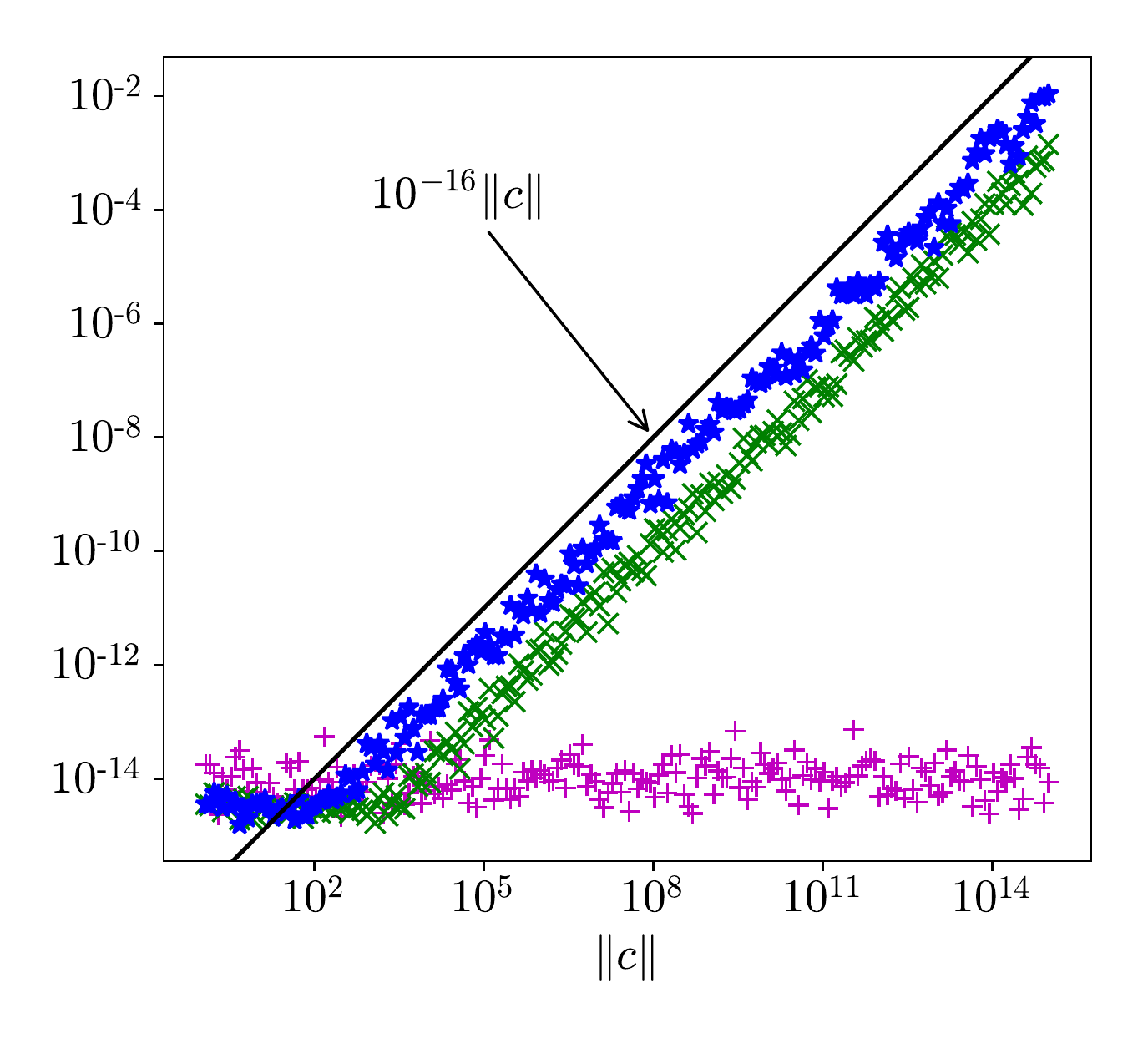}
  \includegraphics[width=0.49\textwidth]{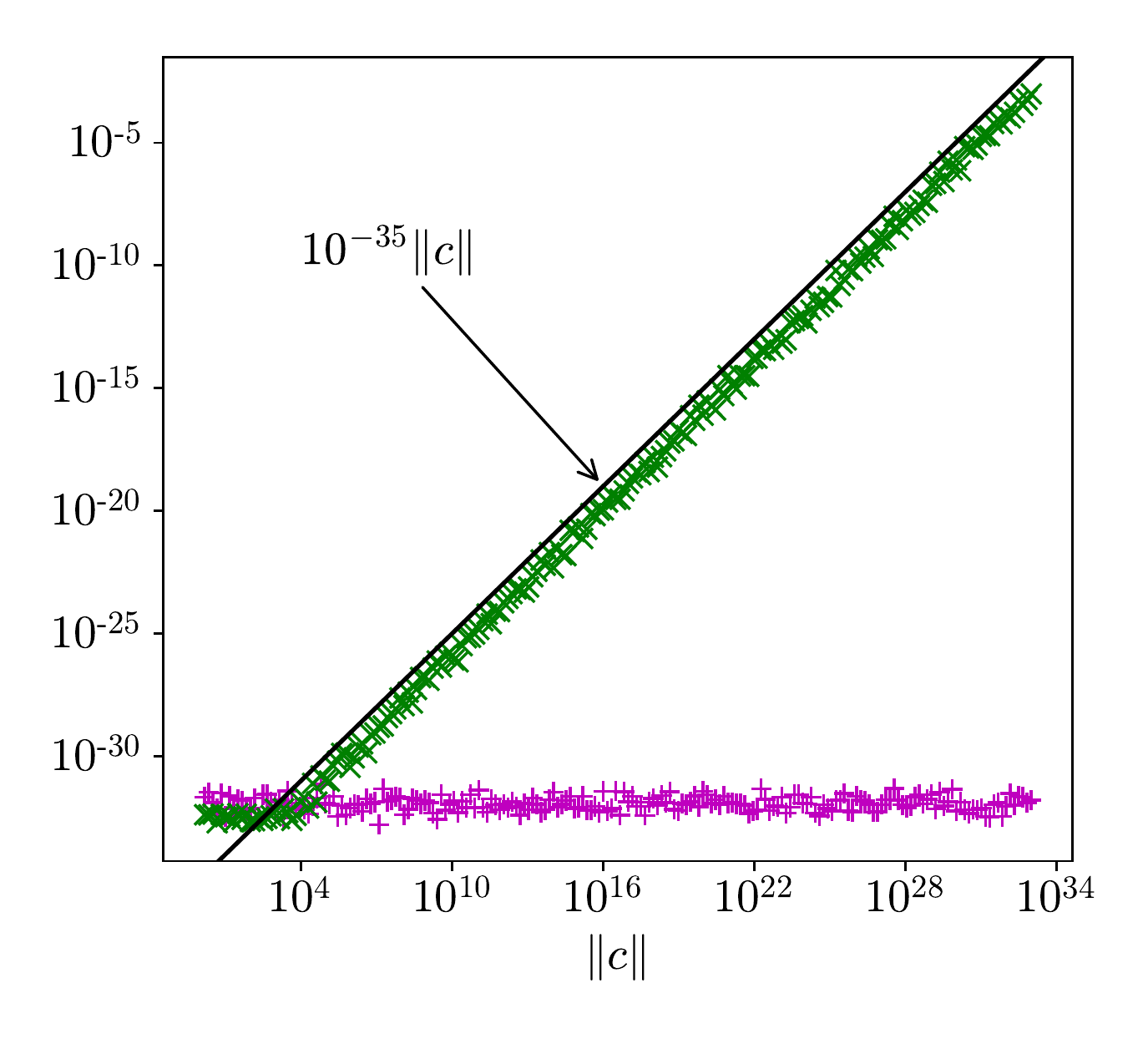}

\caption{ 
  \label{fig:randcoefs} 
The values of $\max_i \abs{\eta(p;\hat x_i)}$ for various values of
$\norm{c}$, in double precision (left) and quadruple precision (right), for
the polynomials $p_\text{rand}(x)$ of order $n=30$, computed by our algorithm,
\texttt{eig}, and \texttt{eig\_nb}, with $\delta=10^{-5}$.  The values are
indicated for our algorithm with purple crosses ($+$), for \texttt{eig} with
green x's ($\times$), and for \texttt{eig\_nb} with blue stars ($\star$).
}

\end{figure}

We found that the colleague matrix has a single large eigenvalue of the size
$\norm{c}$, and the rest of the eigenvalues are small. This is not
surprising, since there is an entry of size $\norm{c}$ is the
$(n,n)$-position of the matrix (from which it follows that $C e_n \approx
\norm{c} e_n$).  Thus, for all three algorithms, $\max_i \abs{\hat z_i}
\approx \norm{c}$.

\subsection{$p_\text{wilk}(x)$: Wilkinson's Polynomial}

Here we consider the famous Wilkinson polynomial, normalized so that
all of its roots are inside the interval $[-1,1]$, defined by the formula
  \begin{align}
p_\text{wilk}(x) = \prod_{i=1}^m \Bigl(x- \Bigl(\frac{2i}{m+1}-1\Bigr)\Bigr).
  \end{align}
We construct an order-$n$ Chebyshev expansion of this
degree-$m$ polynomial, sampling it at $n$ Chebshev points and applying a
linear transformation to obtain the expansion coefficients
(see~\cite{nickapprox}). We then compute the eigenvalues of the colleague
matrix, and set $\delta=10^{-3}$ to extract the real roots (see
formula~(\ref{zrect})). The results of our numerical experiment are shown in
Tables~\ref{tab:wilk} and~\ref{tab:wilk_quad}. We observe that our algorithm
is backward stable, while \texttt{eig} loses accuracy whenever the roots of
the colleague matrix are large.  Plots of the real and complex roots of the
order-$100$ Chebyshev expansion are shown for various degrees of
$p_\text{wilk}(x)$ in Figure~\ref{fig:wilk}.  When the order of the
Wilkinson polynomial becomes large, spurious real roots begin appearing in
the middle of the interval $[-1,1]$. It turns out that the roots computed by
our algorithm are still backward stable, even in this situation; the
function is so small near the middle of the interval that a small relative
perturbation in the Chebyshev coefficients causes additional roots to
appear.

\begin{figure}[h]
  \centering
  \includegraphics[width=0.55\textwidth]{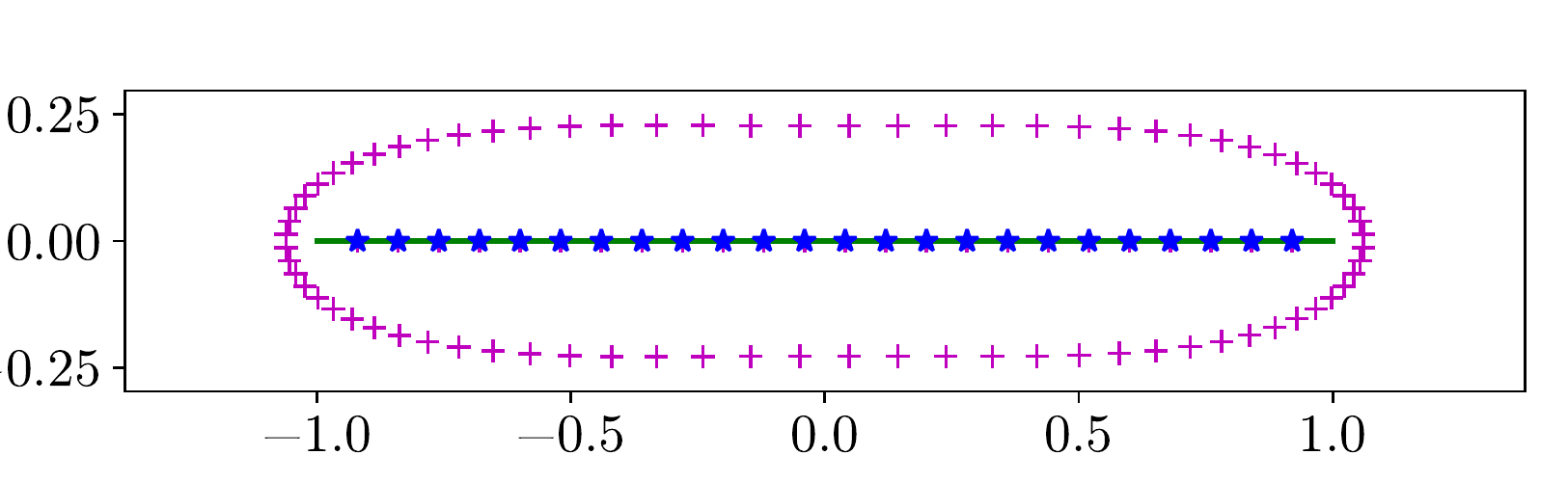} \\
  \includegraphics[width=0.55\textwidth]{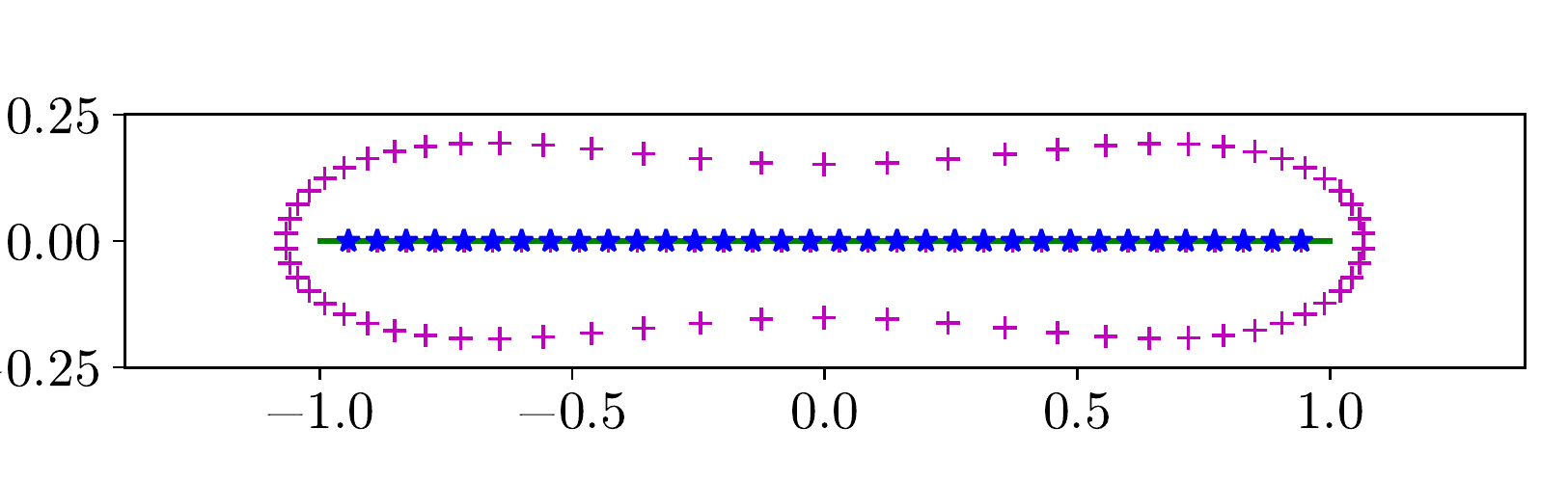} \\
  \includegraphics[width=0.55\textwidth]{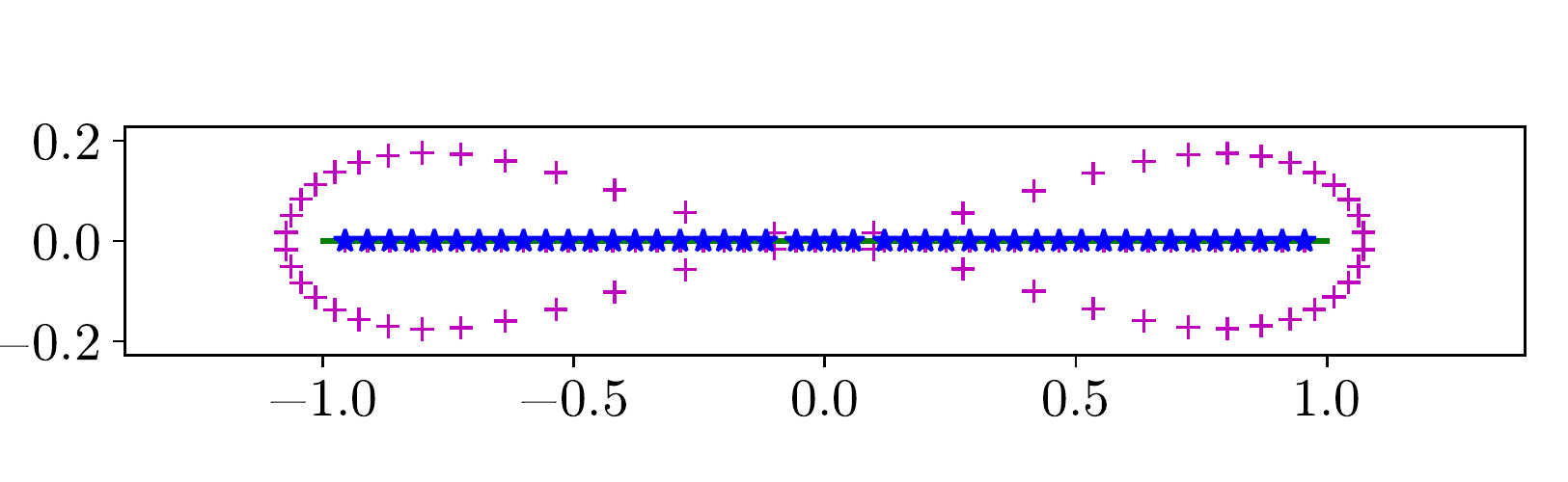} \\
  \includegraphics[width=0.55\textwidth]{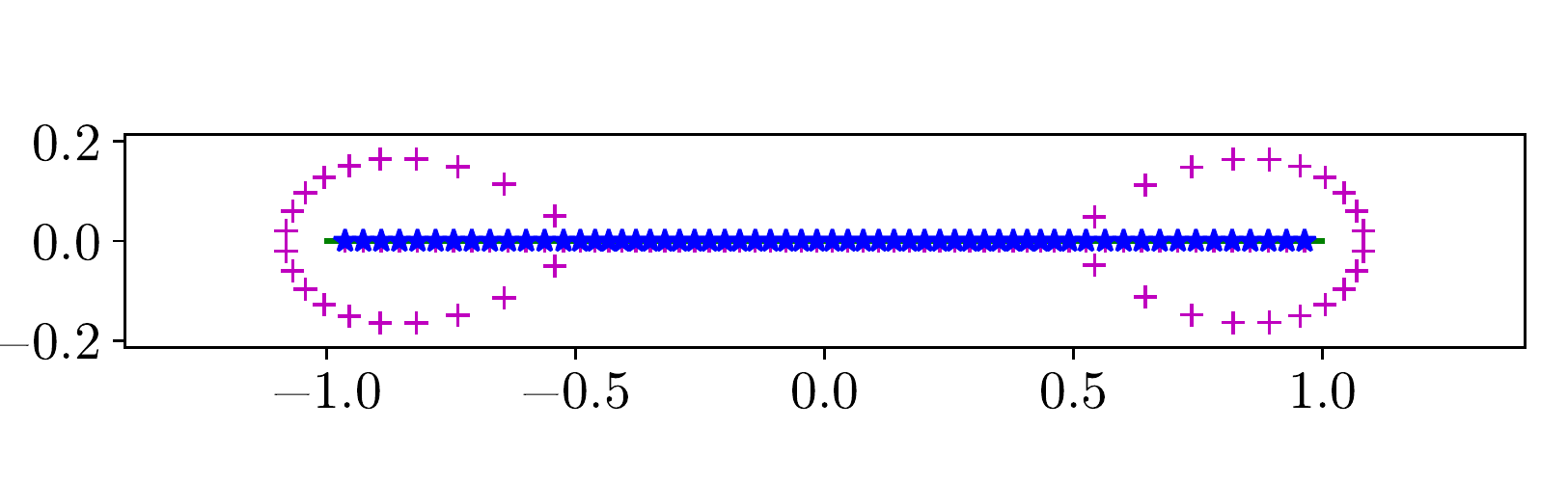}
\caption{ The roots of the Chebyshev expansion of order $100$ of the
Wilkinson polynomial $p_\text{wilk}(x)$, of various degrees, computed by our
algorithm.  The complex roots $\hat z_i$ are plotted with purple crosses
($+$) and the real roots $\hat x_i$ are plotted with blue stars ($\star$).
The Wilkinson polynomial has, in order from top to bottom, orders $24$,
$34$, $44$, and $54$. Observe that the spurious complex roots are
well-separated from the interval $[-1,1]$ when the order is low, but
eventually meet the interval when the order is large.  }
    \label{fig:wilk} 
\end{figure}

\begin{table}[h]
  \centering\small
  \setlength{\tabcolsep}{5pt}
\begin{tabular}{ccccccccc}
& & & & & \multicolumn{2}{c}{\texttt{eig}} & \multicolumn{2}{c}{Algorithm~\ref{alg:qrshift}} \\
\cline{6-7} \cline{8-9}
Degree & $n$ & $\norm{c}$ & \norm{\texttt{bal}(C)} & $\max_i \abs{z_i}$ 
  & $n_\text{roots}$ & $\max_i \abs{\eta(p;\hat x_i)}$ 
  & $n_\text{roots}$ & $\max_i \abs{\eta(p;\hat x_i)}$ \\ 
\hline \T
14 & 100 & $0.45\e{14}$ & $0.32\e{2}$ & $0.11\e{1}$  & 14 & $0.16\e{-13}$ & 14 & $0.71\e{-14}$ \\
24 & 24 & $0.95\e{4}$   & $0.65\e{1}$ & $0.92\e{0}$  & 24 & $0.15\e{-14}$ & 24 & $0.32\e{-14}$ \\
   & 25 & $0.66\e{15}$  & $0.35\e{11}$ & $0.35\e{11}$ & 14$^\dagger$ & $0.11\e{-4}$ & 24 & $0.19\e{-14}$  \\
   & 26 & $0.22\e{15}$  & $0.11\e{6}$ & $0.76\e{5}$  & 24 & $0.51\e{-10}$ & 24 & $0.24\e{-14}$ \\
   & 27 & $0.40\e{16}$  & $0.66\e{4}$ & $0.38\e{4}$  & 24 & $0.90\e{-12}$ & 24 & $0.19\e{-14}$ \\
   & 28 & $0.12\e{15}$  & $0.37\e{3}$ & $0.17\e{3}$  & 24 & $0.58\e{-13}$ & 24 & $0.14\e{-14}$ \\
   & 100 & $0.38\e{14}$ & $0.24\e{2}$ & $0.11\e{1}$  & 24 & $0.10\e{-13}$ & 24 & $0.24\e{-14}$ \\
34 & 100 & $0.34\e{14}$ & $0.17\e{2}$ & $0.11\e{1}$  & 34 & $0.14\e{-13}$ & 34 & $0.12\e{-13}$ \\ 
44 & 100 & $0.32\e{14}$ & $0.16\e{2}$ & $0.11\e{1}$  & 44 & $0.53\e{-13}$ & 44 & $0.41\e{-14}$ \\ 
54 & 100 & $0.30\e{14}$ & $0.16\e{2}$ & $0.11\e{1}$  & 60 & $0.15\e{-13}$ & 60 & $0.28\e{-13}$    
\end{tabular}

\caption{
  \label{tab:wilk}
The results of computing the roots of the Wilkinson polynomial
$p_\text{wilk}(x)$, using our algorithm and \texttt{eig}, with
$\delta=10^{-3}$.  $^\dagger$The error was so large here that some roots
were outside of the region~(\ref{zrect}).
}
\end{table}

\begin{table}[h]
  \centering\small
  \setlength{\tabcolsep}{5pt}
\begin{tabular}{ccccccccc}
& & & & & \multicolumn{2}{c}{\texttt{eig}} & \multicolumn{2}{c}{Algorithm~\ref{alg:qrshift}} \\
\cline{6-7} \cline{8-9}
Degree & $n$ & $\norm{c}$ & \norm{\texttt{bal}(C)} & $\max_i \abs{z_i}$ 
  & $n_\text{roots}$ & $\max_i \abs{\eta(p;\hat x_i)}$ 
  & $n_\text{roots}$ & $\max_i \abs{\eta(p;\hat x_i)}$ \\ 
\hline \T
14 & 100 & $0.12\e{34}$ & $0.38\e{3}$ & $0.14\e{1}$ & 14 & $0.22\e{-31}$ & 14 & $0.15\e{-32}$  \\
24 & 24 & $0.95\e{4}$   & $0.65\e{1}$ & $0.92\e{0}$ & 24 & $0.44\e{-32}$ & 24 & $0.66\e{-33}$  \\
   & 25 & $0.95\e{33}$  & $0.50\e{29}$ & $0.50\e{29}$ & 12$^\dagger$ & $0.29\e{-4}$ & 24 & $0.78\e{-32}$  \\
   & 26 & $0.16\e{33}$  & $0.94\e{14}$ & $0.66\e{14}$ & 24 & $0.11\e{-19}$ & 24 & $0.18\e{-32}$  \\
   & 27 & $0.11\e{36}$  & $0.22\e{11}$ & $0.11\e{11}$ & 24 & $0.55\e{-23}$ & 24 & $0.76\e{-32}$  \\
   & 28 & $0.22\e{33}$  & $0.13\e{8}$ & $0.62\e{7}$ & 24 & $0.34\e{-27}$ & 24 & $0.28\e{-32}$  \\
   & 100 & $0.94\e{33}$ & $0.29\e{3}$ & $0.14\e{1}$ & 24 & $0.53\e{-31}$ & 24 & $0.37\e{-32}$  \\
34 & 100 & $0.70\e{33}$ & $0.26\e{3}$ & $0.15\e{1}$ & 34 & $0.21\e{-31}$ & 34 & $0.68\e{-32}$  \\
44 & 100 & $0.53\e{33}$ & $0.23\e{3}$ & $0.16\e{1}$ & 44 & $0.54\e{-32}$ & 44 & $0.40\e{-32}$  \\
54 & 100 & $0.42\e{33}$ & $0.16\e{3}$ & $0.18\e{1}$ & 54 & $0.28\e{-31}$ & 54 & $0.73\e{-32}$ 
\end{tabular}

\caption{
  \label{tab:wilk_quad}
The results of computing the roots of the Wilkinson polynomial
$p_\text{wilk}(x)$ in extended precision, using our algorithm and
\texttt{eig}, with $\delta=10^{-3}$.  $^\dagger$The error was so large here
that some roots were outside of the region~(\ref{zrect}).
}
\end{table}

\begin{remark}
  \label{rem:balanstab}
The remarkable stability of \texttt{eig} for many of the examples in
Tables~\ref{tab:wilk} and~\ref{tab:wilk_quad} is explained by the following
observation.  The colleague matrix is the sum of a tridiagonal matrix and a
matrix that is all zeros except for the last row, which is essentially equal
to the coefficient vector $c$ (see formula~(\ref{colleague})). When there
are large elements of $c$ near the tail of the vector, the corresponding
large entries in the colleague matrix cannot be balanced away, since they
are very close to the diagonal of the matrix. On the other hand, when all of
the elements of $c$ near the tail of the vector are relatively small, and
the large elements of $c$ appear near the head of the vector, these large
elements can be easily balanced away, since they are far from the diagonal,
and the corresponding elements on the other side of the diagonal are all
zero. The coefficient vector $c$ is usually large only because the last
coefficient of the corresponding non-monic Chebyshev expansion is small. 
If the function being approximated by this non-monic Chebyshev expansion has
been adequately represented, then taking additional terms in the expansion
will result in corresponding expansion coefficients which are all machine
epsilon in size.  Thus, adding terms to the Chebyshev expansion has the
effect of adding elements of size approximately one to the tail of the
coefficient vector $c$; if enough such elements are added, then all the
large elements of $c$ will be closer to the head of the vector, and can be
balanced away.
We also observe that, not unexpectedly, the size of the largest eigenvalue
of the colleague matrix is approximately the same size as the norm of the
colleague matrix after balancing.  Thus, if enough terms are taken in a
Chebyshev expansion, all of the eigenvalues of the colleague matrix will
eventually be small.
All of this indicates that, provided enough terms are taken, a dense
eigensolver combined with balancing can result in a backward stable
rootfinding algorithm.  Of course, we note that balancing the colleague
matrix destroys the Hermitian plus rank-1 structure, which bars the use of
structured QR algorithms depending on this property. 

\end{remark}

\subsection{$f_\text{sin}(x)$: A Smooth Function}

Here we construct an order-$n$ Chebyshev expansion of the smooth function
  \begin{align}
f_\text{sin}(x) = \sin(2+20(x+0.222)^2).
  \end{align}
Since $f_\text{sin}(x)$ is analytic, its expansion coefficients decay
exponentially. When $i \ge 80$, the coefficients $a_i$ are around $10^{-14}$
in size (see Figure~\ref{fig:fsin}). If the coefficients are computed in
extended precision, then, when $i \ge 125$, the coefficients $a_i$ are
around $10^{-34}$ in size.  Since the function is approximated accurately by
a Chebyshev expansion, its roots can be computed from the corresponding
colleague matrix (see, for example,~\cite{boyd2} for a nice discussion).
The results of our numerical experiment are shown in Tables~\ref{tab:sin}
and~\ref{tab:sin_quad}.  Plots of the real and complex roots of the
order-$100$ Chebyshev expansion are shown in Figure~\ref{fig:fsinroots}.

\begin{figure}[h]
  \centering
  \includegraphics[width=0.49\textwidth]{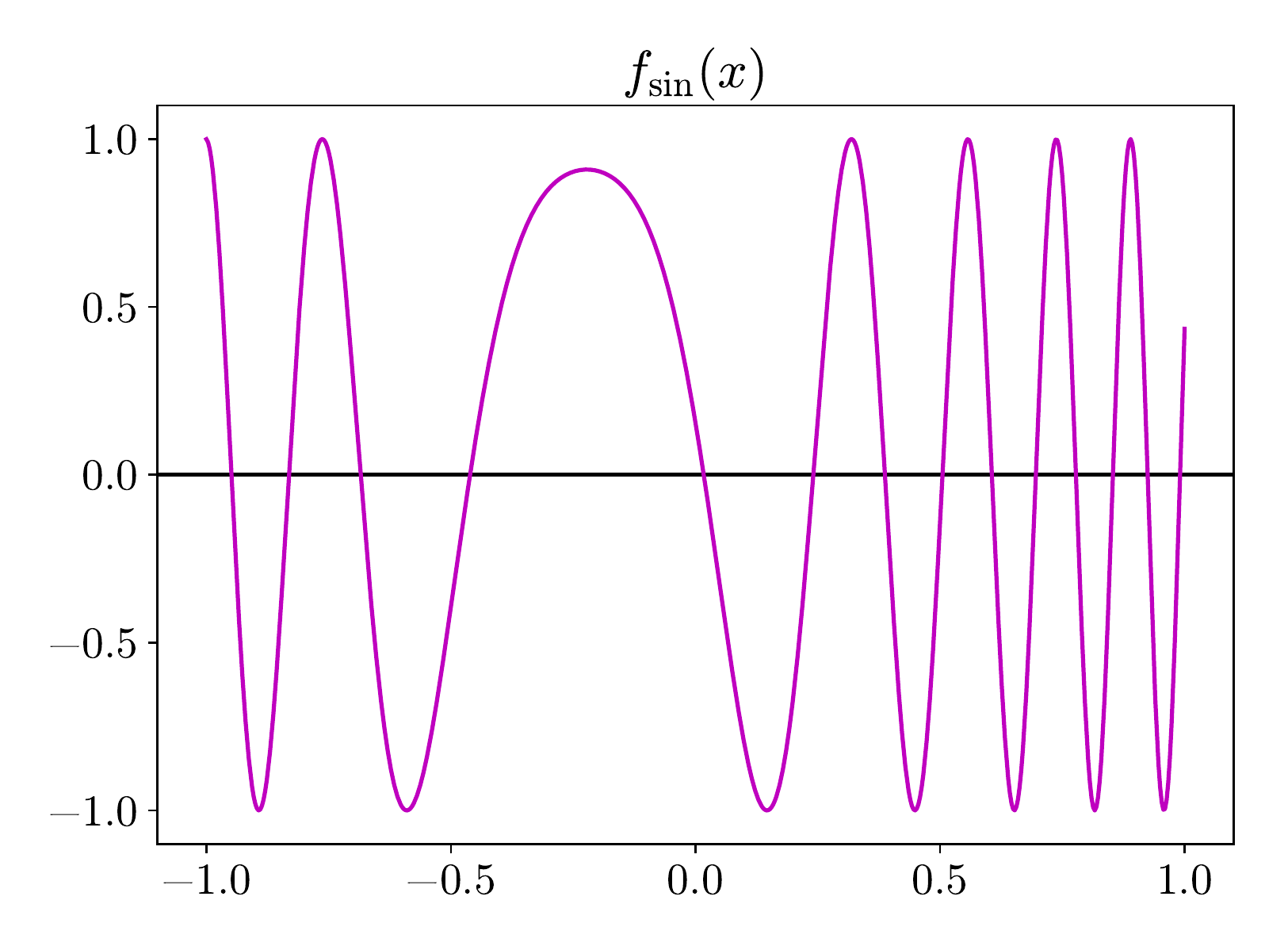}
  \includegraphics[width=0.49\textwidth]{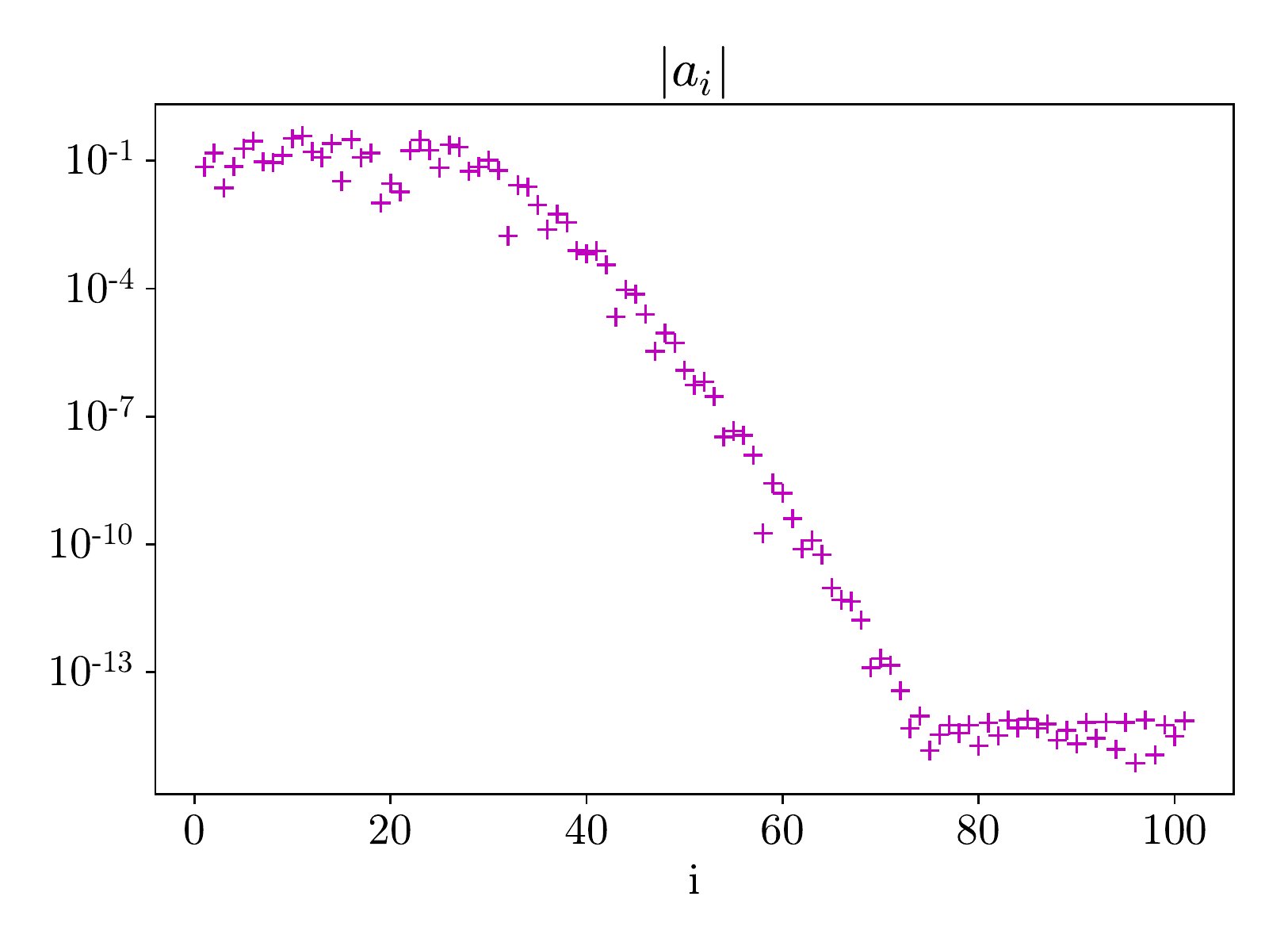}

\caption{ 
  \label{fig:fsin} 
The function $f_\text{sin}(x)$, shown on the left, and the magnitude of
its Chebyshev expansion coefficients $a_i$, shown on the right.
}

\end{figure}

\begin{table}[h]
  \centering\small
  \setlength{\tabcolsep}{5pt}
\begin{tabular}{cccccccc}
& & & & \multicolumn{2}{c}{\texttt{eig}} & \multicolumn{2}{c}{Algorithm~\ref{alg:qrshift}} \\
\cline{5-6} \cline{7-8}
$n$ & $\norm{c}$ & \norm{\texttt{bal}(C)} & $\max_i \abs{z_i}$ 
  & $n_\text{roots}$ & $\max_i \abs{\eta(p;\hat x_i)}$ 
  & $n_\text{roots}$ & $\max_i \abs{\eta(p;\hat x_i)}$ \\ 
\hline \T
80 & $0.89\e{15}$  & $0.16\e{2}$ & $0.29\e{1}$ & 14 & $0.74\e{-14}$ & 14 & $0.10\e{-13}$ \\
100 & $0.14\e{15}$ & $0.14\e{2}$ & $0.11\e{1}$ & 14 & $0.84\e{-14}$ & 14 & $0.26\e{-13}$
\end{tabular}

\caption{
  \label{tab:sin}
The results of computing the roots of the Chebyshev expansion
of the function $f_\text{sin}(x)$,
using our algorithm and \texttt{eig}, with
$\delta=10^{-3}$.  
}

\end{table}

\begin{table}[h]
  \centering\small
  \setlength{\tabcolsep}{5pt}
\begin{tabular}{cccccccc}
& & & & \multicolumn{2}{c}{\texttt{eig}} & \multicolumn{2}{c}{Algorithm~\ref{alg:qrshift}} \\
\cline{5-6} \cline{7-8}
$n$ & $\norm{c}$ & \norm{\texttt{bal}(C)} & $\max_i \abs{z_i}$ 
  & $n_\text{roots}$ & $\max_i \abs{\eta(p;\hat x_i)}$ 
  & $n_\text{roots}$ & $\max_i \abs{\eta(p;\hat x_i)}$ \\ 
\hline \T
125 & $0.92\e{33}$  & $0.30\e{2}$ & $0.25\e{1}$ & 14 & $0.53\e{-32}$ & 14 & $0.50\e{-31}$ \\
200 & $0.49\e{32}$  & $0.29\e{2}$ & $0.11\e{1}$ & 14 & $0.12\e{-31}$ & 14 & $0.60\e{-31}$
\end{tabular}

\caption{
  \label{tab:sin_quad}
The results of computing the roots of the Chebyshev expansion of the
function $f_\text{sin}(x)$ in extended precision, using our algorithm and
\texttt{eig}, with $\delta=10^{-3}$.  
}

\end{table}

\begin{figure}[h]
  \centering
  \includegraphics[width=0.52\textwidth]{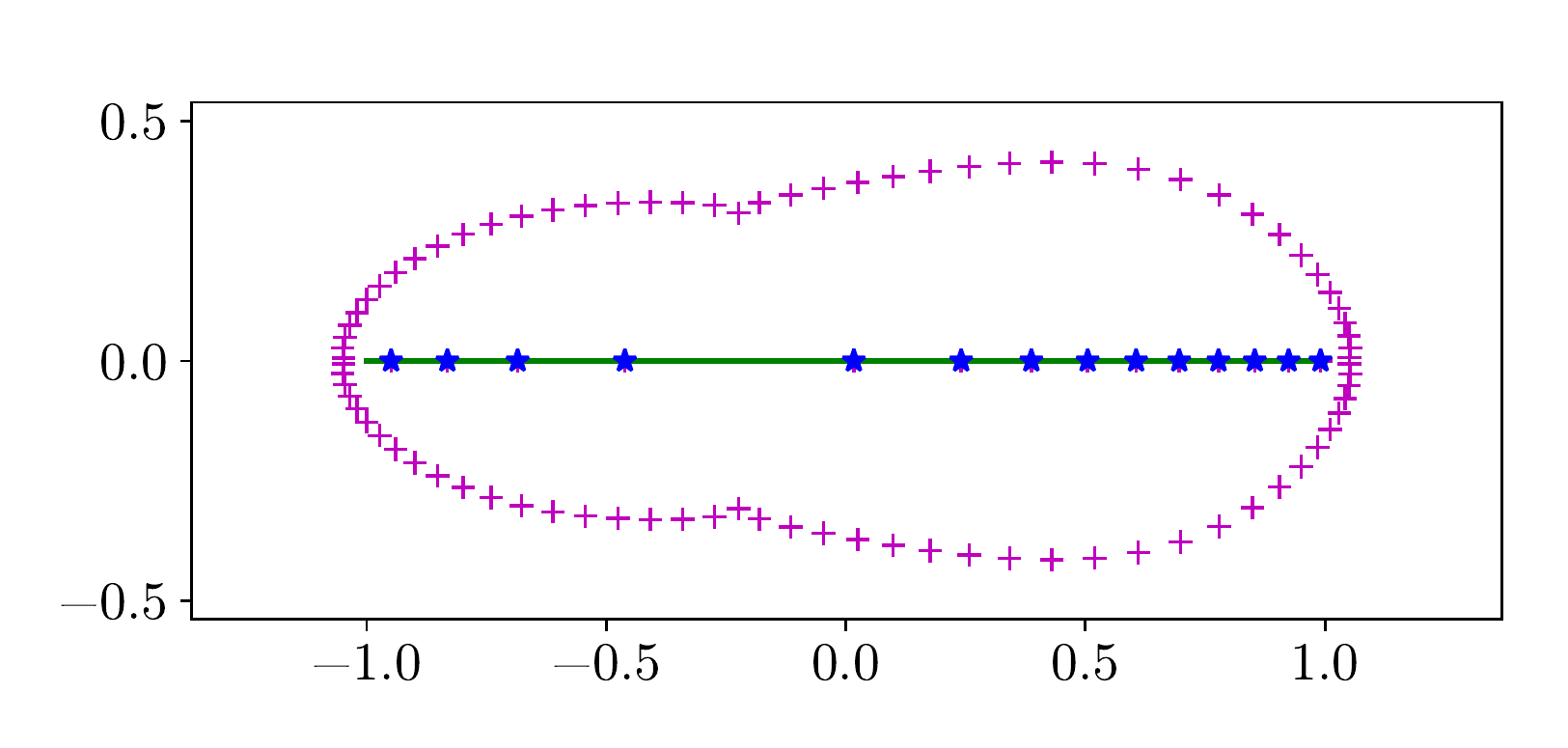}

\caption{ 
  \label{fig:fsinroots} 
The roots of the Chebyshev expansion of order $100$ of the function
$f_\text{sin}(x)$. The complex roots $\hat z_i$ are plotted with purple
crosses ($+$) and the real roots $\hat x_i$ are plotted with blue stars
($\star$).
}
\end{figure}

\subsection{$p_\text{mult}(x)$: A Polynomial with Multiple Roots}

Here we construct an order $n$ Chebyshev expansion of the degree $m$
polynomial
  \begin{align}
&\hspace*{-3em}
f_\text{mult}(x) = (x+\tfrac{1}{2})
(x+\tfrac{1}{3}) (x+0.61) (x-0.121)
\prod_{i=1}^{m-4} (x - (1-10^{-3})).
  \end{align}
This polynomial has four simple roots on the interval $[-1,1]$, and a root
of multiplicity $(m-4)$ at the point $1-10^{-3}$ (see
Figure~\ref{fig:multgraph}).  The results of our numerical experiments are
shown in Tables~\ref{tab:mult} and~\ref{tab:mult_quad}.  We observe that, in
double precision, when the multiplicity of the root is greater than or equal
to 5, not all real roots are found. This is because the error in these roots
is approximately equal to $\epsilon^\frac{1}{5}$, and when $\epsilon \approx
10^{-14}$, we have that $\epsilon^\frac{1}{5} \approx 1.6\e{-3}$; when
$\delta=10^{-3}$, this means that some of these roots can be outside the
region~(\ref{zrect}).  Likewise, since when $\epsilon \approx 10^{-34}$,
$\epsilon^\frac{1}{12} \approx 1.4\e{-3}$, it follows that in extended
precision, real roots are missed when their multiplicity is greater than or
equal to approximately 12.  See the excellent discussion in~\cite{boyd} for
more details.

\begin{figure}[h]
  \centering
  \includegraphics[width=0.49\textwidth]{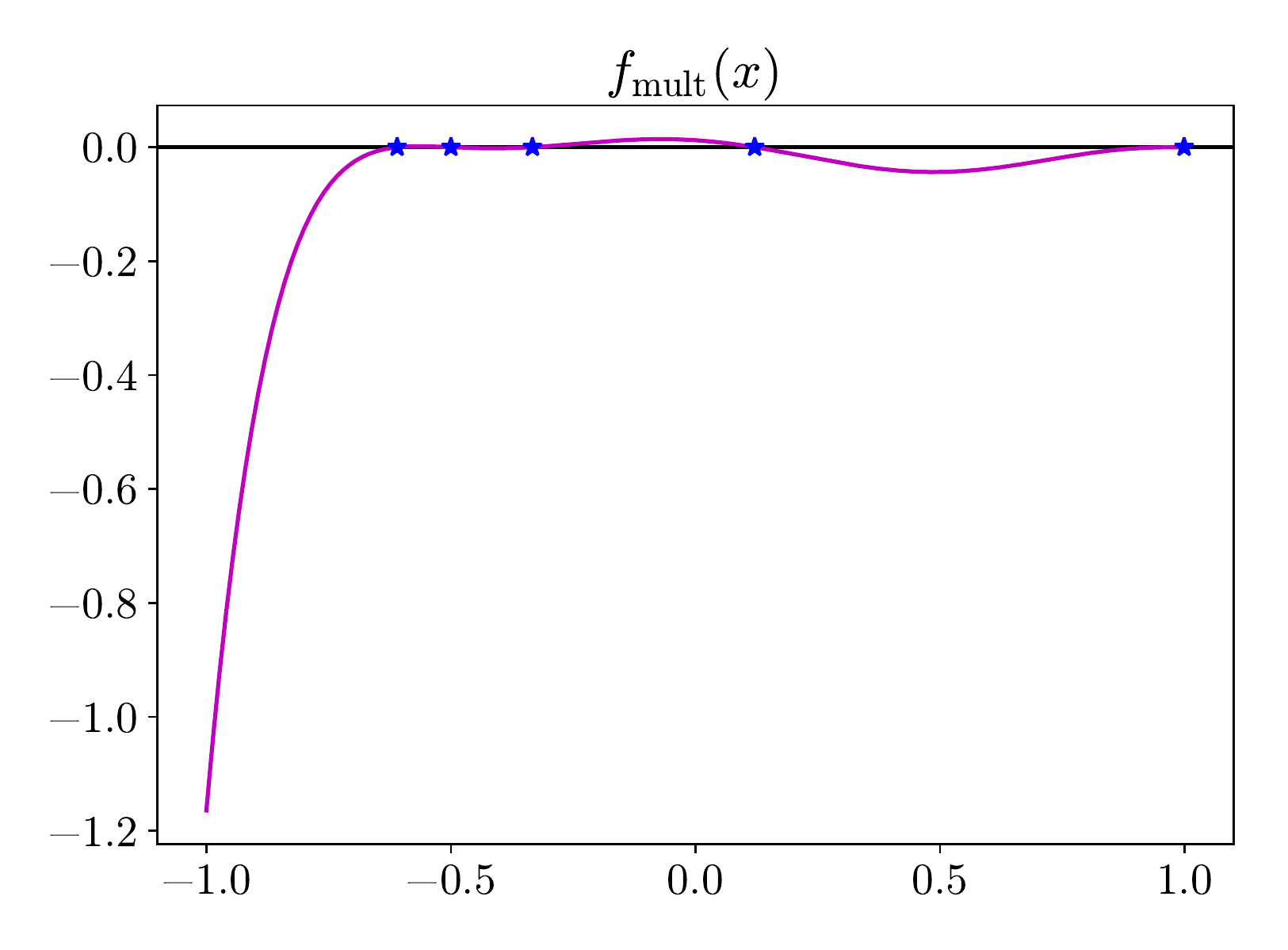}
  \includegraphics[width=0.49\textwidth]{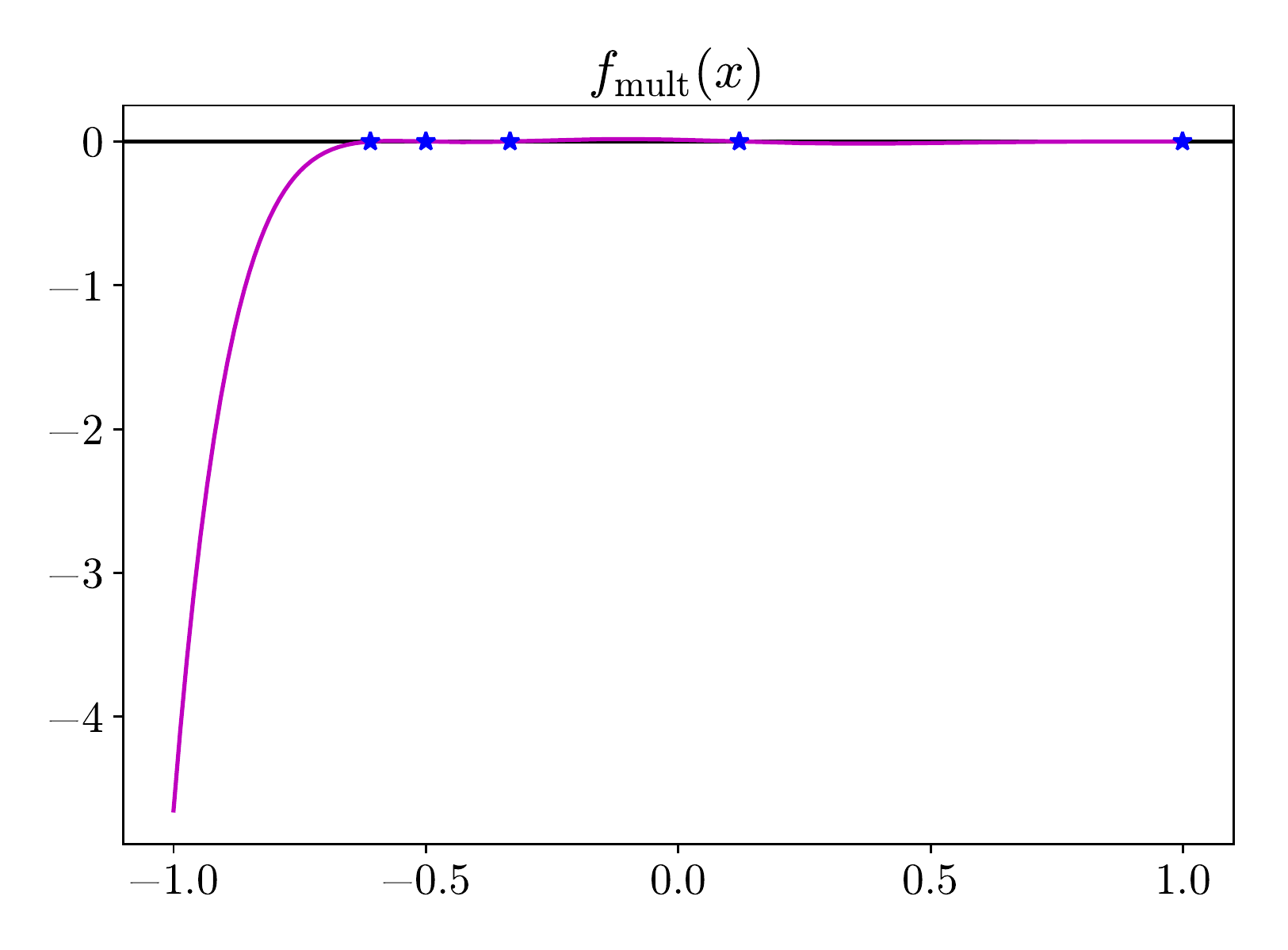}

\caption{ 
  \label{fig:multgraph} 
The polynomial $p_\text{mult}(x)$ of order $7$ on the left, and order $9$ on the 
right. The roots are indicated with blue stars ($\star$).
}
\end{figure}

\begin{table}[h]
  \centering\small
  \setlength{\tabcolsep}{5pt}
\begin{tabular}{ccccccccc}
& & & & & \multicolumn{2}{c}{\texttt{eig}} & \multicolumn{2}{c}{Algorithm~\ref{alg:qrshift}} \\
\cline{6-7} \cline{8-9}
Degree & $n$ & $\norm{c}$ & \norm{\texttt{bal}(C)} & $\max_i \abs{z_i}$ 
  & $n_\text{roots}$ & $\max_i \abs{\eta(p;\hat x_i)}$ 
  & $n_\text{roots}$ & $\max_i \abs{\eta(p;\hat x_i)}$ \\ 
\hline \T
7 & 100 & $0.10\e{15}$ & $0.38\e{2}$ & $0.11\e{1}$  & 7  & $0.82\e{-14}$ & 7 & $0.14\e{-14}$ \\
8 & 8 & $0.12\e{3}$ & $0.50\e{1}$   & $0.10\e{1}$  & 8  & $0.59\e{-15}$ & 8 & $0.93\e{-15}$ \\
  & 9 & $0.54\e{15}$ & $0.22\e{13}$ & $0.22\e{13}$  & 5$^\dagger$ & $0.33\e{-4}$ & 8 & $0.11\e{-14}$ \\
  & 10 & $0.61\e{15}$ & $0.16\e{7}$ & $0.11\e{7}$  & 6$^\dagger$ & $0.14\e{-9}$  & 8 & $0.88\e{-15}$ \\
  & 11 & $0.79\e{15}$ & $0.17\e{5}$  & $0.93\e{4}$  & 8  & $0.20\e{-11}$  & 8 & $0.83\e{-15}$ \\
  & 100 & $0.99\e{14}$ & $0.32\e{2}$ & $0.11\e{1}$  & 8  & $0.64\e{-14}$  & 8 & $0.26\e{-15}$ \\
9 & 100 & $0.97\e{14}$ & $0.32\e{2}$  & $0.11\e{1}$  & 8$^\diamond$  & $0.99\e{-14}$ & 8$^\diamond$ & $0.88\e{-14}$  \\ 
10 & 100 & $0.96\e{14}$ & $0.26\e{2}$  & $0.11\e{1}$  & 8$^\diamond$  & $0.73\e{-15}$ & 8$^\diamond$ & $0.38\e{-15}$  \\
13 & 100 & $0.92\e{14}$ & $0.22\e{2}$ & $0.11\e{1}$  & 12$^\diamond$  & $0.12\e{-14}$ & 12$^\diamond$ & $0.88\e{-15}$    
\end{tabular}

\caption{
  \label{tab:mult}
The results of computing the roots of the polynomial $p_\text{mult}(x)$,
using our algorithm and \texttt{eig}, with $\delta=10^{-3}$.  $^\dagger$The
error was so large here that some roots were outside of the
region~(\ref{zrect}).  $^\diamond$The multiplicity of the rightmost root was
so large here that some roots were outside of the region~(\ref{zrect}).
}
\end{table}

\begin{table}[h]
  \centering\small
  \setlength{\tabcolsep}{5pt}
\begin{tabular}{ccccccccc}
& & & & & \multicolumn{2}{c}{\texttt{eig}} & \multicolumn{2}{c}{Algorithm~\ref{alg:qrshift}} \\
\cline{6-7} \cline{8-9}
Degree & $n$ & $\norm{c}$ & \norm{\texttt{bal}(C)} & $\max_i \abs{z_i}$ 
  & $n_\text{roots}$ & $\max_i \abs{\eta(p;\hat x_i)}$ 
  & $n_\text{roots}$ & $\max_i \abs{\eta(p;\hat x_i)}$ \\ 
\hline \T
10 & 100 & $0.69\e{33}$ & $0.29\e{3}$ & $0.13\e{1}$  & 10 & $0.94\e{-32}$ & 10 & $0.51\e{-33}$ \\
11 & 11 & $0.75\e{4}$ & $0.98\e{1}$ & $0.10\e{1}$  & 11 & $0.76\e{-33}$ & 11 & $0.12\e{-32}$ \\
  & 12 & $0.16\e{34}$ & $0.11\e{30}$ & $0.11\e{30}$  & 6$^\dagger$ & $0.28\e{-5}$ & 11 & $0.95\e{-33}$ \\
  & 13 & $0.37\e{34}$ & $0.53\e{15}$ & $0.35\e{15}$  & 7$^\dagger$ & $0.63\e{-19}$ & 11 & $0.66\e{-33}$ \\
  & 14 & $0.25\e{34}$ & $0.68\e{10}$ & $0.35\e{10}$  & 11  & $0.51\e{-25}$ & 11 & $0.37\e{-33}$ \\
  & 100 & $0.70\e{33}$ & $0.37\e{3}$ & $0.13\e{1}$  & 11 & $0.11\e{-31}$ & 11 & $0.15\e{-32}$ \\
12 & 100 & $0.71\e{33}$ & $0.29\e{3}$ & $0.13\e{1}$  & 12 & $0.20\e{-31}$ & 12 & $0.81\e{-33}$  \\ 
13 & 100 & $0.72\e{33}$ & $0.30\e{3}$ & $0.13\e{1}$  & 13 & $0.17\e{-31}$ & 13 & $0.11\e{-32}$  \\
14 & 100 & $0.72\e{33}$ & $0.28\e{3}$ & $0.13\e{1}$  & 14  & $0.24\e{-31}$ & 14 & $0.12\e{-32}$  \\
15 & 100 & $0.73\e{33}$ & $0.28\e{3}$ & $0.13\e{1}$  & 9$^\diamond$ & $0.21\e{-31}$ & 11$^\diamond$ & $0.21\e{-31}$    
\end{tabular}

\caption{
  \label{tab:mult_quad}
The results of computing the roots of the polynomial $p_\text{mult}(x)$ in
extended precision, using our algorithm and \texttt{eig}, with
$\delta=10^{-3}$.  $^\dagger$The error was so large here that some roots
were outside of the region~(\ref{zrect}).  $^\diamond$The multiplicity of
the rightmost root was so large here that some roots were outside of the
region~(\ref{zrect}).
}
\end{table}

\subsection{$p_\text{yuji}(x)$: A Pathological Example from~\cite{nakatsu}}

Here we consider the order-$8$ polynomial
  \begin{align}
p_\text{yuji}(x) = \sum_{i=0}^8 a_i T_i(x),
  \end{align}
where the coefficient vector $a$ is given by
  \begin{align}
a = \left(\begin{array}{ccccccccc}
-\frac{1}{10} &-\frac{1}{10} &-\frac{1}{10} &-\frac{1}{10} &-\frac{1}{10} & -\frac{1}{10} & 10^{-10} & 1 & 10^{-15}
\end{array}\right),
  \end{align}
described in~\S6.1 of~\cite{nakatsu} (in~\cite{nakatsu}, the authors set the
last element of the coefficient vector to $10^{-20}$; we set it close to
machine epsilon instead).  Observe that the entry in the bottom right corner
of the corresponding colleague matrix is around $10^{15}$ in size.  This
polynomial has seven real roots on the interval $[-1,1]$, and a single large
imaginary root.  We report the results of our numerical experiment in
Table~\ref{tab:yuji}. Clearly, \texttt{eig} struggles to produce any
accuracy at all, while our algorithm returns all the roots to machine
precision.

\begin{table}[h]
  \centering\small
  \setlength{\tabcolsep}{5pt}
\begin{tabular}{ccccccccc}
& & & & & \multicolumn{2}{c}{\texttt{eig}} & \multicolumn{2}{c}{Algorithm~\ref{alg:qrshift}} \\
\cline{6-7} \cline{8-9}
Degree & $n$ & $\norm{c}$ & \norm{\texttt{bal}(C)} & $\max_i \abs{z_i}$ 
  & $n_\text{roots}$ & $\max_i \abs{\eta(p;\hat x_i)}$ 
  & $n_\text{roots}$ & $\max_i \abs{\eta(p;\hat x_i)}$ \\ 
\hline \T
8 & 8 & $0.10\e{16}$ & $0.50\e{15}$ & $0.50\e{15}$  & 7  & $0.21\e{-1}$ & 7 & $0.77\e{-14}$ \\
\end{tabular}

\caption{
  \label{tab:yuji}
The results of computing the roots of the polynomial $p_\text{yuji}(x)$,
using our algorithm and \texttt{eig}, with $\delta=10^{-3}$.
}
\end{table}

\subsection{$f_\text{cas}(x)$: A Pathological Example from~\cite{casulli}}

Here we consider the order-$n$ Chebyshev expansion of the smooth
function
  \begin{align}
f_\text{cas}(x) = \sin\Bigl(\frac{1}{x^2 + 10^{-2}}\Bigr),
  \end{align}
described in~\cite{casulli}. The first 1430 Chebyshev expansion coefficients
of $f_\text{cas}(x)$ are shown in Figure~\ref{fig:cascoefs}. This function
is highly oscillatory, and requires a Chebyshev expansion of order at least
1430 to resolve it.  Our numerical experiments are shown in
Table~\ref{tab:cas}.  Plots of the real and complex roots of the
order-$1600$ Chebyshev expansion are shown in Figure~\ref{fig:casroots}.

\begin{figure}[h]
  \centering
  \includegraphics[width=0.49\textwidth]{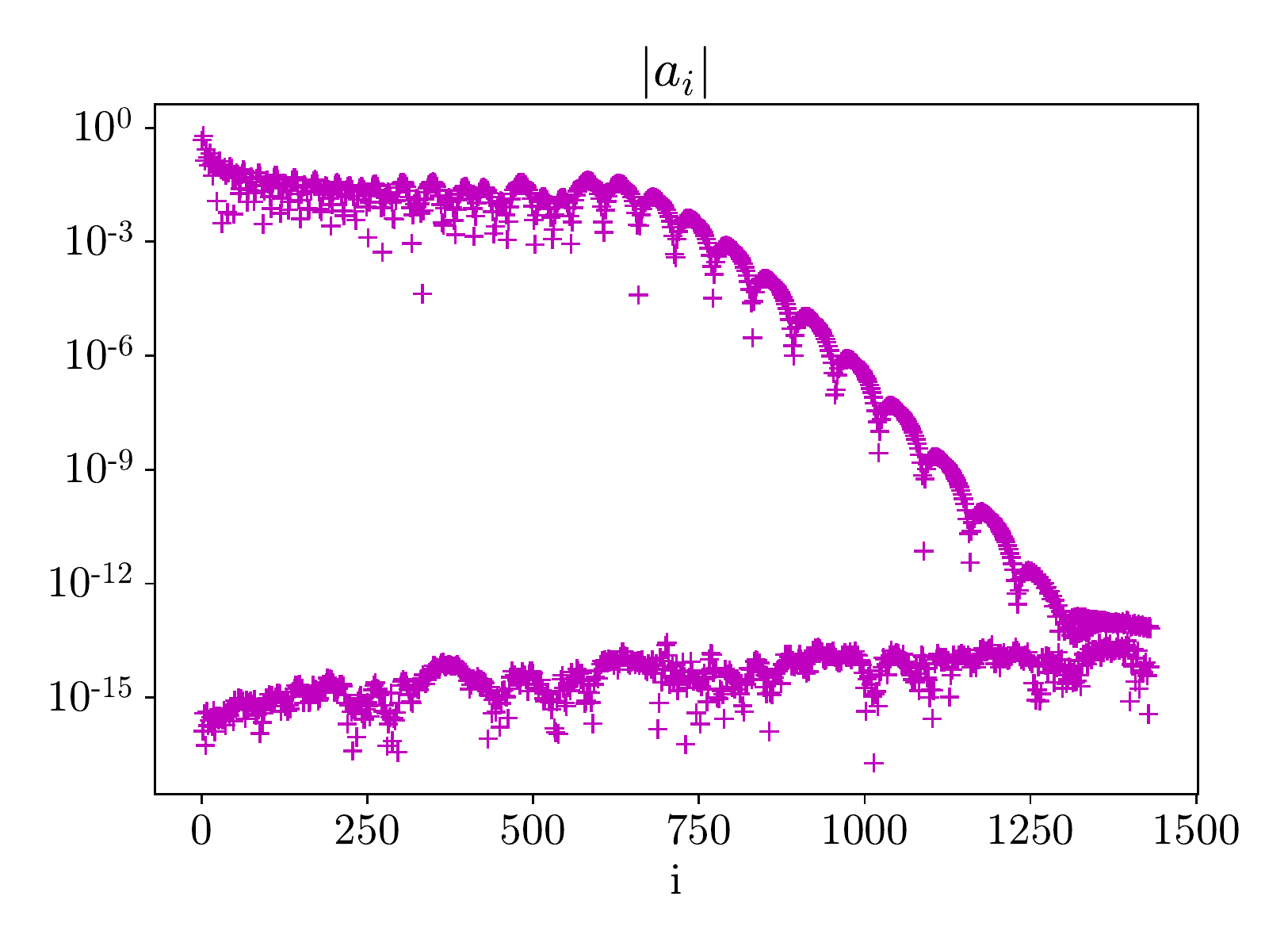}

\caption{ 
  \label{fig:cascoefs} 
The magnitudes of the first 1430 Chebyshev expansion coefficients of
$f_\text{cas}(x)$.
}
\end{figure}

\begin{table}[h]
  \centering\small
  \setlength{\tabcolsep}{5pt}
\begin{tabular}{cccccccc}
& & & & \multicolumn{2}{c}{\texttt{eig}} & \multicolumn{2}{c}{Algorithm~\ref{alg:qrshift}} \\
\cline{5-6} \cline{7-8}
$n$ & $\norm{c}$ & \norm{\texttt{bal}(C)} & $\max_i \abs{z_i}$ 
  & $n_\text{roots}$ & $\max_i \abs{\eta(p;\hat x_i)}$ 
  & $n_\text{roots}$ & $\max_i \abs{\eta(p;\hat x_i)}$ \\ 
\hline \T
1430 & $0.16\e{14}$  & $0.39\e{2}$  & $0.10\e{1}$ & 62 & $0.25\e{-13}$ & 62 & $0.98\e{-12}$ \\
\end{tabular}

\caption{
  \label{tab:cas}
The results of computing the roots of the Chebyshev expansion of the
function $f_\text{cas}(x)$, using our algorithm and \texttt{eig}, with
$\delta=10^{-4}$.  
}

\end{table}

\begin{figure}[h]
  \centering
  \includegraphics[width=0.58\textwidth]{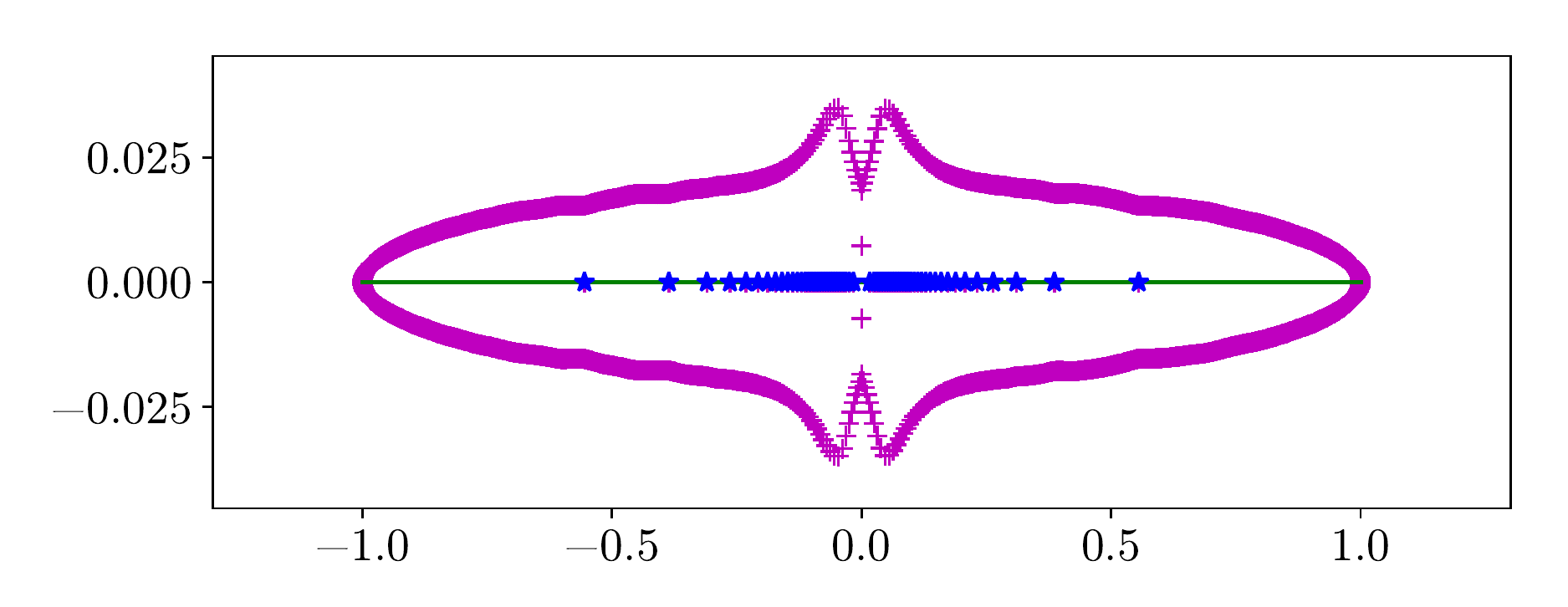}

\caption{ 
  \label{fig:casroots} 
The roots of the Chebyshev expansion of order $1600$ of the function
$f_\text{cas}(x)$. The complex roots $\hat z_i$ are plotted with purple
crosses ($+$) and the real roots $\hat x_i$ are plotted with blue stars
($\star$).
}
\end{figure}

\subsection{CPU Times}

The CPU times of our algorithm are compared to the times of MATLAB's
\texttt{eig} in Figure~\ref{fig:cputime}.  These timing experiments were
performed on polynomials with random, independent, normally distributed
Chebyshev expansion coefficients, with the last coefficient chosen so that
the vector $c$ has the desired norm (see Section~\ref{sec:prand}). We found
that the CPU times do not depend on $\norm{c}$, so we report the results
only for $\norm{c}=2$.  We observe that our algorithm is strictly faster
than \texttt{eig}, even for small inputs, except perhaps for $n=7$, for
which our algorithm and \texttt{eig} cost about the same. The growth in CPU
times taken by our algorithm agrees nicely with the expected asymptotic cost
of $O(n^2)$, while \texttt{eig} shows a growth of $O(n^3)$.

\begin{figure}[h]
  \centering
  \includegraphics[width=0.60\textwidth]{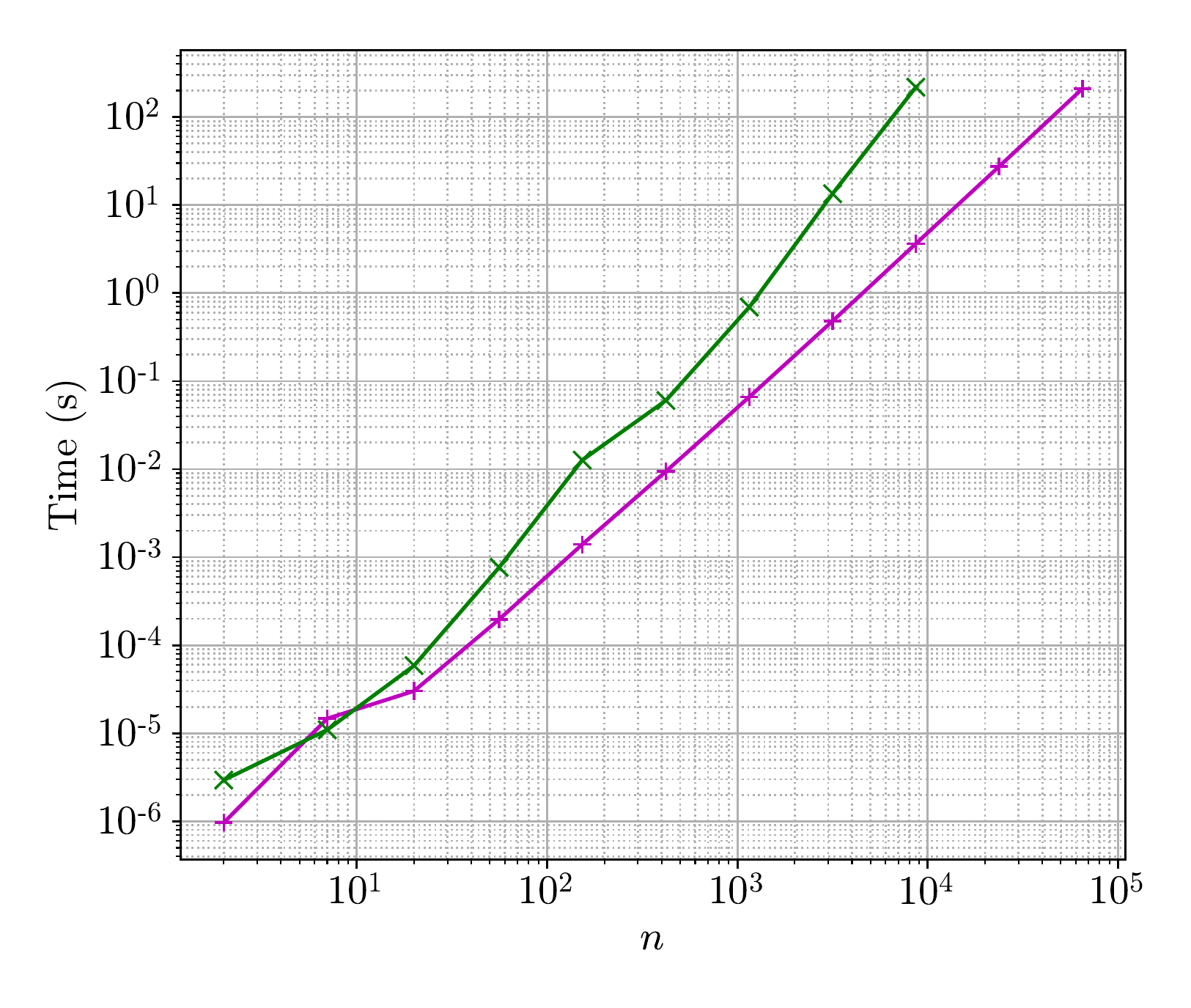}

\caption{ 
  \label{fig:cputime} 
The CPU times of our algorithm, plotted with purple crosses ($+$), and the CPU
times of \texttt{eig}, plotted with green x's ($\times$), for various values
of $n$, where $n$ is the dimensionality of the colleague matrix.
}
\end{figure}

\section{Conclusions and Generalizations}
  \label{sec:conc}

In this manuscript, we describe an explicit, $O(n^2)$ structured QR algorithm
for colleague matrices (more generally, for Hessenberg matrices that are the
sum of a Hermitian matrix and a rank-1 matrix), and prove that it is
componentwise backward stable.  These results can be generalized in several
directions, of which we describe four. First, the algorithm can be modified
in a fairly straightforward way to work on Hessenberg matrices that are the
sum of a Hermitian matrix and a rank-$k$ perturbation (as opposed to a
rank-1 perturbation). Like in the rank-1 case, most of the entries in the
Hermitian part are inferred from the low rank part, except that they are
inferred from a rank-$k$ matrix instead of a rank-1 matrix.  The QR
iteration proceeds similarly, the main difference being that the correction
in Line~\ref{alg:elim:corr} of Algorithm~\ref{alg:elim} becomes a correction
to a row of an $n \times k$ matrix.

Second, the extension of this algorithm to an implicit, $O(n^2)$ structured
QR algorithm that is also componentwise backward stable is fairly
straightforward.  The key observation of this manuscript (that, to maintain
componentwise error bounds, a correction must be applied to the rank-1 part
whenever an entry of the matrix is eliminated) can be applied to a
bulge-chasing algorithm where the matrix is similarly represented by
generators.

Third, we observe that this algorithm can be used to accelerate the
calculation of eigenvalues of general matrices, not necessarily in
Hessenberg form, that are representable as the sum of a Hermitian matrix and
a rank-1 (or rank-$k$) matrix. Such matrices can be quickly reduced to
Hessenberg form in $O(n^3)$ operations, and once they are in Hessenberg
form, our $O(n^2)$ algorithm can be used to compute the eigenvalues. Thus,
the cost of the algorithm is dominated by the reduction to Hessenberg form,
which will have a much smaller constant than the standard algorithm for the
evaluation of the eigenvalues of the original dense matrix.  Furthermore, if
the reduction to Hessenberg form can be done in a componentwise backward
stable fashion, then this scheme results in a componentwise backward stable
eigensolver for general matrices of the form Hermitian plus rank-1 (or
rank-$k$).

Fourth, we observe that our algorithm can be used to find the roots of
polynomials expressed in other bases besides Chebyshev polynomials.  It was
observed in~\cite{barnett} that, given any orthogonal polynomial basis that
satisfies a three-term recurrence relation, and given a polynomial expressed
in that basis, it is possible to construct an analogue of the colleague
matrix from the expansion coefficients. This matrix is a Hessenberg matrix
that is the sum of a (not necessarily symmetric) tridiagonal matrix and a
rank-1 matrix; matrices of this form are called \emph{comrade matrices}.
For all classical orthogonal polynomials, the tridiagonal part can made
symmetric by balancing, without making any entries of the matrix much larger
or much smaller. Our algorithm can then be applied to this new matrix, which
is a Hessenberg matrix that is the sum of a symmetric tridiagonal matrix and
a rank-1 matrix.

\end{document}